\documentclass[11pt]{article}
\usepackage[colorlinks=true,bookmarks=false,linkcolor=blue,urlcolor=blue,citecolor=blue,breaklinks=true]{hyperref}
\usepackage{breakcites}
\usepackage{authblk}

\usepackage{doi}

\usepackage{amsmath}
\usepackage{mathtools}

\usepackage{amsfonts}
\usepackage{amsthm}
\usepackage{amssymb}

\usepackage{color}
\usepackage{bm} % bold math fonts with \bm{math expr}
\usepackage{dsfont} % blackboard bold ones

\usepackage{graphicx}

\usepackage{tikz-cd}

\usepackage{tikz}
\definecolor{mycolor1}{rgb}{0.105882,0.619608,0.466667}
\definecolor{mycolor2}{rgb}{0.85098,0.372549,0.00784314}
\definecolor{mycolor3}{rgb}{0.458824,0.439216,0.701961}
\definecolor{mycolor4}{rgb}{0.905882,0.160784,0.541176}
\definecolor{mycolor5}{rgb}{0.4,0.65098,0.117647}
\definecolor{mycolor6}{rgb}{0.65098,0.462745,0.113725}
\definecolor{mycolor7}{rgb}{0.901961,0.670588,0.00784314}
\definecolor{mycolor8}{rgb}{0.4,0.4,0.4}
\definecolor{mycolor9}{rgb}{0.301961,0,0.294118}
\definecolor{mycolor10}{rgb}{0.0313725,0.25098,0.505882}

\usetikzlibrary{calc}        

\makeatletter
\newif\ifmygrid@coordinates
\tikzset{/mygrid/step line/.style={line width=0.80pt,draw=gray!80},
         /mygrid/steplet line/.style={line width=0.25pt,draw=gray!80}}
\pgfkeys{/mygrid/.cd,
         step/.store in=\mygrid@step,
         steplet/.store in=\mygrid@steplet,
         coordinates/.is if=mygrid@coordinates}
\def\mygrid@def@coordinates(#1,#2)(#3,#4){%
    \def\mygrid@xlo{#1}%
    \def\mygrid@xhi{#3}%
    \def\mygrid@ylo{#2}%
    \def\mygrid@yhi{#4}%
}
\newcommand\DrawGrid[3][]{%
    \pgfkeys{/mygrid/.cd,coordinates=true,step=1,steplet=0.2,#1}%
    \draw[/mygrid/steplet line] #2 grid[step=\mygrid@steplet] #3;
    \draw[/mygrid/step line] #2 grid[step=\mygrid@step] #3;
    \mygrid@def@coordinates#2#3%
    \ifmygrid@coordinates%
        \draw[/mygrid/step line]
        \foreach \xpos in {\mygrid@xlo,...,\mygrid@xhi} {%
          (\xpos,\mygrid@ylo) -- ++(0,-3pt)
                              node[anchor=north] {$\xpos$}
        }
        \foreach \ypos in {\mygrid@ylo,...,\mygrid@yhi} {%
          (\mygrid@xlo,\ypos) -- ++(-3pt,0)
                              node[anchor=east] {$\ypos$}
        };
    \fi%
}
\makeatother

\usepackage{psfrag}

\usepackage{algorithm}
\usepackage{algpseudocode}
\algdef{SE}[DOWHILE]{Do}{doWhile}{\algorithmicdo}[1]{\algorithmicwhile\ #1}%

\usepackage{mathrsfs} % for \mathscr{ABC} style (curly calligraphy)

\usepackage{bigfoot}
\newcommand{\add}[1]{{#1}}
\newcommand{\remove}[1]{}

\newcommand{\removesafe}[1]{}

\newcommand{\transpose}{^\top\! }

\newcommand{\inner}[2]{\left\langle{#1},{#2}\right\rangle}
\newcommand{\innersmall}[2]{\langle{#1},{#2}\rangle}

\newcommand{\trace}{\mathrm{Tr}}

\newcommand{\Sym}{\operatorname{Sym}}

\newcommand{\Proj}{\mathrm{Proj}}

\newcommand{\Exp}{\mathrm{Exp}}

\newcommand{\Retr}{\mathrm{R}}

\newcommand{\T}{\mathrm{T}}

\newcommand{\calP}{\mathcal{P}}

\newcommand{\Rd}{{\mathbb{R}^{d}}}

\newcommand{\Rdd}{{\mathbb{R}^{d\times d}}}

\newcommand{\reals}{{\mathbb{R}}}

\newcommand{\Rn}{{\mathbb{R}^n}}

\newcommand{\grad}{\mathrm{grad}}
\newcommand{\Hess}{\mathrm{Hess}}

\newcommand{\diag}{\mathrm{diag}}
\newcommand{\D}{\mathrm{D}}

\newcommand{\dphi}{\mathrm{d}\phi}
\newcommand{\Ddt}{\mathrm{D}_t}
\newcommand{\Ddttwo}{\mathrm{D}_t^2}
\newcommand{\Ddq}{\mathrm{D}_q}

\newcommand{\dtau}{\mathrm{d}\tau}
\newcommand{\dtheta}{\mathrm{d}\theta}

\newcommand{\calO}{\mathcal{O}}
\newcommand{\calM}{\mathcal{M}}
\newcommand{\calN}{\mathcal{N}}
\newcommand{\calS}{\mathcal{S}}

\newcommand{\Klow}{{K_{\mathrm{low}}}}
\newcommand{\Kup}{{K_{\mathrm{up}}}}
\newcommand{\Sn}{\mathbb{S}^{n-1}}

\newcommand{\norm}[1]{\left\|{#1}\right\|}

\newcommand{\TODOF}[1]{}

\newcommand{\lambdamin}{\lambda_\mathrm{min}}

\newcommand{\flow}{f_{\mathrm{low}}}

\newtheorem{theorem}{Theorem}[section]
\newtheorem{lemma}[theorem]{Lemma}
\newtheorem{proposition}[theorem]{Proposition}
\newtheorem{corollary}[theorem]{Corollary}
\newtheorem{assumption}{A\ignorespaces} %[section]
\newtheorem{definition}[theorem]{Definition} %[section]
\newtheorem{remark}[theorem]{Remark}

\usepackage[lmargin=3cm,rmargin=3cm,bottom=3cm,top=3cm]{geometry}
\usepackage[ansinew]{inputenc}
\usepackage[round,compress]{natbib}

\usepackage{sectsty}
\makeatletter\def\@seccntformat#1{\protect\makebox[0pt][r]{\csname the#1\endcsname\hspace{12pt}}}\makeatother

\newcommand{\sigmamin}{\sigma_{\operatorname{min}}}
\newcommand{\sigmamax}{\sigma_{\operatorname{max}}}
\usepackage{subcaption}

\usepackage{enumitem} % to enumerate with a. b. c. ... to match referee's numbering

\usepackage[normalem]{ulem} % just to strike out in a response to referee

\usepackage{multicol}

\newcommand{\aref}[1]{\hyperref[#1]{A\ref{#1}}}

\newcommand{\ARGD}{\mathtt{TAGD}}
\newcommand{\PARGD}{\mathtt{PTAGD}}
\newcommand{\TSS}{\mathtt{TSS}} % TangentSpaceSteps
\newcommand{\NCE}{\mathtt{NCE}} % NegativeCurvatureExploitation
\newcommand{\perturbed}{\mathtt{perturbed}}

\newcommand{\Xstuck}{\mathcal{X}^{(\mathrm{stuck})}}
\newcommand{\Prob}[1]{\operatorname{Prob}\!\left\{#1\right\}}
\newcommand{\Vol}[1]{\operatorname{Vol}\!\left(#1\right)}
\newcommand{\da}{\mathrm{d}a}
\newcommand{\dy}{\mathrm{d}y}
\newcommand{\dz}{\mathrm{d}z}
\newcommand{\one}{\mathbf{1}}

\title{An accelerated first-order method\\for non-convex optimization on manifolds} % \\ under/with Lipschitz conditions ?
\author{Christopher Criscitiello, Nicolas Boumal}
\affil{EPFL Institute of Mathematics \\ \{christopher.criscitiello, nicolas.boumal\}@epfl.ch}

\date{First posted on arXiv August 5, 2020; updated on November 25, 2021.} % August 5, 2020

\begin{document}

\maketitle

\begin{abstract}
	We describe the first gradient methods on Riemannian manifolds to achieve accelerated rates in the non-convex case.
	Under Lipschitz assumptions on the Riemannian gradient and Hessian of the cost function, these methods find approximate first-order critical points 		faster than regular gradient descent.
	A randomized version also finds approximate second-order critical points.
	Both the algorithms and their analyses build extensively on existing work in the Euclidean case.
	The basic operation consists in running the Euclidean accelerated gradient descent method (appropriately safe-guarded against non-convexity) in the current tangent space, then moving back to the manifold and repeating.
	This requires lifting the cost function from the manifold to the tangent space, which can be done for example through the Riemannian exponential map.
	For this approach to succeed, the lifted cost function (called the pullback) must retain certain Lipschitz properties.
	As a contribution of independent interest, we prove precise claims to that effect, with explicit constants.
	Those claims are affected by the Riemannian curvature of the manifold, which in turn affects the worst-case complexity bounds for our optimization algorithms.
\end{abstract}

\section{Introduction}

%\TODO{
%General flow of this section: \\
%$-$ Motivation \\
%$\rightarrow$ Euclidean results about accelerating GD in nonconvex setting \\
%$\rightarrow$ our results (extension of Euc results) \\
%$\rightarrow$ how we do this -- pullbacks and geometry results we need \\
%$\rightarrow$ more precise statement of results
%}

%\TODO{Comments for Chris, Nicolas can ignore:
%\begin{enumerate}

We consider optimization problems of the form
\begin{align}
	\min_{x \in \calM} f(x)
	\tag{P}
	\label{eq:P}
\end{align}
where $f$ is lower-bounded and twice continuously differentiable on a Riemannian manifold~$\calM$.
For the special case where $\calM$ is a Euclidean space, problem~\eqref{eq:P} amounts to smooth, unconstrained optimization.
The more general case is important for applications notably in scientific computing, statistics, imaging, learning, communications and robotics: see for example~\citep{AMS08,hu2020introoptimmanifolds}.

\add{
For a general non-convex objective $f$, computing a global minimizer of~\eqref{eq:P} is hard.  
Instead, our goal is to compute approximate first- and second-order critical points of~\eqref{eq:P}.  
A number of non-convex problems of interest exhibit the property that second-order critical points are optimal~\citep{boumal2016bmapproach,bandeira2016lowrankmaxcut,ge2016matrix,bhojanapalli2016global,mei2017solving,kawaguchi2016deep,zhang2020symmetry}.
Several of these are optimization problems on nonlinear manifolds.  
Therefore, theoretical guarantees for approximately finding second-order critical points can translate to guarantees for approximately solving these problems.  
%For example,~\citet[Sec.~4]{boumal2016globalrates} show how a guarantee for the Riemannian trust-region algorithm RTR translates to a guarantee for approximately solving Max-Cut.

It is therefore natural to ask for fast algorithms which find approximate second-order critical points on manifolds, within a tolerance $\epsilon$ (see below).
Existing algorithms include RTR~\citep{boumal2016globalrates}, ARC~\citep{agarwal2018arcfirst} and perturbed RGD~\citep{sun2019prgd,criscitiello2019escapingsaddles}.
Under some regularity conditions, ARC uses Hessian-vector products to achieve a rate of $O(\epsilon^{-7/4})$.
In contrast, under the same regularity conditions, perturbed RGD uses only function value and gradient queries, but achieves a poorer rate of $O(\epsilon^{-2})$.
%Both RTR and ARC use Hessian-vector products, and ARC achieves the better rate of $O(\epsilon^{-7/4})$.
%Perturbed RGD on the other hand only requires function value and gradient queries, but achieves a poorer rate of $O(\epsilon^{-2})$.
Does there exist an algorithm which finds approximate second-order critical points with a rate of $O(\epsilon^{-7/4})$ using only function value and gradient queries?
The answer was known to be yes in Euclidean space.
Can it also be done on Riemannian manifolds, hence extending applicability to applications treated in the aforementioned references?
%The only algorithm known to find such points in $O(\epsilon^{-7/4})$ queries is ARC \TODO{cite} -- however this algorithm requires Hessian-vector product queries.
%Our algorithm is the first algorithm using $O(\epsilon^{-7/4})$ gradient and function queries only.
We resolve that question positively with the algorithm $\PARGD$ below.
%The algorithm $\PARGD$ we develop does exactly this.
%For example, following~\citet[Sec.~4]{boumal2016globalrates}, we note that we can apply our algorithm to provide a theoretical guarantee for approximately solving the Max-Cut problem.

%There is a growing body of work showing that a number of problems of interest exhibit the property that second order critical points are optimal~\citep{boumal2016bmapproach,bandeira2016lowrankmaxcut,ge2016matrix,bhojanapalli2016global,mei2017solving,kawaguchi2016deep}.
%Several of these are optimization problems on nonlinear manifolds.
%It is therefore natural to ask for fast algorithms which find approximate second-order critical points on manifolds.
%%The only algorithm known to find such points in $O(\epsilon^{-7/4})$ queries is ARC \TODO{cite} -- however this algorithm requires Hessian-vector product queries.
%%Our algorithm is the first algorithm using $O(\epsilon^{-7/4})$ gradient and function queries only.
%In large-scale problems, computing Hessians of $f$ is often prohibitively expensive, so we shall further require that our algorithms are first-order, that is, they have access to $f$'s value and gradient but nothing else.  

From a different perspective, the recent success of momentum-based first-order methods in machine learning~\citep{ruder2017overview} has encouraged interest in momentum-based first-order algorithms for non-convex optimization which are provably faster than gradient descent~\citep{carmon2017convexguilty,jin2018agdescapes}.
We show such provable guarantees can be extended to optimization under a manifold constraint.
From this perspective, our paper is part of a body of work theoretically explaining the success of momentum methods in non-convex optimization.
%In addition, the recent success of momentum-based first-order methods in machine learning~\citep{ruder2017overview}, has encouraged interest in developing momentum-based first-order algorithms for non-convex optimization which are provably faster than gradient descent~\citep{carmon2017convexguilty,jin2018agdescapes}.
%We therefore ask: can such provable guarantees be extended to optimization under a manifold constraint?
%Our answer is yes.  

There has been significant difficulty in accelerating \emph{geodesically convex} optimization on Riemannian manifolds.
See ``Related literature'' below for more details on best known bounds \citep{ahn2020nesterovs} as well as results proving that acceleration in certain settings is impossible on manifolds~\citep{hamilton2021nogo}.
Given this difficulty, it is not at all clear a priori that it is possible to accelerate \emph{non-convex} optimization on Riemannian manifolds.
Our paper shows that it is in fact possible.

We design two new algorithms and establish worst-case complexity bounds under Lipschitz assumptions on the gradient and Hessian of $f$.
Beyond a theoretical contribution, we hope that this work will provide an impetus to look for more practical fast first-order algorithms on manifolds.
}

\remove{
Our goal is to compute approximate first- and second-order critical points of~\eqref{eq:P} using first-order methods, that is, algorithms which may query $f$'s value and gradient but nothing else.
To this end, we design two new algorithms.
We establish worst-case complexity bounds under Lipschitz assumptions on the derivatives of $f$. %, but not convexity.
}

More precisely, if the gradient of $f$ is $L$-Lipschitz continuous (in the Riemannian sense defined below), it is known that Riemannian gradient descent can find an $\epsilon$-approximate first-order critical point\footnote{That is, a point where the gradient of $f$ has norm smaller than $\epsilon$.} in at most $O(\Delta_f L / \epsilon^2)$ queries, where $\Delta_f$ upper-bounds the gap between initial and optimal cost value~\citep{zhang2016complexitygeodesicallyconvex,bento2017iterationcomplexity,boumal2016globalrates}.
Moreover, this rate is optimal in the special case where $\calM$ is a Euclidean space~\citep{carmon2017lower}, 
\add{but it can be improved under the additional assumption that the Hessian of $f$ is $\rho$-Lipschitz continuous.}

\remove{
It is natural to ask whether this rate can be improved if $f$ has additional properties.
Specifically, we assume the Hessian of $f$ is $\rho$-Lipschitz continuous (still in a Riemannian sense).
(Another interesting structure would be to assume geodesic convexity: we mention recent related efforts momentarily.)

For the Euclidean case, it has been shown under these assumptions that first-order methods require at least $\Omega(\Delta_f L^{3/7} \rho^{2/7} / \epsilon^{12/7})$ queries~\citep[Thm.~2]{carmon2019lowerboundsII}.
Recently, \citet{carmon2017convexguilty} have proposed an algorithm for this setting which requires at most $\tilde O(\Delta_f L^{1/2} \rho^{1/4} / \epsilon^{7/4})$ queries (up to logarithmic factors): this leaves a gap of merely $\tilde O(1/\epsilon^{1/28})$ in the $\epsilon$-dependency.
Their algorithm is deterministic and the complexity is independent of dimension.
Around the same time, \citet{jin2018agdescapes} showed how a related algorithm with randomization can find $(\epsilon, \sqrt{\rho\epsilon})$-approximate second-order critical points
%\footnote{That is, a point where the gradient of $f$ has norm smaller than $\epsilon$ and the eigenvalues of the Hessian of $f$ are at least $-\sqrt{\rho \epsilon}$.} 
with the same complexity, up to polylogarithmic factors in the dimension of the search space and in the (reciprocal of) the probability of failure.
}

\add{Recently in Euclidean space, \citet{carmon2017convexguilty} have proposed a deterministic algorithm for this setting ($L$-Lipschitz gradient, $\rho$-Lipschitz Hessian) which requires at most $\tilde O(\Delta_f L^{1/2} \rho^{1/4} / \epsilon^{7/4})$ queries (up to logarithmic factors), and is independent of dimension.
This is a speed up of Riemannian gradient descent by a factor of $\tilde \Theta(\sqrt{\frac{L}{\sqrt{\rho \epsilon}}})$.
For the Euclidean case, it has been shown under these assumptions that first-order methods require at least $\Omega(\Delta_f L^{3/7} \rho^{2/7} / \epsilon^{12/7})$ queries~\citep[Thm.~2]{carmon2019lowerboundsII}.  
This leaves a gap of merely $\tilde O(1/\epsilon^{1/28})$ in the $\epsilon$-dependency.

Soon after, \citet{jin2018agdescapes} showed how a related algorithm with randomization can find $(\epsilon, \sqrt{\rho\epsilon})$-approximate second-order critical points\footnote{That is, a point where the gradient of $f$ has norm smaller than $\epsilon$ and the eigenvalues of the Hessian of $f$ are at least $-\sqrt{\rho \epsilon}$.} with the same complexity, up to polylogarithmic factors in the dimension of the search space and in the (reciprocal of) the probability of failure.
}
 
Both the algorithm of \citet{carmon2017convexguilty} and that of \citet{jin2018agdescapes} fundamentally rely on Nesterov's accelerated gradient descent method (AGD)~\citep{nesterovagd1983}, with safe-guards against non-convexity.
To achieve improved rates, AGD builds heavily on a notion of momentum which accumulates across several iterations.  
This makes it delicate to extend AGD to nonlinear manifolds, as it would seem that we need to transfer momentum from tangent space to tangent space, all the while keeping track of fine properties.

\remove{
Presumably this could be done and indeed a number of recent papers show progress in analyzing Riemannian versions of AGD for geodesically convex problems~\citep{zhang2018estimatesequence,alimisis2019continuoustime,ahn2020nesterovs,alimisis2020practical}.
(With convexity, one can hope to improve the rate down to $\tilde O(1/\epsilon)$, as is the case in the Euclidean setting---see for example~\citep[Thm.~1]{carmon2019lowerboundsII} and the associated discussion tailored to the computation of $\epsilon$-approximate first-order critical points.) % there's a discussion there about why it's not 1/sqrt(eps) because we consider complexity with bounds on Delta_f, not on dist(x_0, optimum); also, there's a comment there that while the results are given under Lipschitz gradient, having Lipschitz Hessian wouldn't help, so it does fit our narrative here. These are details though.
Not assuming convexity, we follow a different approach.
}

In this paper, we build heavily on the Euclidean work of \citet{jin2018agdescapes} to show the following.
Assume $f$ has Lipschitz continuous gradient and Hessian on a complete Riemannian manifold satisfying some curvature conditions.
With at most $\tilde O(\Delta_f L^{1/2} \hat\rho^{1/4} / \epsilon^{7/4})$ queries (where $\hat\rho$ is larger than $\rho$ by an additive term affected by $L$ and the manifold's curvature),
\begin{enumerate}
	\item It is possible to compute an $\epsilon$-approximate first-order critical point of $f$ with a deterministic first-order method,
	\item It is possible to compute an $(\epsilon, \sqrt{\hat\rho \epsilon})$-approximate second-order critical point of $f$ with a randomized first-order method.
\end{enumerate}
In the first case, the complexity is independent of the dimension of $\calM$.
In the second case, the complexity includes polylogarithmic factors in the dimension of $\calM$ and in the probability of failure.
This parallels the Euclidean setting.
In both cases (and in contrast to the Euclidean setting), the Riemannian curvature of $\calM$ affects the complexity in two ways: (a) because $\hat\rho$ is larger than $\rho$, and (b) because the results only apply when the target accuracy $\epsilon$ is small enough in comparison to some curvature-dependent thresholds.
It is an interesting open question to determine whether such a curvature dependency is inescapable.

We call our first algorithm $\ARGD$ for \emph{tangent accelerated gradient descent},\footnote{We refrain from calling our first algorithm ``accelerated Riemannian gradient descent,'' thinking this name should be reserved for algorithms which emulate the momentum approach on the manifold directly.} and the second algorithm $\PARGD$ for \emph{perturbed tangent accelerated gradient descent}.
Both algorithms and (even more so) their analyses closely mirror the perturbed accelerated gradient descent algorithm (PAGD) of \citet{jin2018agdescapes}, with one core design choice that facilitates the extension to manifolds:
instead of transporting momentum from tangent space to tangent space, we run several iterations of AGD (safe-guarded against non-convexity) in individual tangent spaces. %until we reach a desired target.
After an AGD run in the current tangent space, we ``retract'' back to a new point on the manifold and initiate another AGD run in the new tangent space.
In so doing, we only need to understand the fine behavior of AGD in one tangent space at a time.
Since tangent spaces are linear spaces, we can capitalize on existing Euclidean analyses.
This general approach is in line with prior work in~\citep{criscitiello2019escapingsaddles}, \add{and is an instance of the dynamic trivializations framework of~\citet{lezcano2019trivializations}.}

In order to run AGD on the tangent space $\T_x\calM$ at $x$, we must ``pullback'' the cost function $f$ from $\calM$ to $\T_x\calM$.
A geometrically pleasing way to do so is via the exponential map\footnote{The exponential map is a \emph{retraction}: our main optimization results are stated for general retractions.} $\Exp_x \colon \T_x\calM \to \calM$, whose defining feature is that for each $v \in \T_x\calM$ the curve $\gamma(t) = \Exp_x(tv)$ is the geodesic of $\calM$ passing through $\gamma(0) = x$ with velocity $\gamma'(0) = v$.
Then, $\hat f_x = f \circ \Exp_x$ is a real function on $\T_x\calM$ called the \emph{pullback} of $f$ at $x$.
To analyze the behavior of AGD applied to $\hat f_x$, the most pressing question is:
\begin{quote}
	\emph{To what extent does $\hat f_x = f \circ \Exp_x$ inherit the Lipschitz properties of $f$?}
\end{quote}
In this paper, we show that if $f$ has Lipschitz continuous gradient and Hessian \emph{and if the gradient of $f$ at $x$ is sufficiently small}, then $\hat f_x$ \emph{restricted to a ball around the origin of $\T_x\calM$} has Lipschitz continuous gradient and retains \emph{partial Lipschitz-type properties for its Hessian}.
The norm condition on the gradient and the radius of the ball are dictated by the Riemannian curvature of $\calM$.
These geometric results are of independent interest.

Because $\hat f_x$ retains only partial Lipschitzness, our algorithms depart from the Euclidean case in the following ways: (a) at points where the gradient is still large, we perform a simple gradient step; and (b) when running AGD in $\T_x\calM$, we are careful not to leave the prescribed ball around the origin: if we ever do, we take appropriate action.
For those reasons and also because we do not have full Lipschitzness but only radial Lipschitzness for the Hessian of $\hat f_x$, minute changes throughout the analysis of \citet{jin2018agdescapes} are in order.

To be clear, in their current state, $\ARGD$ and $\PARGD$ are theoretical constructs.
As one can see from later sections, running them requires the user to know the value of several parameters that are seldom available (including the Lipschitz constants $L$ and $\rho$); the target accuracy $\epsilon$ must be set ahead of time; and the tuning constants as dictated here by the theory are (in all likelihood) overly cautious.
\add{To mitigate this, we show in Appendix~\ref{backtrackingTAGD} that a simple modification of $\ARGD$, called $\mathtt{backtrackingTAGD}$, finds $\epsilon$-approximate first-order critical points efficiently without knowledge of the Lipschitz constants $L$ or $\rho$.}

Moreover, to compute the gradient of $\hat f_x$ we need to differentiate through the exponential map (or a retraction, as the case may be).
This is sometimes easy to do in closed form (see~\citep{lezcano2019trivializations} for families of examples), but it could be a practical hurdle.
\add{On the other hand, our algorithms do not require parallel transport.}
\remove{Therefore, } It remains an interesting open question to develop \emph{practical} accelerated gradient methods for non-convex problems on manifolds.

In closing this introduction, we give simplified statements of our main results.
These are all phrased under the following assumption (see Section~\ref{sec:pullbacks} for geometric definitions):
\begin{assumption} \label{assu:Mandfintrinsic}
	The Riemannian manifold $\calM$ and the cost function $f \colon \calM \to \reals$ have these properties:
	\begin{itemize}
		\item $\calM$ is complete, its sectional curvatures are in the interval $[-K, K]$ and the covariant derivative of its Riemann curvature endomorphism is bounded by $F$ in operator norm; and
		\item $f$ is lower-bounded by $\flow$, has $L$-Lipschitz continuous gradient $\grad f$ and $\rho$-Lipschitz continuous Hessian $\Hess f$ on $\calM$.
	\end{itemize}
\end{assumption}

\subsection*{Main geometry results}

As a geometric contribution, we show that pullbacks through the exponential map retain certain Lipschitz properties of $f$.
Explicitly, at a point $x \in \calM$ we have the following statement.
\begin{theorem}
	Let $x \in \calM$.
	Under~\aref{assu:Mandfintrinsic}, let $B_x(b)$ be the closed ball of radius $b \leq \min\!\left( \frac{1}{4\sqrt{K}}, \frac{K}{4F} \right)$ around the origin in $\T_x\calM$.
%	If the norm of the gradient of $f$ at $x$ is at most $Lb$, then
	If $\|\grad f(x)\| \leq Lb$, then
	\begin{enumerate}
		\item The pullback $\hat f_x = f \circ \Exp_x$ has $2L$-Lipschitz continuous gradient $\nabla \hat f_x$ on $B_x(b)$, and
		\item For all $s \in B_x(b)$, we have $\|\nabla^2 \hat f_x(s) - \nabla^2 \hat f_x(0)\| \leq \hat\rho \|s\|$ with $\hat \rho = \rho + L\sqrt{K}$.
	\end{enumerate}
	(Above, $\|\cdot\|$ denotes both the Riemannian norm on $\T_x\calM$ and the associated operator norm.
	 Also, $\nabla \hat f_x$ and $\nabla^2 \hat f_x$ are the gradient and Hessian of $\hat f_x$ on the Euclidean space $\T_x\calM$.)
\end{theorem}
\noindent We expect this result to be useful in several other contexts.
Section~\ref{sec:pullbacks} provides a more complete (and somewhat more general) statement.
At the same time and independently, \citet{lezcano2020curvaturedependent} develops similar geometric bounds and applies them to study gradient descent in tangent spaces\add{---see ``Related literature'' below for additional details}.

\subsection*{Main optimization results}

We aim to compute approximate first- and second-order critical points of $f$, as defined here:
% with $\grad f$ and $\Hess f$ denoting the Riemannian gradient and Hessian of $f$:
\begin{definition}
	A point $x \in \calM$ is an \emph{$\epsilon$-FOCP} for~\eqref{eq:P} if $\|\grad f(x)\| \leq \epsilon$.
	A point $x \in \calM$ is an \emph{$(\epsilon_1, \epsilon_2)$-SOCP} for~\eqref{eq:P} if $\|\grad f(x)\| \leq \epsilon_1$ and $\lambdamin(\Hess f(x)) \geq -\epsilon_2$, where $\lambdamin(\cdot)$ extracts the smallest eigenvalue of a self-adjoint operator.
\end{definition}

In Section~\ref{sec:focp} we define and analyze the algorithm $\ARGD$.
Resting on the geometric result above, that algorithm with the exponential retraction warrants the following claim about the computation of first-order points.
The $O(\cdot)$ notation is with respect to scaling in $\epsilon$.

\begin{theorem} \label{thm:masterFOCPexp}
	If~\aref{assu:Mandfintrinsic} holds,
	there exists an algorithm \add{($\ARGD$)} which,
	given any $x_0 \in \calM$
	and small enough tolerance $\epsilon > 0$, namely, % (with $C$ some universal constant),
	\begin{align}
		\epsilon & \leq \frac{1}{144} \min\!\left(\frac{1}{K} \hat \rho, \frac{K^2}{F^2} \hat \rho, \frac{36 \ell^2}{\hat\rho}\right) = \frac{1}{144} \min\!\left(\frac{1}{K}, \frac{K^2}{F^2}, \left(\frac{12 L}{\rho + L\sqrt{K}}\right)^2 \right) (\rho + L\sqrt{K}),
		\label{eq:epsilonconditionsthm}
	\end{align}
	produces an $\epsilon$-FOCP for~\eqref{eq:P}
	using at most a constant multiple of $T$ function and pullback gradient queries, and a similar number of evaluations of the exponential map,
	where
	\begin{align*}
		T & = (f(x_0) - \flow) \frac{\hat{\rho}^{1/4} \ell^{1/2}}{\epsilon^{7/4}} \log\!\left(\frac{16\ell}{\sqrt{\hat \rho \epsilon}}\right)^{6} \\
%			+
%			\frac{\ell^{1/2}}{\hat \rho^{1/4} \epsilon^{1/4}} \log\!\left(\frac{16\ell}{\sqrt{\hat \rho \epsilon}}\right) ------- rmeoved the +1 in helper theorem on May 28, 2020
		  & = O\!\left( (f(x_0) - \flow) (\rho + L\sqrt{K})^{1/4} L^{1/2} \cdot \frac{1}{\epsilon^{7/4}} \log\!\left(\frac{1}{\epsilon}\right)^6 \right),
			% \chi \approx \frac{1}{2}\log_2\!\left(\frac{16\ell}{\sqrt{\hat \rho \epsilon}}\right) -- approx because needs to be rounded somewhat to ensure $\mathscr{T}$ is an integer multiple of 4, but that can be absorbed in the "a constant multiple of" I think: we're certainly not going to need to make $\chi$ more than twice bigger to meet that requirement.
	\end{align*}
	with $\ell = 2L$ and $\hat\rho = \rho + L\sqrt{K}$.
	The algorithm uses no Hessian queries and is deterministic.
\end{theorem}
\noindent This result is dimension free but not curvature free because $K$ and $F$ constrain $\epsilon$ and affect~$\hat\rho$.

\add{
\begin{remark}
In the statements of all theorems and lemmas, the notations $O(\cdot), \Theta(\cdot)$ only hide universal constants, i.e., numbers like $\frac{1}{2}$ or $100$.  They do not hide any parameters.
Moreover, $\tilde{O}(\cdot), \tilde{\Theta}(\cdot)$ only hide universal constants and logarithmic factors in the parameters.
\end{remark}}

\add{
\begin{remark}
If $\epsilon$ is large enough (that is, if $\epsilon > \Theta(\frac{\ell^2}{\hat\rho})$), then $\ARGD$ reduces to vanilla Riemannian gradient descent with constant step-size.
The latter is known to produce an $\epsilon$-FOCP in $O(1/\epsilon^2)$ iterations, yet our result here announces this same outcome in $O(1/\epsilon^{7/4})$ iterations.
This is not a contradiction: when $\epsilon$ is large, $1/\epsilon^{7/4}$ can be worse than $1/\epsilon^2$.
In short: the rates are only meaningful for small $\epsilon$, in which case $\ARGD$ does use accelerated gradient descent steps.
\end{remark}
}

In Section~\ref{sec:socp} we define and analyze the algorithm $\PARGD$.
With the exponential retraction, the latter warrants the following claim about the computation of second-order points.
\begin{theorem} \label{thm:masterSOCPexp}
	If~\aref{assu:Mandfintrinsic} holds,
	there exists an algorithm \add{($\PARGD$)} which,
	given any $x_0 \in \calM$,
	any $\delta \in (0, 1)$
	and small enough tolerance $\epsilon > 0$ (same condition as in Theorem~\ref{thm:masterFOCPexp})
	produces an $\epsilon$-FOCP for~\eqref{eq:P}
	using at most a constant multiple of $T$ function and pullback gradient queries, and a similar number of evaluations of the exponential map,
	where
	\begin{align*}
	T & = (f(x_0) - \flow) \frac{\hat{\rho}^{1/4} \ell^{1/2}}{\epsilon^{7/4}} \log\!\left( \frac{d^{1/2} \ell^{3/2} \Delta_f}{(\hat \rho \epsilon)^{1/4} \epsilon^2 \delta} \right)^{6}
		+
		\frac{\ell^{1/2}}{\hat \rho^{1/4} \epsilon^{1/4}} \log\!\left( \frac{d^{1/2} \ell^{3/2} \Delta_f}{(\hat \rho \epsilon)^{1/4} \epsilon^2 \delta} \right) \\
	& = O\!\left( (f(x_0) - \flow) (\rho + L\sqrt{K})^{1/4} L^{1/2} \cdot \frac{1}{\epsilon^{7/4}} \log\!\left(\frac{d}{\epsilon \delta}\right)^6 \right),
	\end{align*}
	with $\ell = 2L$, $\hat\rho = \rho + L\sqrt{K}$, $d = \dim \calM$ and any $\Delta_f \geq \max(f(x_0) - \flow, \sqrt{\epsilon^3 / \hat{\rho}})$.
	With probability at least $1 - 2\delta$, that point is also $(\epsilon, \sqrt{\hat\rho \epsilon})$-SOCP.
	The algorithm uses no Hessian queries and is randomized.
\end{theorem}
\noindent This result is \emph{almost} dimension free, and still not curvature free for the same reasons as above.

\add{
\subsection*{Related literature} \label{relatedwork}
At the same time and independently, \citet{lezcano2020curvaturedependent} develops geometric bounds similar to our own.
Both papers derive the same second-order inhomogenous linear ODE (ordinary differential equation) describing the behavior of the second derivative of the exponential map.
\citet{lezcano2020curvaturedependent} then uses ODE comparison techniques to derive the geometric bounds, while the present work uses a bootstrapping technique.
%which yield comparatively tighter bounds albeit requiring a little bit more analysis.
\citet{lezcano2020curvaturedependent} applies these bounds to study gradient descent in tangent spaces, whereas we study non-convex accelerated algorithms for finding first- and second-order critical points.

The technique of pulling back a function to a tangent space is frequently used in other settings within optimization on manifolds.
See for example the recent papers of~\citet{bergmann2020fenchel} and~\citet{lezcanocasado2020adaptive}.
Additionally, the use of Riemannian Lipschitz conditions in optimization as they appear in Section 2 can be traced back to~\cite[Def.~4.1]{da1998geodesic} and~\citep[Def.~2.2]{ferreira2002kantorovichnewton}.

Accelerating optimization algorithms on Riemannian manifolds has been well-studied in the context of \textit{geodesically convex} optimization problems.  
Such problems can be solved globally, and usually the objective is to bound the suboptimality gap rather than finding approximate critical points.
A number of papers have studied Riemannian versions of AGD; however, none of these papers have been able to achieve a truly accelerated rate for convex optimization.
%Most notably are the papers of \citet{zhang2018estimatesequence} and \citet{ahn2020nesterovs}.
\citet{zhang2018estimatesequence} show that if the initial iterate is sufficiently close to the minimizer, then acceleration is possible.
Intuitively this makes sense, since manifolds are locally Euclidean.
\citet{ahn2020nesterovs} pushed this further, developing an algorithm converging strictly faster than RGD, and which also achieves acceleration when sufficiently close to the minimizer.

\citet{alimisis2019continuoustime,alimisis2020practical,alimisis2021momentum} analyze the problem of acceleration on the class of nonstrongly convex functions, as well as under weaker notions of convexity.
Interestingly, they also show that in the continuous limit (using an ODE to model optimization algorithms) acceleration is possible.  
However, it is unclear whether the discretization of this ODE preserves a similar acceleration.

Recently,~\citet{hamilton2021nogo} have shown that true acceleration (in the geodesically convex case) is impossible in the hyperbolic plane, in the setting where function values and gradients are corrupted by a very small amount of noise.
In contrast, in the analogous Euclidean setting, acceleration is possible even with noisy oracles~\citep{RePEc:cor:louvco:2013016}.
%Given the difficulty of accelerating in geodesically convex optimization, it is not a priori obvious that acceleration is possible in the setting we consider.
}

\section{Riemannian tools and regularity of pullbacks} \label{sec:pullbacks}

In this section, we build up to and state our main geometric result.
As we do so, we provide a few reminders of Riemannian geometry.
For more on this topic, we recommend the modern textbooks by~\citet{lee2012smoothmanifolds,lee2018riemannian}.
For book-length, optimization-focused introductions see~\citep{AMS08,boumal2020intromanifolds}.
Some proofs of this section appear in Appendices~\ref{app:PTvDExp} and~\ref{app:controllingcprimeprime}.

We consider a manifold $\calM$ with Riemannian metric $\inner{\cdot}{\cdot}_x$ and associated norm $\|\cdot\|_x$ on the tangent spaces $\T_x\calM$.
(In other sections, we omit the subscript $x$.)
The \emph{tangent bundle}
\begin{align*}
	\T\calM = \{ (x, s) : x \in \calM \textrm{ and } s \in \T_x\calM \}
\end{align*}
is itself a smooth manifold.
The Riemannian metric provides a notion of gradient.
\begin{definition}
	The \emph{Riemannian gradient} of a differentiable function $f \colon \calM \to \reals$ is the unique vector field $\grad f$ on $\calM$ which satisfies:
	\begin{align*}
		\D f(x)[s] & = \inner{\grad f(x)}{s}_x & \textrm{ for all } (x, s) \in \T\calM,
	\end{align*}
	where $\D f(x)[s]$ is the directional derivative of $f$ at $x$ along $s$.
\end{definition}
The Riemannian metric further induces a uniquely defined \emph{Riemannian connection} $\nabla$ (used to differentiate vector fields on $\calM$) and an associated \emph{covariant derivative} $\Ddt$ (used to differentiate vector fields along curves on $\calM$).
(The symbol $\nabla$ here is not to be confused with its use elsewhere to denote differentiation of scalar functions on Euclidean spaces.)
Applying the connection to the gradient vector field, we obtain Hessians.
\begin{definition}
	The \emph{Riemannian Hessian} of a twice differentiable function $f \colon \calM \to \reals$ at $x$ is the linear operator $\Hess f(x)$ to and from $\T_x\calM$ defined by
	\begin{align*}
		\Hess f(x)[s] & = \nabla_s \grad f = \left. \Ddt \grad f(c(t)) \right|_{t = 0},
	\end{align*}
	where in the last equality $c$ can be any smooth curve on $\calM$ satisfying $c(0) = x$ and $c'(0) = s$.
	This operator is self-adjoint with respect to the metric $\inner{\cdot}{\cdot}_x$.
\end{definition}

\add{
We can also define the Riemmannian third derivative $\nabla^3 f$ (a tensor of order three), see~\citep[Ch. 10]{boumal2020intromanifolds} for details.
We write $\norm{\nabla^3 f(x)} \leq \rho$ to mean $\left|\nabla^3 f(x)(u, v, w)\right| \leq \rho$ for all unit vectors $u, v, w \in \T_x\calM$.
}

A \emph{retraction} $\Retr$ is a smooth map from (a subset of) $\T\calM$ to $\calM$ with the following property:
for all $(x, s) \in \T\calM$, the smooth curve $c(t) = \Retr(x, ts) = \Retr_x(ts)$ on $\calM$ passes through $c(0) = x$ with velocity $c'(0) = s$.
Such maps are used frequently in Riemannian optimization in order to move on a manifold.
For example, a key ingredient of Riemannian gradient descent is the curve $c(t) = \Retr_x(-t \grad f(x))$ which initially moves away from $x$ along the negative gradient direction.

To a curve $c$ we naturally associate a velocity vector field $c'$.
Using the covariant derivative $\Ddt$, we differentiate this vector field along $c$ to define the \emph{acceleration} $c'' = \Ddt c'$ of $c$: this is also a vector field along $c$.
In particular, the \emph{geodesics} of $\calM$ are the curves with zero acceleration.

The \emph{exponential map} $\Exp \colon \calO \to \calM$---defined on an open subset $\calO$ of the tangent bundle---is a special retraction whose curves are geodesics.
Specifically, $\gamma(t) = \Exp(x, ts) = \Exp_x(ts)$ is the unique geodesic on $\calM$ which passes through $\gamma(0) = x$ with velocity $\gamma'(0) = s$.
%A manifold is \emph{complete} if the domain of $\Exp$ is the whole tangent bundle, meaning all geodesics are defined on the whole real line.
If the domain of $\Exp$ is the whole tangent bundle, we say $\calM$ is \emph{complete}.

To compare tangent vectors in distinct tangent spaces, we use \emph{parallel transports}.
Explicitly, let $c$ be a smooth curve connecting the points $c(0) = x$ and $c(1) = y$.
We say a vector field $Z$ along $c$ is \emph{parallel} if its covariant derivative $\Ddt Z$ is zero.
%Then, we may think of $Z(0) \in \T_x\calM$ and $Z(1) \in \T_y\calM$ as ``comparable'' (in some sense).
Conveniently, for any given $v \in \T_x\calM$ there exists a unique parallel vector field along $c$ whose value at $t = 0$ is $v$.
Therefore, the value of that vector field at $t = 1$ is a well-defined vector in $\T_y\calM$: we call it the parallel transport of $v$ from $x$ to $y$ along $c$.
We introduce the notation
\begin{align*}
	P_t^c \colon \T_{c(0)}\calM \to \T_{c(t)}\calM
\end{align*}
to denote parallel transport along a smooth curve $c$ from $c(0)$ to $c(t)$.
This is a linear isometry: $(P_t^c)^{-1} = (P_t^c)^*$, where the star denotes an adjoint with respect to the Riemannian metric.
For the special case of parallel transport along the geodesic $\gamma(t) = \Exp_x(ts)$, we write
\begin{align}
	P_{ts} \colon \T_x\calM \to \T_{\Exp_x(ts)}\calM
	\label{eq:Pts}
\end{align}
with the meaning $P_{ts} = P_t^\gamma$.

Using these tools, we can define Lipschitz continuity of gradients and Hessians.
Note that in the particular case where $\calM$ is a Euclidean space we have $\Exp_x(s) = x + s$ and parallel transports are identities, so that this reduces to the usual definitions.
\begin{definition}
	The gradient of $f \colon \calM \to \reals$ is \emph{$L$-Lipschitz continuous} if
	\begin{align}
		\|P_s^* \grad f(\Exp_x(s)) - \grad f(x)\|_x & \leq L \|s\|_x & \textrm{ for all } (x, s) \in \calO,
		\label{eq:LipschitzGrad}
	\end{align}
	\add{where $P_s^*$ is the adjoint of $P_s$ with respect to the Riemannian metric.}
	
	The Hessian of $f$ is \emph{$\rho$-Lipschitz continuous} if
	\begin{align}
		\|P_s^* \circ \Hess f(\Exp_x(s)) \circ P_s - \Hess f(x)\|_x & \leq \rho \|s\|_x & \textrm{ for all } (x, s) \in \calO,
		\label{eq:LipschitzHess}
	\end{align}
	where $\|\cdot\|_x$ denotes both the Riemannian norm on $\T_x\calM$ and the associated operator norm.
\end{definition}
It is well known that these Lipschitz conditions are equivalent to convenient inequalities, often used to study the complexity of optimization algorithms.
More details appear in~\citep[Ch.~10]{boumal2020intromanifolds}.
\remove{
Variations on this theme occur early, for example in~\cite[Def.~4.1]{da1998geodesic} and~\citep[Def.~2.2]{ferreira2002kantorovichnewton}.
}
\begin{proposition} \label{prop:lipschitzinequalitiesusual}
	If a function $f \colon \calM \to \reals$ has $L$-Lipschitz continuous gradient, then
	\begin{align*}
		\left| f(\Exp_x(s)) - f(x) - \inner{\grad f(x)}{s}_x \right| & \leq \frac{L}{2} \|s\|_x^2 & \textrm{ for all } (x, s) \in \calO.
	\end{align*}
	\add{If in addition $f$ is twice differentiable, then $\norm{\Hess f(x)} \leq L$ for all $x \in \calM$.}
	
	If $f$ has $\rho$-Lipschitz continuous Hessian, then
	\begin{align*}
		\left| f(\Exp_x(s)) - f(x) - \inner{\grad f(x)}{s}_x - \frac{1}{2} \inner{s}{\Hess f(x)[s]}_x \right| & \leq \frac{\rho}{6} \|s\|_x^3 && \textrm{ and} & \\
		\left\| P_{s}^* \grad f(\Exp_x(s)) - \grad f(x) - \Hess f(x)[s] \right\|_x & \leq \frac{\rho}{2} \|s\|_x^2 && \textrm{ for all } (x, s) \in \calO. & 	
	\end{align*}
	\add{If in addition $f$ is three times differentiable, then $\norm{\nabla^3 f(x)} \leq \rho$ for all $x \in \calM$.}
	
	The other way around, if $f$ is three times continuously differentiable and the stated inequalities hold, then its gradient and Hessian are Lipschitz continuous with the stated constants.
\end{proposition}
For sufficiently simple algorithms, these inequalities may be all we need to track progress in a sharp way.
As an example, the iterates of Riemannian gradient descent with constant step-size $1/L$ satisfy $x_{k+1} = \Exp_{x_k}(s_k)$ with $s_k = -\frac{1}{L}\grad f(x_k)$.
It follows directly from the first inequality above that $f(x_k) - f(x_{k+1}) \geq \frac{1}{2L} \|\grad f(x_k)\|^2$.
From there, it takes a brief argument to conclude that this method finds a point with gradient smaller than $\epsilon$ in at most $2L(f(x_0) - \flow)\frac{1}{\epsilon^2}$ steps.
A similar (but longer) story applies to the analysis of Riemannian trust regions and adaptive cubic regularization~\citep{boumal2016globalrates,agarwal2018arcfirst}.

However, the inequalities in Proposition~\ref{prop:lipschitzinequalitiesusual} fall short when finer properties of the algorithms are only visible at the scale of multiple combined iterations.
This is notably the case for accelerated gradient methods.
For such algorithms, individual iterations may not achieve spectacular cost decrease, but a long sequence of them may accumulate an advantage over time (using momentum).
To capture this advantage in an analysis, it is not enough to apply inequalities above to individual iterations.
As we turn to assessing a string of iterations jointly by relating the various gradients and step directions we encounter, the nonlinearity of $\calM$ generates significant hurdles.

For these reasons, we study the pullbacks of the cost function, namely, the functions
\begin{align}
	\hat f_x  & = f \circ \Exp_x \colon \T_x\calM \to \reals.
\end{align}
\add{Each pullback is defined on a linear space, hence we can in principle run any Euclidean optimization algorithm on $\hat f_x$ directly: our strategy is therefore to apply a momentum-based method on $\hat{f}_x$.
To this end, we now work towards showing that if $f$ has Lipschitz continuous gradient and Hessian then $\hat f_x$ also has certain Lipschitz-type properties.}
\remove{Each pullback is defined on a linear space, hence we can in principle run any Euclidean optimization algorithm on $\hat f_x$ directly.
In what follows, we establish that if $f$ has Lipschitz continuous gradient and Hessian, then $\hat f_x$ also has certain Lipschitz-type properties.}
%In subsequent sections, we show that these properties are sufficient to generalize Euclidean algorithms and analyses with limited friction.

The following formulas appear in~\citep[Lem.~5]{agarwal2018arcfirst}: we are interested in the case $\Retr = \Exp$.
% for general retractions---we particularize it to the exponential map for easier reading.
(We use $\nabla$ and $\nabla^2$ to designate gradients and Hessians of functions on Euclidean spaces: not to be confused with the connection $\nabla$.)
\begin{lemma} \label{lem:derivativespullback}
	Given $f \colon \calM \to \reals$ twice continuously differentiable and $(x, s)$ in the domain of a retraction $\Retr$, the gradient and Hessian of the pullback $\hat f_x = f \circ \Retr_x$ at $s \in \T_x\calM$ are given by
	\begin{align}
		\nabla \hat f_x(s) & = T_s^* \grad f(\Retr_x(s)) & \textrm{ and } &&  % \label{eq:gradientpullback} \\
		\nabla^2 \hat f_x(s) & = T_s^* \circ \Hess f(\Retr_x(s)) \circ T_s + W_s, %\label{eq:hessianpullback}
		\label{eq:gradhesspullback}
	\end{align}
	where $T_s$ is the differential of $\Retr_x$ at $s$ (a linear operator):
	\begin{align}
		T_s & = \D\Retr_x(s) \colon \T_{x}\calM \to \T_{\Retr_x(s)}\calM,
		\label{eq:Ts}
	\end{align}
	and $W_s$ is a self-adjoint linear operator on $\T_x\calM$ defined through polarization by
	\begin{align}
		\inner{W_s[\dot s]}{\dot s}_x & = \inner{\grad f(\Retr_x(s))}{c''(0)}_{\Retr_x(s)},
		\label{eq:Whessianpullback}
	\end{align}
	with $c''(0) \in \T_{\Retr_x(s)}\calM$ the (intrinsic) acceleration on $\calM$ of $c(t) = \Retr_x(s+t\dot s)$ at $t = 0$.
\end{lemma}

\add{
\begin{remark}
Throughout, $s, \dot s, \ddot s$ will simply denote tangent vectors.
%The dot notation used here, e.g., $\dot s$, does \textit{not} refer to any type of derivative.
%This will be clear from context.
\end{remark}
}

We turn to curvature.
The \emph{Lie bracket} of two smooth vector fields $X, Y$ on $\calM$ is itself a smooth vector field, conveniently expressed in terms of the Riemannian connection as $[X, Y] = \nabla_X Y - \nabla_Y X$.
Using this notion, the \emph{Riemann curvature endomorphism} $R$ of $\calM$ is an operator which maps three smooth vector fields $X, Y, Z$ of $\calM$ to a fourth smooth vector field as:
\begin{align}
	R(X, Y) Z & = \nabla_X \nabla _Y Z - \nabla_Y \nabla _X Z - \nabla_{[X, Y]} Z.
\end{align}
Whenever $R$ is identically zero, we say $\calM$ is \emph{flat}: % Lee Thm 7.10
this is the case notably when $\calM$ is a Euclidean space and when $\calM$ has dimension one (e.g., a circle is flat, while a sphere is not).

Though it is not obvious from the definition, the value of the vector field $R(X, Y) Z$ at $x \in \calM$ depends on $X, Y, Z$ only through their value at $x$.
Therefore, given $u, v, w \in \T_x\calM$ we can make sense of the notation $R(u, v)w$ as denoting the vector in $\T_x\calM$ corresponding to $R(X, Y)Z$ at $x$, where $X, Y, Z$ are arbitrary smooth vector fields whose values at $x$ are $u, v, w$, respectively.
The map $(u, v, w) \mapsto R(u, v)w$ is linear in each input.

Two linearly independent tangent vectors $u, v$ at $x$ span a two-dimensional plane of $\T_x\calM$.
The \emph{sectional curvature} of $\calM$ along that plane is a real number $K(u, v)$ defined as
\begin{align}
	K(u, v) & = \frac{\inner{R(u, v)v}{u}_x}{\|u\|_x^2\|v\|_x^2 - \inner{u}{v}_x^2}.
	\label{eq:sectionalcurvature}
\end{align}
Of course, all sectional curvatures of flat manifolds are zero.
Also, all sectional curvatures of a sphere of radius $r$ are $1/r^2$ and all sectional curvatures of the hyperbolic space with parameter $r$ are $-1/r^2$---see~\citep[Thm.~8.34]{lee2018riemannian}.
% Sometimes, manifolds whose sectional curvatures all are in the interval $[-K, K]$ for some $K \geq 0$ are called \emph{$K$-pinched}. % Gromov p110 uses it that way, but then in this paper (which is the top result for the term) they define it very differently on page 2: https://www.ams.org/journals/jams/2009-22-01/S0894-0347-08-00613-9/S0894-0347-08-00613-9.pdf -- just don't.

Using the connection $\nabla$, we differentiate the curvature endomorphism $R$ as follows.
Given any smooth vector field $U$, we let $\nabla_U R$ be an operator of the same type as $R$ itself, in the sense that it maps three smooth vector fields $X, Y, Z$ to a fourth one denoted $(\nabla_U R)(X, Y)Z$ through
\begin{align}
	(\nabla_U R)(X, Y)Z & = \nabla_U(R(X, Y) Z) - R(\nabla_U X, Y) Z - R(X, \nabla_U Y) Z - R(X, Y) \nabla_U Z.
	\label{eq:nablaUR}
\end{align}
Observe that this formula captures a convenient chain rule on $\nabla_U(R(X, Y) Z)$.
As for $R$, the value of $\nabla R(X, Y, Z, U) \triangleq (\nabla_U R)(X, Y)Z$ at $x$ depends on $X, Y, Z, U$ only through their values at $x$.
Therefore, $\nabla R$ unambiguously maps $u, v, w, z \in \T_x\calM$ to $\nabla R(u, v, w, z) \in \T_x\calM$, linearly in all inputs.
We say the operator norm of $\nabla R$ at $x$ is bounded by $F$ if
\begin{align*}
	\|\nabla R(u, v, w, z)\|_x & \leq F \|u\|_x \|v\|_x \|w\|_x \|z\|_x
\end{align*}
for all $u, v, w, z \in \T_x\calM$.
We say $\nabla R$ has operator norm bounded by $F$ if this holds for all $x$. %  \in \calM
If $F = 0$ (that is, $\nabla R \equiv 0$), we say $R$ is \emph{parallel} and $\calM$ is called \emph{locally symmetric}.
This is notably the case for manifolds with constant sectional curvature---Euclidean spaces, spheres and hyperbolic spaces---and (Riemannian) products thereof~\citep[pp219--221]{oneill}.
% The part about products is not fully obvious because products of non-flat manifolds with constant sectional curvature do not have constant sectional curvature, simply because along planes that span two elements of the product sectional curvature will be zero.

We are ready to state the main result of this section.
Note that $\calM$ need not be complete.
\begin{theorem} \label{thm:pullbacklipschitz}
	Let $\calM$ be a Riemannian manifold whose sectional curvatures are in the interval $[\Klow, \Kup]$, and let $K = \max(|\Klow|, |\Kup|)$.
	Also assume $\nabla R$---the covariant derivative of the Riemann curvature endomorphism $R$---is bounded by $F$ in operator norm.
	%	\TODO{These conditions on curvature only need to hold along $\gamma$ I imagine? Would need to check the comparison theorems we called upon in the previous theorem too.}
	Let $f \colon \calM \to \reals$ be twice continuously differentiable and select $b > 0$ such that
	\begin{align*}
		b \leq \min\!\left( \frac{1}{4\sqrt{K}}, \frac{K}{4F} \right).
	\end{align*}
	Pick any point $x \in \calM$ such that $\Exp_x$ is defined on the closed ball $B_x(b)$ of radius $b$ around the origin in $\T_x\calM$.
	We have the following three conclusions:
	\begin{enumerate}
		\item If $f$ has $L$-Lipschitz continuous gradient and $\|\grad f(x)\|_x \leq Lb$, then $\hat f_x = f \circ \Exp_x$ has $2L$-Lipschitz continuous gradient in $B_x(b)$, that is, for all $u, v \in B_x(b)$ it holds that $\|\nabla \hat f_x(u) - \nabla \hat f_x(v)\|_x \leq 2L \|u - v\|_x$.
		\item If moreover $f$ has $\rho$-Lipschitz continuous Hessian, then $\|\nabla^2 \hat f_x(s) - \nabla^2 \hat f_x(0)\|_x \leq \hat\rho \|s\|_x$ for all $s \in B_x(b)$, with $\hat \rho = \rho + L \sqrt{K}$.
		\item For all $s \in B_x(b)$, the singular values of $T_s = \D\Exp_x(s)$ lie in the interval $[2/3, 4/3]$.
	\end{enumerate}
\end{theorem}
\noindent A few comments are in order:
\begin{enumerate}
	\item For locally symmetric spaces ($F = 0$), we interpret $K/F$ as infinite (regardless of $K$).
	\item If $\calM$ is compact, then it is complete and there necessarily exist finite $K$ and $F$.
	      See work by~\citet{greene1978boundedcurvature} for a discussion on non-compact manifolds.
	\add{
	\item If $\calM$ is a homogeneous Riemannian manifold (not necessarily compact), then there exist finite $K$ and $F$, and these can be assessed by studying a single point on the manifold.
	This follows directly from the definition of homogeneous Riemannian manifold~\citep[p55]{lee2018riemannian}.
%	Indeed, fix any $y \in \calM$, let $K$ denote the maximum sectional curvature at $y$, and $F = \norm{\nabla R}$.
%	Then for any $x \in \calM$ there is an isometry $\phi \colon \calM \rightarrow \calM$ with $\phi(x) = y$ (definition of homogeneous).
%	Isometries preserve the curvature tensor, i.e., $\inner{R(U, V) W}{X} = \inner{R(\D \phi(x) U, \D \phi(x) V) \D \phi(x) W}{\D\phi(x) X}$ for all $U, V, W, X \in T_x \calM$, which allows us to conclude.
	
	\item All symmetric spaces are homogeneous and locally symmetric~\citep[Exercise~6-19, Exercise~7-3 and p78]{lee2018riemannian} so there exists finite $K$ and $F = 0$.  
	Let $\Sym(d)$ be the set of real $d \times d$ symmetric matrices.
	The set of $d\times d$ positive definite matrices
	$$\calP_d = \{P \in \Sym(d) : P \succ 0\}$$ 
	endowed with the so-called affine invariant metric 
$$\inner{X}{Y}_P = \trace(P^{-1} X P^{-1} Y) \quad \text{for } P \in \calP_d \text{ and } X, Y \in \T_P \calP_d \cong \Sym(d)$$
is a noncompact symmetric space of nonconstant curvature.
It is commonly used in practice~\citep{Bhatia,sra2015conicgeometricoptimspd,moakher2005diffgeomspdmeans,moakher2006symmetric}.	
%	$\calP_d$ has nonpositive sectional curvature~\citep[Prop.~3.1]{dolcetti2018differential}.
	In Appendix~\ref{appPDmatrices}, we show that $K = \frac{1}{2}$ and $F=0$ are the right constants for this manifold.
	}
	\item The following statements are equivalent: (a) $\calM$ is complete; (b) $\Exp$ is defined on the whole tangent bundle: $\calO = \T\calM$; and (c) for some $b > 0$, $\Exp_x$ is defined on $B_x(b)$ for all $x \in \calM$. In later sections, we need to apply Theorem~\ref{thm:pullbacklipschitz} at various points of $\calM$ with constant $b$, which is why we then assume $\calM$ is complete.
	\item The properties of $T_s$ are useful in combination with Lemma~\ref{lem:derivativespullback} to relate gradients and Hessians of the pullbacks to gradients and Hessians on the manifold. For example, if $\nabla \hat f_x(s)$ has norm $\epsilon$, then $\grad f(\Exp_x(s))$ has norm somewhere between $\frac{3}{4}\epsilon$ and $\frac{3}{2}\epsilon$. Under the conditions of the theorem, $W_s$~\eqref{eq:Whessianpullback} is bounded as $\|W_s\|_x \leq \frac{9}{4} K \|\nabla \hat f_x(s)\|_x \|s\|_x$.
	\item We only get satisfactory Lipschitzness at points where the gradient is bounded by $Lb$.
	Fortunately, for the algorithms we study, whenever we encounter a point with gradient larger than that threshold it is sufficient to take a simple gradient descent step.
	%As a result, this is not an obstacle for analysis.
\end{enumerate}
See Section~\ref{sec:discussion} for additional comments regarding the restriction to balls of radius $b$ and regarding the only-partial Lipschitzness of the Hessian: those are the two main sources of technicalities in adapting Euclidean analyses to the Riemannian case in subsequent sections.

To prove Theorem~\ref{thm:pullbacklipschitz}, we must control $\nabla^2 \hat f_x(s)$.
According to Lemma~\ref{lem:derivativespullback}, this requires controlling both $T_s$ (a differential of the exponential map) and $c''(0)$ (the intrinsic initial acceleration of a curve defined via the exponential map, but which is not itself a geodesic in general). %(unless for example the manifold is flat, $\dot s$ is aligned with $s$, or $s = 0$).
On both counts, we must study differentials of exponentials.
\emph{Jacobi fields} are the tool of choice for such tasks.
As a first step, we use Jacobi fields to investigate the difference between $T_s$ and $P_s$: two linear operators from $\T_x\calM$ to $\T_{\Exp_x(s)}\calM$.
We prove a general result in Appendix~\ref{app:PTvDExp} (exact for constant sectional curvature) and state a sufficient particular case here.
Control of $T_s$ follows as a corollary because $P_s$ (parallel transport) is an isometry.
\begin{proposition} \label{prop:TsminPsparticular}
	Let $\calM$ be a Riemannian manifold whose sectional curvatures are in the interval $[\Klow, \Kup]$, and let $K = \max(|\Klow|, |\Kup|)$. For any $(x, s) \in \calO$ with $\|s\|_x \leq \frac{\pi}{\sqrt{K}}$,
	\begin{align}
		\|(T_s - P_s)[\dot s]\|_{\Exp_x(s)} & \leq \frac{1}{3} K \|s\|_x^2 \|\dot s_\perp\|_x,
		\label{eq:PminTboundgeneral}
	\end{align}
	where $\dot s_\perp = \dot s - \frac{\inner{s}{\dot s}_x}{\inner{s}{s}_x}s$ is the component of $\dot s$ orthogonal to $s$.
\end{proposition}
\begin{corollary} \label{cor:Tssingularvalues}
	Let $\calM$ be a Riemannian manifold whose sectional curvatures are in the interval $[\Klow, \Kup]$, and let $K = \max(|\Klow|, |\Kup|)$. For any $(x, s) \in \calO$ with $\|s\|_x \leq \frac{1}{\sqrt{K}}$,
	\begin{align}
		\sigmamin(T_s) & \geq \frac{2}{3} & & \textrm{ and } & \sigmamax(T_s) & \leq \frac{4}{3}.
	\end{align}
\end{corollary}
\begin{proof}
	By Proposition~\ref{prop:TsminPsparticular}, the operator norm of $T_s - P_s$ is bounded above by $\frac{1}{3} K \|s\|_x^2 \leq \frac{1}{3}$. Furthermore, parallel transport $P_s$ is an isometry: its singular values are equal to 1. Thus,
	\begin{align*}
		\sigmamax(T_s) & = \sigmamax(P_s + T_s - P_s) \leq \sigmamax(P_s) + \sigmamax(T_s - P_s) \leq 1 + \frac{1}{3} = \frac{4}{3}.
	\end{align*}
	Likewise, with min/max taken over unit-norm vectors $u \in \T_x\calM$ and writing $y = \Exp_x(s)$,
	\begin{align*}
		\sigmamin(T_s) & = \min_{u} \|T_s u\|_y \geq \min_{u} \|P_s u\|_y - \|(T_s - P_s)u\|_y = 1 - \max_{u} \|(T_s - P_s)u\|_y \geq \frac{2}{3}. \qedhere
	\end{align*}
\end{proof}
We turn to controlling the term $c''(0)$ which appears in the definition of operator $W_s$ in the expression for $\nabla^2 \hat f_x(s)$ provided by Lemma~\ref{lem:derivativespullback}.
We present a detailed proof in Appendix~\ref{app:controllingcprimeprime} for a general statement, and state a sufficient particular case here.
The proof is fairly technical: it involves designing an appropriate non-linear second-order ODE on the manifold and bounding its solutions.
The ODE is related to the Jacobi equation, except we had to differentiate to the next order, and the equation is not homogeneous.
We argue in the appendix that the result would be exact for manifolds with constant sectional curvature and with small $s$ if we optimized constants for that case.
\begin{proposition} \label{prop:cprimeprimecontrol}
	Let $\calM$ be a Riemannian manifold whose sectional curvatures are in the interval $[\Klow, \Kup]$, and let $K = \max(|\Klow|, |\Kup|)$.
	Further assume $\nabla R$ is bounded by $F$ in operator norm.
	%	\TODO{These conditions on curvature only need to hold along $\gamma$ I imagine? Would need to check the comparison theorems we called upon in the previous theorem too.}
	Pick any $(x, s) \in \calO$ such that
	\begin{align*}
		\|s\|_x & \leq \min\!\left( \frac{1}{4\sqrt{K}}, \frac{K}{4F} \right).
	\end{align*}
	For any $\dot s \in \T_x\calM$, the curve $c(t) = \Exp_x(s + t \dot s)$ has initial acceleration bounded as
	\begin{align*}
		\|c''(0)\|_{\Exp_x(s)} & \leq \frac{3}{2} K \|s\|_x \|\dot s\|_x \|\dot s_\perp\|_x,
	\end{align*}
	where $\dot s_\perp = \dot s - \frac{\inner{s}{\dot s}_x}{\inner{s}{s}_x} s$ is the component of $\dot s$ orthogonal to $s$.
\end{proposition}
Equipped with all of the above, it is now easy to prove the main theorem of this section.
\begin{proof}[Proof of Theorem~\ref{thm:pullbacklipschitz}.]
	Consider the pullback $\hat f_x = f \circ \Exp_x$ defined on $\T_x\calM$.
	Since $\T_x\calM$ is linear, it is a classical exercise to verify that $\nabla \hat f_x$ is $2L$-Lipschitz continuous in $B_x(b)$ if and only if $\|\nabla^2 \hat f_x(s)\|_x \leq 2L$ for all $s$ in $B_x(b)$.
	Using Lemma~\ref{lem:derivativespullback}, we start bounding the Hessian as follows:
	\begin{align*}
		\|\nabla^2 \hat f_x(s)\|_x & \leq \sigmamax(T_s^*) \sigmamax(T_s) \|\Hess f(\Exp_x(s))\|_{\Exp_x(s)} + \|W_s\|_x,
	\end{align*}
	with operator $W_s$ defined by~\eqref{eq:Whessianpullback}.
	Since $\grad f$ is $L$-Lipschitz continuous, $\|\Hess f(y)\|_y \leq L$ for all $y \in \calM$ (this follows fairly directly from Proposition~\ref{prop:lipschitzinequalitiesusual}).
	To bound $W_s$, we start with a Cauchy--Schwarz inequality then we consider the worst case for the magnitude of $c''(0)$:
	\begin{align*}
		\|W_s\|_x & \leq \|\grad f(\Exp_x(s))\|_{\Exp_x(s)} \cdot \max_{\dot s \in \T_x\calM, \|\dot s\|_x = 1} \|c''(0)\|_{\Exp_x(s)}.
	\end{align*}
	Combining these steps yields a first bound of the form
	\begin{align}
		\|\nabla^2 \hat f_x(s)\|_x & \leq \sigmamax(T_s)^2 L + \|\grad f(\Exp_x(s))\|_{\Exp_x(s)} \cdot \max_{\dot s \in \T_x\calM, \|\dot s\|_x = 1} \|c''(0)\|_{\Exp_x(s)}.
		\label{eq:initialboundnablatwo}
	\end{align}
%	\TODO{We argue through an example in Appendix~\ref{apdx:lipschitztightexample} that this upper-bound is tight up to $O(\|s\|_x^2)$ terms.---Could remove}
	To proceed, we keep working on the $W_s$-terms: use Proposition~\ref{prop:cprimeprimecontrol}, $L$-Lipschitz-continuity of the gradient, and our bounds on the norms of $s$ and $\grad f(x)$ to see that:
	\begin{align}
		\|W_s\|_x & \leq \max_{\dot s \in \T_x\calM, \|\dot s\|_x = 1} \|c''(0)\|_{\Exp_x(s)} \cdot \|\grad f(\Exp_x(s))\|_{\Exp_x(s)} \nonumber\\
					 & \leq \frac{3}{2} K \|s\|_x \cdot \|P_s^*\grad f(\Exp_x(s)) - \grad f(x) + \grad f(x)\|_{x} \nonumber\\
					 & \leq \frac{3}{2} K \|s\|_x \cdot \left( L \|s\|_x + \|\grad f(x)\|_x \right) \nonumber\\
					 & \leq 3KLb\|s\|_x % \nonumber\\
					   \leq \frac{3}{4} L \sqrt{K} \|s\|_x \leq \frac{3}{16}L. \label{eq:Wsbound}
	\end{align}
	Returning to~\eqref{eq:initialboundnablatwo} and using Corollary~\ref{cor:Tssingularvalues} to bound $T_s$ confirms that
	\begin{align*}
		\|\nabla^2 \hat f_x(s)\|_x & \leq \frac{16}{9} L + \frac{3}{16} L  < 2L.
	\end{align*}
	Thus, $\nabla \hat f_x$ is $2L$-Lipschitz continuous in the ball of radius $b$ around the origin in $\T_x\calM$.
	
	To establish the second part of the claim, we use the same intermediate results and $\rho$-Lipschitz continuity of the Hessian. First, using Lemma~\ref{lem:derivativespullback} twice and noting that $W_0 = 0$ so that $\nabla^2 \hat f_x(0) = \Hess f(x)$, we have:
	\begin{align*}
		\nabla^2 \hat f_x(s) - \nabla^2 \hat f_x(0) & = P_s^* \circ \Hess f(\Exp_x(s)) \circ P_s - \Hess f(x) \\
													& \qquad + (T_s - P_s)^* \circ \Hess f(\Exp_x(s)) \circ T_s \\
													& \qquad + P_s^* \circ \Hess f(\Exp_x(s)) \circ (T_s - P_s) \\
													& \qquad + W_s.
	\end{align*}
	We bound this line by line calling upon Proposition~\ref{prop:TsminPsparticular}, Corollary~\ref{cor:Tssingularvalues} and~\eqref{eq:Wsbound} to get:
	\begin{align*}
		\|\nabla^2 \hat f_x(s) - \nabla^2 \hat f_x(0)\|_x & \leq \rho \|s\|_x + \frac{4}{9} L K \|s\|_x^2 + \frac{1}{3} L K \|s\|_x^2 + 3LKb\|s\|_x \\
			& \leq \left( \rho + \frac{1}{9} L \sqrt{K} + \frac{1}{12} L \sqrt{K} + \frac{3}{4} L \sqrt{K} \right) \|s\|_x \\
			& \leq \left( \rho + L \sqrt{K} \right) \|s\|_x.
	\end{align*}
	This shows a type of Lipschitz continuity of the Hessian of the pullback with respect to the origin, in the ball of radius $b$.
\end{proof}

\section{Assumptions and parameters for $\ARGD$ and $\PARGD$} \label{sec:assuparams}

Our algorithms apply to the minimization of $f \colon \calM \to \reals$ on a Riemannian manifold $\calM$ equipped with a retraction $\Retr$ defined on the whole tangent bundle $\T\calM$.
The pullback of $f$ at $x \in \calM$ is $\hat f_x = f \circ \Retr_x \colon \T_x\calM \to \reals$.
In light of Section~\ref{sec:pullbacks}, we make the following assumptions.
\begin{assumption} \label{assu:Mandf}
	There exists a constant $\flow$ such that $f(x) \geq \flow$ for all $x \in \calM$.
	Moreover, $f$ is twice continuously differentiable and there exist constants $\ell$, $\hat \rho$ and $b$ such that, for all $x \in \calM$ with $\|\grad f(x)\| \leq \frac{1}{2} \ell b$,
	\begin{enumerate}
		\item $\nabla \hat f_x$ is $\ell$-Lipschitz continuous in $B_x(3b)$ (in particular, $\|\nabla^2 \hat f_x(0)\| \leq \ell$),
		\item $\|\nabla^2 \hat f_x(s) - \nabla^2 \hat f_x(0)\| \leq \hat\rho \|s\|$ for all $s \in B_x(3b)$, and
		\item $\sigmamin(T_s) \geq \frac{1}{2}$ with $T_s = \D\Retr_x(s)$ for all $s \in B_x(3b)$, % \TODO{do we need something about $\sigmamax$? I think not.}
	\end{enumerate}
	where $B_x(3b) = \{ u \in \T_x\calM : \|u\| \leq 3b \}$.
	Finally, for all $(x, s) \in \T\calM$ it holds that
	\begin{enumerate}[resume]
		\item  $\hat f_x(s) \leq \hat f_x(0) + \innersmall{\nabla \hat f_x(0)}{s} + \frac{\ell}{2}\|s\|^2$.
	\end{enumerate}
\end{assumption}
The first three items in~\aref{assu:Mandf} confer Lipschitz properties to the derivatives of the pullbacks $\hat f_x$ restricted to balls around the origins of tangent spaces: these are the balls where we shall run accelerated gradient steps.
We only need these guarantees at points where the gradient is below a threshold.
For all other points, a regular gradient step provides ample progress: the last item in~\aref{assu:Mandf} serves that purpose only, see Proposition~\ref{prop:Case1}.

Section~\ref{sec:pullbacks} tells us that~\aref{assu:Mandf} holds in particular when we use the exponential map as a retraction and $f$ itself has appropriate (Riemannian) Lipschitz properties.
This is the link between Theorems~\ref{thm:masterFOCPexp} and~\ref{thm:masterSOCPexp} in the introduction and Theorems~\ref{thm:masterFOCP} and~\ref{thm:masterSOCP} in later sections.
\begin{corollary}
	If we use the exponential retraction $\Retr = \Exp$ and~\aref{assu:Mandfintrinsic} holds, then~\aref{assu:Mandf} holds with $\flow$, 
	and 
	\begin{align} \label{ellrhohatb}
	\ell &= 2L, && \hat \rho = \rho + L\sqrt{K}, && b = \frac{1}{12}\min\!\left(\frac{1}{\sqrt{K}}, \frac{K}{F}\right).
	\end{align}
\end{corollary}

With constants as in~\aref{assu:Mandf}, we further define a number of parameters.
First, the user specifies a tolerance $\epsilon$ which must not be too loose: see Remark~\ref{remark:conditionsonepsilon} below.
\remove{
The first condition on $\epsilon$ is analogous to what one finds in the Euclidean case~\citep{jin2018agdescapes,carmon2017convexguilty}.
the second condition is specific to our treatment:}
\begin{assumption} \label{assu:epsilon}
	The tolerance $\epsilon > 0$ satisfies $\sqrt{\hat \rho \epsilon} \leq \frac{1}{2}\ell$ and $\epsilon \leq b^2 \hat \rho$.
\end{assumption}

\noindent Then, we fix a first set of parameters (see~\citep{jin2018agdescapes} for more context; in particular, $\kappa$ plays the role of a condition number; under \aref{assu:epsilon}, we have $\kappa \geq 2$):
\begin{align} \label{definingparams1}
	\eta & = \frac{1}{4\ell}, &
	\kappa & = \frac{\ell}{\sqrt{\hat \rho \epsilon}}, &
	\theta & = \frac{1}{4\sqrt{\kappa}}, &
	\gamma & = \frac{\sqrt{\hat\rho \epsilon}}{4}, &
	s & = \frac{1}{32} \sqrt{\frac{\epsilon}{\hat\rho}}.
\end{align}
We define a second set of parameters based on some $\chi \geq 1$ (as set in some of the lemmas and theorems below) and a universal constant $c > 0$ (implicitly defined as the smallest real satisfying a finite number of lower-bounds required throughout the paper):
\begin{align} \label{definingparams2}
	          r & = \eta \epsilon \chi^{-5}c^{-8}, &
	\mathscr{T} & = \sqrt{\kappa}\chi c, &
	\mathscr{E} & = \sqrt{\frac{\epsilon^3}{\hat{\rho}}}\chi^{-5}c^{-7}, &
	\mathscr{L} & = \sqrt{\frac{4\epsilon}{\hat{\rho}}} \chi^{-2}c^{-3}, &
	\mathscr{M} & = \frac{\epsilon \sqrt{\kappa}}{\ell}c^{-1}.
\end{align}
When we say ``with $\chi \geq A \geq 1$'' (for example, in Theorems~\ref{thm:masterFOCP} and~\ref{thm:masterSOCP}), we mean: ``with $\chi$ the smallest value larger than $A$ such that $\mathscr{T}$ is a positive integer multiple of $4$.''

Lemma~\ref{lem:params} in Appendix~\ref{app:params} lists useful relations between the parameters.
%In particular, it is arranged that $2\ell\mathscr{M} < \frac{1}{2} \ell b$.

\add{
\begin{remark} \label{remark:conditionsonepsilon}
Both conditions in~\aref{assu:epsilon} can be well understood, and neither is particularly stringent.

The first condition is analogous to what one finds in the Euclidean case~\citep{jin2018agdescapes,carmon2017convexguilty}.
This condition can be traced back to two facts.
First, our main theorem announces an improvement over Riemannian gradient descent by a factor $\sqrt{\frac{L}{\sqrt{\hat \rho \epsilon}}}$.
%For this to be a speed-up, we need $\sqrt{\frac{L}{\sqrt{\hat \rho \epsilon}}} \geq 1$, i.e, $\sqrt{\hat{\rho} \epsilon} \leq L$.
That factor is at least 1 exactly when $\sqrt{\hat{\rho} \epsilon} \leq L$.
Second, consider $\PARGD$ whose goal is to find an $(\epsilon, \sqrt{\hat{\rho} \epsilon})$-approximate second-order critical point.
If $L \leq \sqrt{\hat{\rho} \epsilon}$, then every point $x \in \calM$ satisfies $\norm{\Hess f(x)} \leq \sqrt{\hat{\rho} \epsilon}$ (because Lipschitz gradients imply bounded Hessians).
%In that regime, we can use Riemannian gradient descent or $\ARGD$ to find an approximate second-order critical point.  
Therefore, the task of finding approximate second-order critical points is only interesting when $\sqrt{\hat{\rho} \epsilon} \leq L$.

The second condition on $\epsilon$ is specific to our treatment.
This assumption is also mild.
First, note that this condition becomes \textit{less restrictive} as $L$ and $\rho$ \emph{increase}.
Second, consider the case where $\calM$ is a sphere or hyperbolic space of curvature $K \neq 0$.
We argue in the next paragraph that if $\epsilon$ is greater than a constant times $b^2 \hat \rho$ then every point $x \in \calM$ is $\epsilon$-approximate first-order critical.  %, and otherwise $\ARGD$ finds an $\epsilon$-approximate critical point in $\tilde{O}(L^{1/2} \hat{\rho}^{1/4} \epsilon^{-7/4})$ queries.  
Therefore, the optimization scenario is only interesting when $\epsilon \ll b^2 \hat{\rho}$.  %, as otherwise it is trivial to attain the target tolerance.

When $\calM$ is a sphere or hyperbolic space of curvature $K \neq 0$, it can be shown that any three times differentiable function $f$ with $\rho$-Lipschitz Hessian satisfies 
$$\norm{\grad f(x)} \leq \frac{2 \rho}{K} \leq 288 b^2 \hat \rho \text{ for all } x \in \calM.$$
This can be easily deduced from Proposition~\ref{prop:lipschitzinequalitiesusual} and the following identity, which follows from applying the Ricci identity~\citep[Thm.~7.14]{lee2018riemannian} to $\grad f$:
$$\nabla^3 f(U, W, V) - \nabla^3 f(U, V, W) = \inner{R(W, V)U}{\grad f}.$$
Therefore, if $\epsilon \geq 288 b^2 \hat \rho$, then all points $x \in \calM$ are $\epsilon$-approximate first-order critical points.
%If instead $\epsilon \leq 288 b^2 \hat \rho$, then $\frac{\epsilon}{288} \leq b^2 \hat \rho$, and we can ask $\ARGD$ to find an $\frac{\epsilon}{288}$-approximate first order critical point.
%This course costs $\tilde{O}(L^{1/2} \hat{\rho}^{1/4} \epsilon^{-7/4})$ by our main theorems.
%Finally,~\citet{sun2019prgd} also make a similar, but less explicit, assumption on $\epsilon$.
\end{remark}
}

\section{Accelerated gradient descent in a ball of a tangent space} \label{sec:agdTxM}

\begin{algorithm}[t]
	\caption{$\TSS(x, s_0)$ with $(x, s_0) \in \T\calM$ and parameters $\epsilon, \eta, b, \theta, \gamma, s, \mathscr{T}$} % \ell, r, T, \mathscr{M}
	\label{algo:TSS}
	\begin{algorithmic}[1]
		\State If $s_0$ is not provided, set $s_0 = 0$ and $\perturbed = \textbf{false}$; otherwise, set $\perturbed = \textbf{true}$.
		\State $v_0 = 0$
		\For {$j$ in $0, 1, \ldots, \mathscr{T} - 1$}
			\State\label{step:uj} $u_j = s_j + (1-\theta_j) v_j$ with \Comment{AGD: capped momentum step}
			\begin{align}
				\theta_j = \begin{cases}
				\theta \textrm{ if } \|s_j + (1-\theta) v_j\| \leq 2b, \\
				\hat \theta \in [\theta, 1] \textrm{ such that } \|s_j + (1-\hat\theta) v_j\| = 2b \textrm{ otherwise}.
				\end{cases}
				\label{eq:thetaj}
			\end{align}
			\If {\eqref{eq:NCC} triggers with $(x, s_j, u_j)$} \Comment{Negative curvature detection} % $\hat{f}_x(s_j) < \hat{f}_x(u_j) + \innersmall{\nabla \hat{f}_x(u_j)}{s_j - u_j} - \frac{\gamma}{2}\norm{s_j - u_j}^2$
				\State \textbf{return} $\Retr_x(\NCE(x, s_j, v_j))$ \Comment{(Cases 2a, 3a)} % and \textbf{false}
			\EndIf
			\State $s_{j+1} = u_j - \eta \nabla \hat{f}_x(u_j)$ \Comment{AGD: gradient step}
			\State $v_{j+1} = s_{j+1} - s_j$ \Comment{AGD: momentum update}
			%\TODO{important that below inequality is strict ...}
			%\Comment{Leave tangent space $\T_x\calM$ prematurely.}
			%\If {$\norm{u_j'} \geq b$} \Comment{Cases 2(b) and 3(c): $u_j'$ leaves ball $B_x(b)$.}
			%\State $s_{\mathscr{T}} \gets u_j$
			%\State \textbf{break}
			%\EndIf
			\If{\Big($\|s_{j+1}\| > b$\Big) \textbf{or} \Big((\textbf{not} $\perturbed$) \textbf{and} $\|\nabla \hat{f}_x(s_{j + 1})\| \leq \epsilon / 2$\Big)}
				\label{step:leaveball}
				\State \textbf{return} $\Retr_x(s_{j+1})$ \Comment{(Cases 2b, 2c, 3b)} % and \textbf{false}
				%\Comment{Leave tangent space $\T_x\calM$ prematurely.}
			\EndIf
		\EndFor
		\State \textbf{return} $\Retr_x{(s_{\mathscr{T}})}$ % and the boolean ``$f(\Retr_x(s_{\mathscr{T}})) \geq f(x) - \mathscr{E}/2$'' %\Comment{Return to the manifold $\calM$.}
		\Comment{(Cases 2d, 3d)}
	\end{algorithmic}
\end{algorithm}

The main ingredient of algorithms $\ARGD$ and $\PARGD$ is $\TSS$: the \emph{tangent space steps} algorithm.
Essentially, the latter runs the classical accelerated gradient descent algorithm (AGD) from convex optimization on $\hat f_x$ in a tangent space $\T_x\calM$, with a few tweaks:
\begin{enumerate}
	\item Because $\hat f_x$ need not be convex, $\TSS$ monitors the generated sequences for signs of non-convexity.
	      If $\hat f_x$ happens to behave like a convex function along the sequence $\TSS$ generates, then we reap the benefits of convexity.
	      Otherwise, the direction along which $\hat f_x$ behaves in a non-convex way can be used as a good descent direction.
	      This is the idea behind the ``convex until proven guilty'' paradigm developed by~\citet{carmon2017convexguilty} and also exploited by~\citet{jin2018agdescapes}.
	      Explicitly, given $x \in \calM$ and $s, u \in \T_x\calM$, for a specified parameter $\gamma > 0$, we check the \emph{negative curvature condition} (one might also call it the \emph{non-convexity condition)}~\eqref{eq:NCC}:
	      \begin{align}
		      \hat{f}_x(s) < \hat{f}_x(u) + \innersmall{\nabla \hat{f}_x(u)}{s - u} - \frac{\gamma}{2}\norm{s - u}^2.
		      \tag{NCC}
	    	  \label{eq:NCC}
	      \end{align}
	      If~\eqref{eq:NCC} triggers with a triplet $(x, s, u)$ and $s$ is not too large, we can exploit that fact to generate substantial cost decrease using the \emph{negative curvature exploitation} algorithm, $\NCE$: see Lemma~\ref{lem:TSSNCE}.
	      (This is about curvature of the cost function, not the manifold.)
	\item In contrast to the Euclidean case in~\citep{jin2018agdescapes}, our assumption~\aref{assu:Mandf} provides Lipschitz-type guarantees only in a ball of radius $3b$ around the origin in $\T_x\calM$.
	      Therefore, we must act if iterates generated by $\TSS$ leave that ball.
	      This is done in two places.
	      First, the momentum step in step~\ref{step:uj} of $\TSS$ is capped so that $\|u_j\|$ remains in the ball of radius $2b$ around the origin.
	      % This may require adjusting the parameter $\theta_j$ which may become as large as $1$.
	      Second, if $s_{j+1}$ leaves the ball of radius $b$ (as checked in step~\ref{step:leaveball}) then we terminate this run of $\TSS$ by returning to the manifold.
	      Lemma~\ref{lem:TSSballs} guarantees that the iterates indeed remain in appropriate balls, that $\theta_j$~\eqref{eq:thetaj} in the capped momentum step is uniquely defined, and that if a momentum step is capped, then immediately after that $\TSS$ terminates.
\end{enumerate}
The initial momentum $v_0$ is always set to zero.
By default, the AGD sequence is initialized at the origin: $s_0 = 0$.
However, for $\PARGD$ we sometimes want to initialize at a different point (a perturbation away from the origin): this is only relevant for Section~\ref{sec:socp}.

\begin{algorithm}[t]
	\caption{$\NCE(x, s_j, v_j)$ with $x \in \calM$, $s_j, v_j \in \T_x\calM$ and parameter $s$}
	\label{algo:NCE}
	\begin{algorithmic}[1]
		\If {$\norm{v_j} \geq s$}
			\State \textbf{return} $s_j$
		\Else
			\State $\dot{v} = s \frac{v_j}{\norm{v_j}}$
			\State \textbf{return} $\text{argmin}_{ \dot{s} \in \{s_j, s_j + \dot{v},s_j - \dot{v}\}}{\hat{f}_x(\dot{s})}$
		\EndIf
	\end{algorithmic}
\end{algorithm}

In the remainder of this section, we provide four general purpose lemmas about $\TSS$.
Proofs are in Appendix~\ref{app:agdTxM}.
We note that $\ARGD$ and $\PARGD$ call $\TSS$ only at points $x$ where $\|\grad f(x)\| \leq \frac{1}{2} \ell b$.
The first lemma below notably guarantees that, for such runs, all iterates $u_j, s_j$ generated by $\TSS$ remain (a fortiori) in balls of radius $3b$, so that the strongest provisions of~\aref{assu:Mandf} always apply: we use this fact often without mention.
\begin{lemma}[$\TSS$ stays in balls] \label{lem:TSSballs}
	Fix parameters and assumptions as laid out in Section~\ref{sec:assuparams}.
	Let $x \in \calM$ satisfy $\|\grad f(x)\| \leq \frac{1}{2}\ell b$.
	If $\TSS(x)$ or $\TSS(x, s_0)$ (with $\|s_0\| \leq b$) defines vectors $u_0, \ldots, u_q$ (and possibly more), then it also defines vectors $s_0, \ldots, s_q$, and we have:
	\begin{align*}
		\|s_0\|, \ldots, \|s_q\| & \leq b,  &  \|u_0\|, \ldots, \|u_q\| & \leq 2b, & \textrm{ and } && 2\eta\gamma \leq \theta \leq \theta_j \leq 1.
	\end{align*}
%	\begin{enumerate}
%		\item $\|s_0\|, \ldots, \|s_q\| \leq b$ and
%		\item $\|u_0\|, \ldots, \|u_q\| \leq 2b$.
%	\end{enumerate}
	If $s_{q+1}$ is defined, then
	$\|s_{q+1}\| \leq 3b$ and,
	if $\|u_q\| = 2b$, then $\|s_{q+1}\| > b$
	and $u_{q+1}$ is undefined.
%	\begin{enumerate}
%		\item $\|s_{q+1}\| \leq 3b$ and
%		\item If $\|u_q\| = 2b$, then $\|s_{q+1}\| > b$ and $u_{q+1}$ is \emph{not} defined.
%	\end{enumerate}
\end{lemma}
Along the iterates of AGD, the value of the cost function $\hat f_x$ may not monotonically decrease.
Fortunately, there is a useful quantity which monotonically decreases along iterates: \citet{jin2018agdescapes} call it the \emph{Hamiltonian}.
In several ways, it serves the purpose of a Lyapunov function.
Importantly, the Hamiltonian decreases regardless of any special events that occur while running $\TSS$.
It is built as a combination of the cost function value and the momentum.
The next lemma makes this precise: we use monotonic decrease of the Hamiltonian often without mention.
This corresponds to~\citep[Lem.~9 and~20]{jin2018agdescapes}.
\begin{lemma}[Hamiltonian decrease] \label{lem:TSSEj}
	Fix parameters and assumptions as laid out in Section~\ref{sec:assuparams}.
	Let $x \in \calM$ satisfy $\|\grad f(x)\| \leq \frac{1}{2}\ell b$.
	For each pair $(s_j, v_j)$ defined by $\TSS(x)$ or $\TSS(x, s_0)$ (with $\|s_0\| \leq b$), define the Hamiltonian
	\begin{align}
		E_j & = \hat f_x(s_j) + \frac{1}{2\eta} \|v_j\|^2.
		\label{eq:Ej}
	\end{align}
	If $E_{j+1}$ is defined, then $E_j$, $\theta_j$ and $u_j$ are also defined and: % \TODO{check: $\theta_j = \theta$? --- no, not necessarily.}
	\begin{align*}
		E_{j+1} & \leq E_j - \frac{\theta_j}{2\eta} \|v_j\|^2 -  \frac{\eta}{4} \|\nabla \hat f_x(u_j)\|^2 \leq E_j.
	\end{align*}
	If moreover $\|v_j\| \geq \mathscr{M}$,
%	or \TODO{$\|\nabla \hat f_x(s_j)\| \geq 2\ell\mathscr{M}$--needed?},
	then $E_j - E_{j+1} \geq \frac{4\mathscr{E}}{\mathscr{T}}$.
\end{lemma}
\citet{jin2018agdescapes} formalize an important property of $\TSS$ sequences in the Euclidean case, namely, the fact that ``either the algorithm makes significant progress or the iterates do not move much.''
They call this the \emph{improve or localize} phenomenon.
The next lemma states this precisely in our context.
This corresponds to~\citep[Cor.~11]{jin2018agdescapes}.
\begin{lemma}[Improve or localize] \label{lem:TSStraveldist}
	Fix parameters and assumptions as laid out in Section~\ref{sec:assuparams}.
	Let $x \in \calM$ satisfy $\|\grad f(x)\| \leq \frac{1}{2}\ell b$.
	If $\TSS(x)$ or $\TSS(x, s_0)$ (with $\|s_0\| \leq b$) defines vectors $s_0, \ldots, s_q$ (and possibly more), then $E_0, \ldots, E_q$ are defined by~\eqref{eq:Ej} and, for all $0 \leq q' \leq q$,
	\begin{align*}
		\|s_q - s_{q'}\|^2 \leq (q-q') \sum_{j = q'}^{q-1} \|s_{j+1} - s_j\|^2 \leq 16 \sqrt{\kappa} \eta (q-q') (E_{q'} - E_q).
	\end{align*}
	For $q' = 0$ in particular, using $E_0 = \hat f_x(s_0)$ we can write $E_q \leq \hat f_x(s_0) - \frac{\|s_q - s_0\|^2}{16 \sqrt{\kappa} \eta q}$.
\end{lemma}
As outlined earlier, in case the $\TSS$ sequence witnesses non-convexity in $\hat f_x$ through the~\eqref{eq:NCC} check, we call upon the $\NCE$ algorithm to exploit this event.
The final lemma of this section formalizes the fact that this yields appropriate cost improvement.
(Indeed, if $\|s_j\| > \mathscr{L}$ one can argue that sufficient progress was already achieved; otherwise, the lemma applies and we get a result from $E_j \leq E_0 = \hat f_x(s_0)$.)
This corresponds to~\citep[Lem.~10 and 17]{jin2018agdescapes}.
\begin{lemma}[Negative curvature exploitation] \label{lem:TSSNCE}
	Fix parameters and assumptions as laid out in Section~\ref{sec:assuparams}.
	Let $x \in \calM$ satisfy $\|\grad f(x)\| \leq \frac{1}{2}\ell b$.
	Assume $\TSS(x)$ or $\TSS(x, s_0)$ (with $\|s_0\| \leq b$) defines $u_j$, so that $s_j, v_j$ are also defined, and $E_j$ is defined by~\eqref{eq:Ej}.
	If \eqref{eq:NCC} triggers with $(x, s_j, u_j)$
%	, then \TODO{I don't think we use the first statement anywhere---move to proof?}
%	\begin{align*}
%		\hat f_x(\NCE(x, s_j, v_j)) & \leq E_j - \frac{s^2}{2} \min\!\left( \frac{1}{\eta}, \gamma - 2\hat\rho(s + \|s_j\|) \right).
%	\end{align*}
%	In particular, if
	and
	$\|s_j\| \leq \mathscr{L}$, then $\hat f_x(\NCE(x, s_j, v_j)) \leq E_j - 2\mathscr{E}$.
\end{lemma}

\section{First-order critical points} \label{sec:focp}

Our algorithm to compute $\epsilon$-approximate first-order critical points on Riemannian manifolds is $\ARGD$: this is a deterministic algorithm which does not require access to the Hessian of the cost function.
Our main result regarding $\ARGD$, namely, Theorem~\ref{thm:masterFOCP}, states that it does so in a bounded number of iterations.
As worked out in Theorem~\ref{thm:masterFOCPexp}, this bound scales as $\epsilon^{-7/4}$, up to polylogarithmic terms.
The complexity is independent of the dimension of the manifold.
\add{In Appendix~\ref{backtrackingTAGD}, we present a simple modification of $\ARGD$, called $\mathtt{backtrackingTAGD}$, which also finds $\epsilon$-approximate first-order critical points in $\tilde{O}(\epsilon^{-7/4})$ without knowledge of $L$ or $\rho$.}

The proof of Theorem~\ref{thm:masterFOCP} rests on two propositions introduced hereafter in this section.
They themselves rest on two lemmas introduced later still in this section.
Interestingly, it is only in the proof of Theorem~\ref{thm:masterFOCP} that we track the behavior of iterates of $\ARGD$ across multiple points on the manifold.
This is done by tracking decrease of the value of the cost function $f$.
All supporting results (lemmas and propositions) handle a single tangent space at a time.
As a result, lemmas and propositions fully benefit from the linear structure of tangent spaces.
This is why we can salvage most of the Euclidean proofs of \citet{jin2018agdescapes}, up to mostly minor (but numerous and necessary) changes.

\begin{algorithm}[t]
	\caption{$\ARGD(x_0)$ with $x_0 \in \calM$ and parameters $\epsilon, \ell, \eta, b, \theta, \gamma, s, \mathscr{T}, \mathscr{M}$} %, r, T
	\label{algo:ARGD}
	\begin{algorithmic}[1]
		\State $t \gets 0$
		\While {\textbf{true}} % $t \leq T$
			\If {$\norm{\grad f(x_t)} > 2 \ell \mathscr{M}$}
				\State $x_{t+1} = \Retr_{x_t}(-\eta \grad f(x_t))$ \Comment{Case 1: one Riemannian gradient step}
				\State {$t \gets t+1$}
			\ElsIf {$\norm{\grad f(x_t)} > \epsilon$}
				\State $x_{t+\mathscr{T}} = \TSS(x_t)$ \Comment{Case 2: accelerated gradient in $\T_{x_t}\calM$}
				\State $t \gets t + \mathscr{T}$
			\Else
				\State \textbf{return} $x_t$ \Comment{Approximate FOCP}
			\EndIf
		\EndWhile
	\end{algorithmic}
\end{algorithm}

\begin{theorem} \label{thm:masterFOCP}
	Fix parameters and assumptions as laid out in Section~\ref{sec:assuparams}, with
	\begin{align} \label{chidefinition}
		\chi \geq \log_2(\theta^{-1}) \geq 1.
	\end{align}
	Given $x_0 \in \calM$, $\ARGD(x_0)$ returns $x_t \in \calM$ satisfying $f(x_t) \leq f(x_0)$ and $\|\grad f(x_t)\| \leq \epsilon$ with
	\begin{align} \label{T1definition}
		t \leq T_1 \triangleq \frac{f(x_0) - \flow}{\mathscr{E}} \mathscr{T}.
	\end{align}
%	To reach termination,
	Running the algorithm requires at most $2T_1$ pullback gradient queries and $3T_1$ function queries (but no Hessian queries), and a similar number of calls to the retraction.
\end{theorem}
\begin{proof}[Proof of Theorem~\ref{thm:masterFOCP}]
	The call to $\ARGD(x_0)$ generates a sequence of points $x_{t_0}, x_{t_1}, x_{t_2}, \ldots$ on $\calM$, with $t_0 = 0$.
	A priori, this sequence may be finite or infinite.
	Considering two consecutive indices $t_i$ and $t_{i+1}$, we either have $t_{i+1} = t_i + 1$ (if the step from $x_{t_i}$ to $x_{t_{i+1}}$ is a single gradient step (Case 1)) or $t_{i+1} = t_i + \mathscr{T}$ (if that same step is obtained through a call to $\TSS$ (Case 2)). Moreover:
	\begin{itemize}
		\item In Case 1, Proposition~\ref{prop:Case1} applies and guarantees
		\begin{align*}
		 	f(x_{t_i}) - f(x_{t_{i+1}}) \geq \frac{\mathscr{E}}{\mathscr{T}} = \frac{\mathscr{E}}{\mathscr{T}}(t_{i+1} - t_i).
		\end{align*}
		\item In Case 2, Proposition~\ref{prop:Case2} applies and guarantees that if $\|\grad f(x_{t_{i+1}})\| > \epsilon$ then
		\begin{align*}
			f(x_{t_i}) - f(x_{t_{i+1}}) \geq \mathscr{E} = \frac{\mathscr{E}}{\mathscr{T}}(t_{i+1} - t_i).
		\end{align*}
	\end{itemize}
	It is now clear that $\ARGD(x_0)$ terminates after a finite number of steps.
	Indeed, if it does not, then the above reasoning shows that the algorithm produces an amortized decrease in the cost function $f$ of $\frac{\mathscr{E}}{\mathscr{T}}$ per unit increment of the counter $t$, yet the value of $f$ cannot decrease by more than $f(x_0) - \flow$ because $f$ is globally lower-bounded by $\flow$.

	Accordingly, assume $\ARGD(x_0)$ generates $x_{t_0}, \ldots, x_{t_k}$ and terminates there, returning $x_{t_k}$.
	We know that $f(x_{t_k}) \leq f(x_0)$ and $\|\grad f(x_{t_k})\| \leq \epsilon$.
	Moreover, from the discussion above and $t_0 = 0$, we know that
	\begin{align*}
		f(x_0) - \flow \geq f(x_0) - f(x_{t_k}) = \sum_{i = 0}^{k-1} f(x_{t_i}) - f(x_{t_{i+1}}) \geq \frac{\mathscr{E}}{\mathscr{T}} \sum_{i = 0}^{k-1} t_{i+1} - t_i = \frac{\mathscr{E}}{\mathscr{T}} t_k.
	\end{align*}
	Thus, $t_k \leq \frac{f(x_0) - \flow}{\mathscr{E}} \mathscr{T} \triangleq T_1$.

	How much work does it take to run the algorithm?
	Each (regular) gradient step requires one gradient query and increases the counter by one.
	Each run of $\TSS$ requires at most $2 \mathscr{T}$ gradient queries and $2\mathscr{T} + 3 \leq 3\mathscr{T}$ function queries ($3 \leq \mathscr{T}$ because $\mathscr{T}$ is a positive integer multiple of $4$) and increases the counter by $\mathscr{T}$.
	Therefore, by the time $\ARGD$ produces $x_t$ it has used at most $2t$ gradient queries and $3t$ function queries.
\end{proof}

The two following propositions form the backbone of the proof of Theorem~\ref{thm:masterFOCP}.
Each handles one of the two possible cases in one (outer) iteration of $\ARGD$, namely:
Case~1 is a ``vanilla'' Riemannian gradient descent step, while Case~2 is a call to $\TSS$ to run (modified) AGD in the current tangent space.
The former has a trivial and standard proof.
The latter relies on all lemmas from Section~\ref{sec:agdTxM} and on two additional lemmas introduced later in this section, all following~\citet{jin2018agdescapes}.
\begin{proposition}[Case 1] \label{prop:Case1}
	Fix parameters and assumptions as laid out in Section~\ref{sec:assuparams}.
	Assume $x \in \calM$ satisfies $\|\grad f(x)\| > 2\ell \mathscr{M}$.
	Then, $x_+ = \Retr_{x}(-\eta \grad f(x))$ satisfies $f(x) - f(x_+) \geq \frac{\mathscr{E}}{\mathscr{T}}$.
\end{proposition}
\begin{proof}[Proof of Proposition~\ref{prop:Case1}]
	This follows directly by property 4 in~\aref{assu:Mandf} with $\hat f_x = f \circ \Retr_x$ since $\hat f_x(0) = f(x)$ and $\nabla \hat f_x(0) = \grad f(x)$ by properties of retractions, and also using $\ell\eta = 1/4$:
	\begin{align*}
		f(x_+) & = \hat f_x(-\eta \grad f(x))
		         \leq \hat f_x(0) - \eta \|\grad f(x)\|^2 + \frac{\ell}{2} \|\eta \grad f(x)\|^2
%		         = f(x) - \eta\left(1 - \frac{\ell\eta}{2}\right)\|\grad f(x)\|^2
		         \leq f(x) - (7/8) \ell \mathscr{M}^2.
	\end{align*}
	To conclude, it remains to use that $(7/8) \ell \mathscr{M}^2 \geq \frac{\mathscr{E}}{\mathscr{T}}$, as shown in Lemma~\ref{lem:params}.
\end{proof}
The next proposition corresponds mostly to~\citep[Lem.~12]{jin2018agdescapes}.
\begin{proposition}[Case 2] \label{prop:Case2}
	Fix parameters and assumptions as laid out in Section~\ref{sec:assuparams}, with
	\begin{align} \label{chidefintion}
		\chi \geq \log_2(\theta^{-1}) \geq 1.
	\end{align}
	If $x \in \calM$ satisfies $\|\grad f(x)\| \leq 2\ell \mathscr{M}$,
	then $x_{\mathscr{T}} = \TSS(x)$ falls in one of two cases:
	\begin{enumerate}
		\item Either $\|\grad f(x_{\mathscr{T}})\| \leq \epsilon$ and $f(x) - f(x_{\mathscr{T}}) \geq 0$,
		\item Or $\|\grad f(x_{\mathscr{T}})\| > \epsilon$ and $f(x) - f(x_{\mathscr{T}}) \geq \mathscr{E}$.
% The following one is wrong
%		\item Either $f(x) - f(x_{\mathscr{T}}) \geq \mathscr{E}$,
%		\item Or $0 \leq f(x) - f(x_{\mathscr{T}}) < \mathscr{E}$ and $\|\grad f(x_{\mathscr{T}})\| \leq \epsilon$.
	\end{enumerate}
\end{proposition}
\begin{proof}[Proof of Proposition~\ref{prop:Case2}]
	By Lemma~\ref{lem:params}, $\|\grad f(x)\| \leq 2\ell\mathscr{M} < \frac{1}{2} \ell b$. Thus, the strongest provisions of~\aref{assu:Mandf} apply at $x$, as do Lemmas~\ref{lem:TSSballs}, \ref{lem:TSSEj}, \ref{lem:TSStraveldist} and~\ref{lem:TSSNCE}.
	Let $u_j, s_j, v_j$ for $j = 0, 1, \ldots$ be the vectors generated by the computation of $x_{\mathscr{T}} = \TSS(x)$.
	Note that $s_0 = v_0 = 0$.
	There are several cases to consider, based on how $\TSS$ terminates:
	\begin{itemize}
		\item (Case 2a) The negative curvature condition \eqref{eq:NCC} triggers with $(x, s_j, u_j)$. % for some $j < \mathscr{T}$.
			  There are two cases to check.
			  Either $\|s_j\| \leq \mathscr{L}$, in which case Lemma~\ref{lem:TSSEj} tells us $E_j \leq E_0 = f(x)$ and Lemma~\ref{lem:TSSNCE} further tells us that
			  \begin{align*}
			  	f(x_{\mathscr{T}}) = \hat f_x(\NCE(x, s_j, v_j)) \leq E_j - 2\mathscr{E} \leq f(x) - 2\mathscr{E}.
			  \end{align*}
			  Or $\|s_j\| > \mathscr{L}$, in which case Lemma~\ref{lem:TSStraveldist} used with $q = j < \mathscr{T}$ and $s_0 = 0$ implies
			  \begin{align*}
			  	E_j \leq f(x) - \frac{\mathscr{L}^2}{16 \sqrt{\kappa} \eta \mathscr{T}} = f(x) - \mathscr{E}.
			  \end{align*}
			  (See Lemma~\ref{lem:params} for that last equality.)
			  Owing to how $\NCE$ works, we always have $f(x_{\mathscr{T}}) = \hat f_x(\NCE(x, s_j, v_j)) \leq \hat f_x(s_j) \leq E_j$ (the last inequality is by definition of $E_j$~\eqref{eq:Ej}).
			  Thus, we conclude that $f(x_{\mathscr{T}}) \leq f(x) - \mathscr{E}$.
			  		
		\item (Case 2b) The iterate $s_{j+1}$ leaves the ball of radius $b$, that is, $\|s_{j+1}\| > b$. % for some $j < \mathscr{T}$.
			  In this case, apply Lemma~\ref{lem:TSStraveldist} with $q = j + 1 \leq \mathscr{T}$ and $s_0 = 0$ to claim
			  \begin{align*}
				f(x_{\mathscr{T}}) = \hat f_x(s_{j+1}) \leq E_{j+1} \leq f(x) - \frac{\|s_{j+1}\|^2}{16 \sqrt{\kappa} \eta \mathscr{T}} \leq f(x) - \frac{\mathscr{L}^2}{16 \sqrt{\kappa} \eta \mathscr{T}} = f(x) - \mathscr{E}.
			  \end{align*}
			  (The first inequality is by definition of $E_{j+1}$~\eqref{eq:Ej}; subsequently, we use $\|s_{j+1}\| > b > \mathscr{L}$ as in Lemma~\ref{lem:params}.)
		
		\item (Case 2c) The iterate $s_{j+1}$ satisfies $\|\nabla \hat f_x(s_{j+1})\| \leq \epsilon/2$. % for some $j < \mathscr{T}$.
		      Recall the chain rule identity relating gradients of $f$ and gradients of the pullback $\hat f_x = f \circ \Retr_x$ with $T_s = \D\Retr_x(s)$:
		      \begin{align*}
		      	\nabla \hat f_x(s) = T_s^* \grad f(\Retr_x(s)).
		      \end{align*}
		      In our situation, $x_{\mathscr{T}} = \Retr_x(s_{j+1})$ and $\|s_{j+1}\| \leq b$ (otherwise, Case 2b applies).
		      Thus, \aref{assu:Mandf} ensures $\sigmamin(T_{s_{j+1}}) \geq \frac{1}{2}$ and we deduce that
		      \begin{align*}
		      	\|\grad f(x_{\mathscr{T}})\| = \|(T_{s_{j+1}}^*)^{-1} \nabla \hat f_x(s_{j+1})\| \leq \|(T_{s_{j+1}}^*)^{-1}\| \|\nabla \hat f_x(s_{j+1})\|\| \leq 2 \cdot \frac{\epsilon}{2} = \epsilon.
		      \end{align*}
		
		\item (Case 2d) None of the other events occur: $\TSS(x)$ runs its $\mathscr{T}$ iterations in full. %, with $\theta_j = \theta$ for all $j$.
			  In this case, we apply the logic in the proof of~\citep[Lem.~12]{jin2018agdescapes}, as follows.
			  We consider two cases.
			  In the first case, $E_{0} - E_{\mathscr{T}/2} > \mathscr{E}$.
			  Then, we apply Lemma~\ref{lem:TSSEj} to claim that $E_{0} - E_{\mathscr{T}} \geq E_{0} - E_{\mathscr{T}/2} \geq \mathscr{E}$.
			  Moreover, $E_0 = f(x)$ and $E_{\mathscr{T}} \geq \hat f_x(s_\mathscr{T}) = f(x_{\mathscr{T}})$. Thus, in this case, $f(x) - f(x_{\mathscr{T}}) \geq \mathscr{E}$.
			  In the second case, $E_{0} - E_{\mathscr{T}/2} \leq \mathscr{E}$.
			  Then, Lemma~\ref{lem:focpB} applies and we learn the following:
			  Let $\calS$ denote the linear subspace of $\T_x\calM$ spanned by the eigenvectors of $\nabla^2 \hat f_x(0)$ associated to eigenvalues strictly larger than $\frac{\theta^2}{\eta(2-\theta)^2}$.
			  Let $P_\calS$ denote orthogonal projection to $\calS$.
			  For each $j$ in $\{ \mathscr{T}/4, \ldots, \mathscr{T}/2 \}$ we have
			  \begin{align*}
				  \|P_\calS \nabla \hat f_x(s_j)\| & \leq \epsilon / 6 & \textrm{ and } && \inner{P_\calS v_j}{\nabla^2 \hat f_x(0)[P_\calS v_j]} & \leq \sqrt{\hat\rho \epsilon} \mathscr{M}^2.
			  \end{align*}
			  Let $\tau$ be the first index in the range $\{\mathscr{T}/4, \ldots, \mathscr{T}\}$ for which $\|v_\tau\| \leq \mathscr{M}$.
			  Again, there are two possibilities.
			  In the first case, $\tau > \mathscr{T}/2$.
			  Then, $\|v_j\| > \mathscr{M}$ for all $j$ in $\{\mathscr{T}/4, \ldots, \mathscr{T}/2\}$.
			  The last part of Lemma~\ref{lem:TSSEj} implies that, for each such $j$, $E_{j} - E_{j+1} \geq \frac{4\mathscr{E}}{\mathscr{T}}$.
			  It follows that $E_{\mathscr{T}/4} - E_{\mathscr{T}/2} \geq \mathscr{E}$.
			  Conclude this case with Lemma~\ref{lem:TSSEj} which justifies these statements: $f(x) = E_0$, $f(x_{\mathscr{T}}) = \hat f_x(s_{\mathscr{T}}) \leq E_{\mathscr{T}}$, and:
			  \begin{align*}
			  	f(x) - f(x_{\mathscr{T}}) \geq E_0 - E_{\mathscr{T}} \geq E_{\mathscr{T}/4} - E_{\mathscr{T}/2} \geq \mathscr{E}.
			  \end{align*}
			  In the second case, $\tau \in \{\mathscr{T}/4, \ldots, \mathscr{T}/2\}$.
			  We aim to apply Lemma~\ref{lem:focpA}: there are a few preconditions to check.
			  Here is what we already know:
			  \begin{align*}
			  	\|v_\tau\| & \leq \mathscr{M}, & \inner{P_\calS v_\tau}{\nabla^2 \hat f_x(0)[P_\calS v_\tau]} & \leq \sqrt{\hat\rho \epsilon} \mathscr{M}^2, & \textrm{ and } && \|P_\calS \nabla \hat f_x(s_\tau)\| & \leq \epsilon / 6.
			  \end{align*}
			  Regarding the third one above: we know that $\|\nabla \hat{f}_x(s_{\tau})\| > \epsilon / 2$ because $\TSS(x)$ did not terminate with $s_\tau$.
			  We deduce that
			  \begin{align*}
			  	\|\nabla \hat f_x(s_\tau) - P_{\calS} \nabla \hat f_x(s_\tau)\| & \geq \|\nabla \hat f_x(s_\tau)\| - \|P_{\calS} \nabla \hat f_x(s_\tau)\| \geq \frac{\epsilon}{2} - \frac{\epsilon}{6} > \frac{\epsilon}{6}.
			  \end{align*}
			  We now have a final pair of cases to check.
			  Either $\|s_\tau\| \leq \mathscr{L}$, in which case Lemma~\ref{lem:focpA} applies: it follows that $E_{\tau-1} - E_{\tau + \mathscr{T}/4} \geq \mathscr{E}$, and by arguments similar as above we conclude that $f(x) - f(x_{\mathscr{T}}) \geq \mathscr{E}$.
			  Or $\|s_\tau\| > \mathscr{L}$, in which case Lemma~\ref{lem:TSStraveldist} implies (using $s_0 = 0$):
			  \begin{align*}
			  	f(x_{\mathscr{T}}) \leq E_{\mathscr{T}} \leq E_\tau \leq f(x) - \frac{\mathscr{L}^2}{16\sqrt{\kappa}\eta\tau} \leq f(x) - \mathscr{E}.
			  \end{align*}
			  (For the second and last inequalities, we use $\tau < \mathscr{T}$ and Lemmas~\ref{lem:TSSEj} and~\ref{lem:params}.)
	\end{itemize}
	This covers all possibilities.
\end{proof}

%\TODO{
%	mention that for FOCP target we only need to consider $\TSS(x)$, but the next lemma is helpful for SOCP target too (is that so?). -- I think so, actually: move this lemma to the common parts; here, only keep the lemma that introduces limitations on $\chi$. Relabel them. --- Well, no, actually,
%	It seems like Lemma~\ref{lem:focpA} is only needed for FOCP. Thus, we could simplify its statement somewhat by making it only about $\TSS(x)$. It's not a big gain of course, but... TBD.
%	Or perhaps it comes up not in the proof of the SOCP proposition but in the proof of the SOCP theorem? Don't think so. TBD.
%}

%\TODO{Is $\|\grad f(x)\| \leq \frac{1}{2}\ell b$ sufficient for the next lemma or do we need $\|\grad f(x)\| \leq 2\ell \mathscr{M}$? -- .5 l b is sufficient: I checked.}

The next two lemmas support Proposition~\ref{prop:Case2}.
Proofs are in Appendix~\ref{app:FOCP}.
They correspond to~\citep[Lem.~21 and ~22]{jin2018agdescapes}.
Notice that it is in Lemma~\ref{lem:focpB} that the condition on $\chi$ originates, then finds its way into the conditions of Theorem~\ref{thm:masterFOCP} through Proposition~\ref{prop:Case2}.
Ultimately, this causes the polylogarithmic factor in the complexity of Theorem~\ref{thm:masterFOCPexp}.
\begin{lemma} \label{lem:focpA}
	Fix parameters and assumptions as laid out in Section~\ref{sec:assuparams}.
	Let $x \in \calM$ satisfy $\|\grad f(x)\| \leq \frac{1}{2}\ell b$.
	Let $\calS$ denote the linear subspace of $\T_x\calM$ spanned by the eigenvectors of $\nabla^2 \hat f_x(0)$ associated to eigenvalues strictly larger than $\frac{\theta^2}{\eta(2-\theta)^2}$. %, and let $\calS^c$ be its orthogonal complement.
	Let $P_\calS$
	% and $P_{\calS^c}$
	denote orthogonal projection to $\calS$. % and $\calS^c$, respectively.
	Assume $\TSS(x)$
	% or $\TSS(x, s_0)$ (with $\|s_0\| \leq b$) -- this is fie, but not necessary
	runs its course in full.
	
	If there exists $\tau \in \{\mathscr{T}/4, \ldots, \mathscr{T}/2\}$ such that
	\begin{align*}
		\|s_\tau\| & \leq \mathscr{L}, & \|\nabla \hat f_x(s_\tau) - P_{\calS} \nabla \hat f_x(s_\tau)\| & \geq \epsilon / 6, \\
		\|v_\tau\| & \leq \mathscr{M}, \textrm{ and } & \inner{P_\calS v_\tau}{\nabla^2 \hat f_x(0)[P_\calS v_\tau]} & \leq \sqrt{\hat \rho \epsilon} \mathscr{M}^2,
	\end{align*}
	then $E_{\tau-1} - E_{\tau + \mathscr{T}/4} \geq \mathscr{E}$.
\end{lemma}
%\TODO{Is $\|\grad f(x)\| \leq 2\ell \mathscr{M}$ necessary for the next lemma or is $\|\grad f(x)\| \leq \frac{1}{2}\ell b$ sufficient? -- It's necessary to assume 2lM.}
\begin{lemma} \label{lem:focpB}
	Fix parameters and assumptions as laid out in Section~\ref{sec:assuparams}, with
	\begin{align*}
		\chi \geq \log_2(\theta^{-1}) \geq 1.
	\end{align*}
	Let $x \in \calM$ satisfy $\|\grad f(x)\| \leq 2\ell \mathscr{M}$. % !! is this necessary? if not, it's bewildering to change the displayed condition all of a sudden, compared to (1/2) \ell b.
	Let $\calS$ denote the linear subspace of $\T_x\calM$ spanned by the eigenvectors of $\nabla^2 \hat f_x(0)$ associated to eigenvalues strictly larger than $\frac{\theta^2}{\eta(2-\theta)^2}$.
	Let $P_\calS$ denote orthogonal projection to $\calS$.
	Assume $\TSS(x)$ runs its course in full.
	
	If $E_0 - E_{\mathscr{T}/2} \leq \mathscr{E}$, then for each $j$ in $\{ \mathscr{T}/4, \ldots, \mathscr{T}/2 \}$ we have
	\begin{align*}
		\|P_\calS \nabla \hat f_x(s_j)\| & \leq \epsilon / 6 & \textrm{ and } && \inner{P_\calS v_j}{\nabla^2 \hat f_x(0)[P_\calS v_j]} & \leq \sqrt{\hat\rho \epsilon} \mathscr{M}^2.
	\end{align*}
\end{lemma}

\add{
In Lemmas~\ref{lem:focpA} and~\ref{lem:focpB}, if $\calS$ is empty then $P_{\calS}$ maps all vectors to the zero vector, and the statements still hold.
}
%\TODO{check that this is true -- yep it is true: main reason is Lemma 5.4 almost exclusively works with S^c}

\section{Second-order critical points} \label{sec:socp}

As discussed in the previous section, $\ARGD$ produces $\epsilon$-approximate first-order critical points at an accelerated rate, deterministically.
Such a point might happen to be an approximate second-order critical point, or it might not.
In order to produce approximate second-order critical points, $\PARGD$ builds on top of $\ARGD$ as follows.

Whenever $\ARGD$ produces a point with gradient smaller than $\epsilon$, $\PARGD$ generates a random vector $\xi$ close to the origin in the current tangent space and runs $\TSS$ starting from that perturbation.
The run of $\TSS$ itself is deterministic.
However, the randomized initialization has the following effect: if the current point is not an approximate second-order critical point, then with high probability the sequence generated by $\TSS$ produces significant cost decrease.
Intuitively, this is because the current point is a saddle point, and gradient descent-type methods slowly but likely escape saddles.
If this happens, we simply proceed with the algorithm.
Otherwise, we can be reasonably confident that the point from which we ran the perturbed $\TSS$ is an approximate second-order critical point, and we terminate there.

\begin{algorithm}[t]
	\caption{$\PARGD(x_0)$ with $x_0 \in \calM$ and parameters $\epsilon, \ell, \eta, b, \theta, \gamma, s, r, \mathscr{T}, \mathscr{E}, \mathscr{M}$} % , T
	\label{algo:PARGD}
	\begin{algorithmic}[1]
		\State $t \gets 0$
		\While {\textbf{true}} % $t \leq T$
			\If {$\norm{\grad f(x_t)} > 2 \ell \mathscr{M}$}
				\State $x_{t+1} = \Retr_{x_t}(-\eta \grad f(x_t))$ \Comment{Case 1: one Riemannian gradient step}
				\State {$t \gets t+1$}
			\ElsIf {$\norm{\grad f(x_t)} > \epsilon$}
				\State  $x_{t+\mathscr{T}} = \TSS(x_t)$ \Comment{Case 2: accelerated gradient in $\T_{x_t}\calM$}
				\State $t \gets t + \mathscr{T}$
			\Else
				\State $\xi \sim \text{Uniform}(B_{x_t}(r))$ \Comment{Random perturbation}
%				\State $(x_{t+\mathscr{T}}, \isSOCP) \gets \TSS(x_t, \xi)$ \Comment{Case 3: accelerated gradient in $\T_{x_t}\calM$}
				\State $x_{t+\mathscr{T}} = \TSS(x_t, \xi)$ \Comment{Case 3: Perturbed accelerated gradient in $\T_{x_t}\calM$}
				\If {$f(x_t) - f(x_{t+\mathscr{T}}) < \frac{1}{2}\mathscr{E}$}
					\State \textbf{return} $x_t$ \Comment{Approximate FOCP, likely an approximate SOCP}
				\EndIf
				\State $t \gets t + \mathscr{T}$
			\EndIf
		\EndWhile
	\end{algorithmic}
\end{algorithm}

Our main result regarding $\PARGD$, namely, Theorem~\ref{thm:masterSOCP}, states that it computes approximate second-order critical points with high probability in a bounded number of iterations.
As worked out in Theorem~\ref{thm:masterSOCPexp}, this bound scales as $\epsilon^{-7/4}$, up to polylogarithmic terms which include a dependency in the dimension of the manifold and the probability of success.

Mirroring Section~\ref{sec:focp}, the proof of Theorem~\ref{thm:masterSOCP} rests on the two propositions of that section and on an additional proposition introduced hereafter in this section.
The latter proposition rests on a lemma introduced later still.
\begin{theorem} \label{thm:masterSOCP}
	Pick any $x_0 \in \calM$.
	Fix parameters and assumptions as laid out in Section~\ref{sec:assuparams},
	with $d = \dim \calM$, $\delta \in (0, 1)$, any $\Delta_f \geq \max\!\left(f(x_0) - \flow, \sqrt{\frac{\epsilon^3}{\hat{\rho}}}\right)$ and
	\begin{align*}
		\chi \geq \log_2\!\left( \frac{d^{1/2} \ell^{3/2} \Delta_f}{(\hat \rho \epsilon)^{1/4} \epsilon^2 \delta} \right) \geq \log_2(\theta^{-1}) \geq 1.
	\end{align*}
	The call to $\PARGD(x_0)$ returns $x_t \in \calM$ satisfying $f(x_t) \leq f(x_0)$, $\|\grad f(x_t)\| \leq \epsilon$ and (with probability at least $1-2\delta$) also $\lambdamin(\nabla^2 \hat f_{x_t}(0)) \geq -\sqrt{\hat\rho \epsilon}$ with
	\begin{align}
		t + \mathscr{T} \leq T_2 \triangleq \left( 2 + 4\frac{f(x_0) - \flow}{\mathscr{E}} \right) \mathscr{T}.
	\end{align}
	To reach termination, the algorithm requires at most $2T_2$ pullback gradient queries and $4T_2$ function queries (but no Hessian queries), and a similar number of calls to the retraction.
\end{theorem}
Notice how this result gives a (probabilistic) guarantee about the smallest eigenvalue of the Hessian of the pullback $\hat f_x$ at $0$ rather than about the Hessian of $f$ itself at $x$.
Owing to Lemma~\ref{lem:derivativespullback}, the two are equal in particular when we use the exponential retraction (more generally, when we use a \emph{second-order retraction}): see also~\citep[\S3.5]{boumal2016globalrates}.
\begin{proof}[Proof of Theorem~\ref{thm:masterSOCP}]
	The proof starts the same way as that of Theorem~\ref{thm:masterFOCP}.
	The call to $\PARGD(x_0)$ generates a sequence of points $x_{t_0}, x_{t_1}, x_{t_2}, \ldots$ on $\calM$, with $t_0 = 0$.
	A priori, this sequence may be finite or infinite.
	Considering two consecutive indices $t_i$ and $t_{i+1}$, we either have $t_{i+1} = t_i + 1$ (if the step from $x_{t_i}$ to $x_{t_{i+1}}$ is a single gradient step (Case 1)) or $t_{i+1} = t_i + \mathscr{T}$ (if that same step is obtained through a call to $\TSS$, with or without perturbation (Cases 3 and 2 respectively)). Moreover:
	\begin{itemize}
		\item In Case 1, Proposition~\ref{prop:Case1} applies and guarantees
		\begin{align*}
			f(x_{t_i}) - f(x_{t_{i+1}}) \geq \frac{\mathscr{E}}{\mathscr{T}} = \frac{\mathscr{E}}{\mathscr{T}}(t_{i+1} - t_i).
		\end{align*}
		The algorithm does not terminate here.
		
		\item In Case 2, Proposition~\ref{prop:Case2} applies and guarantees that if $\|\grad f(x_{t_{i+1}})\| > \epsilon$ then
		\begin{align*}
			f(x_{t_i}) - f(x_{t_{i+1}}) \geq \mathscr{E} = \frac{\mathscr{E}}{\mathscr{T}}(t_{i+1} - t_i),
		\end{align*}
		and the algorithm does not terminate here.
		
		If however $\|\grad f(x_{t_{i+1}})\| \leq \epsilon$, then $f(x_{t_i}) - f(x_{t_{i+1}}) \geq 0$ and the step from $x_{t_{i+1}}$ to $x_{t_{i+2}}$ does not fall in Case 2: it must fall in Case 3.
		(Indeed, it cannot fall in Case~1 because the fact that a Case~2 step occurred tells us $\epsilon < 2\ell\mathscr{M}$.)
		The algorithm terminates with $x_{t_{i+1}}$ unless $f(x_{t_{i+1}}) - f(x_{t_{i+2}}) \geq \frac{1}{2}\mathscr{E}$.
		In other words, if the algorithm does not terminate with $x_{t_{i+1}}$, then
		\begin{align*}
			f(x_{t_{i}}) - f(x_{t_{i+2}}) = f(x_{t_{i}}) - f(x_{t_{i+1}}) + f(x_{t_{i+1}}) - f(x_{t_{i+2}}) \geq \frac{1}{2}\mathscr{E} = \frac{\mathscr{E}}{4\mathscr{T}}(t_{i+2} - t_{i}).
		\end{align*}
%		For Case 1, we have $f(x_{t_{i+1}}) - f(x_{t_{i+2}}) \geq \frac{\mathscr{E}}{\mathscr{T}}$ as discussed above, hence
%		\begin{align*}
%			f(x_{t_{i}}) - f(x_{t_{i+2}}) = f(x_{t_{i}}) - f(x_{t_{i+1}}) + f(x_{t_{i+1}}) - f(x_{t_{i+2}}) \geq \frac{\mathscr{E}}{\mathscr{T}} = \frac{\mathscr{E}}{\mathscr{T}(\mathscr{T}+1)}(t_{i+2} - t_{i}).
%		\end{align*}
%		and the algorithm does not terminate with $x_{t_{i+1}}$. \TODO{This is not enough, and rightly so: the gradient step on its own doesn't achieve enough to compensate for the lack of achievement of the last $\mathscr{T}$ increments. We may need to tinker with the if-else if-else structure of $\PARGD$. Careful that Proposition~\ref{prop:Case3} now relies on the fact that it is only called if $\|\grad f(x)\| \leq \min(\epsilon, 2\ell\mathscr{M})$: that is currently true, but...}
		
		\item In Case 3, the algorithm terminates with $x_{t_{i}}$ unless
		\begin{align*}
			f(x_{t_{i}}) - f(x_{t_{i+1}}) \geq \frac{1}{2}\mathscr{E} = \frac{\mathscr{E}}{2\mathscr{T}}(t_{i+1} - t_{i}).
		\end{align*}
	\end{itemize}

%		we compute $x_{t_{i+1}} = \TSS(x_{t_i}, \xi)$ with a random $\xi$. %, producing $x_{t_{i+1}}$ and a boolean variable, $\isSOCP$.
%		There are two possibilities:
%		either $\lambdamin(\nabla^2 \hat f_{x_{t_i}}(0)) \geq -\sqrt{\hat\rho \epsilon}$, % (and we would be happy to return $x_{t_i}$),
%		or not.
%		In the latter event, Proposition~\ref{prop:Case3} applies, telling us that
%		\begin{align*}
%			f(x_{t_i}) - f(x_{t_{i+1}}) \geq \frac{1}{2}\mathscr{E} = \frac{\mathscr{E}}{2\mathscr{T}}(t_{i+1} - t_i)
%		\end{align*}
%		with probability at least $1 - \frac{\delta \mathscr{E}}{3\Delta_f}$, in which case the algorithm does not return $x_{t_i}$.
	
	Clearly, $\PARGD(x_0)$ must terminate after a finite number of steps.
	Indeed, if it does not, then the above reasoning shows that the algorithm produces an amortized decrease in the cost function $f$ of $\frac{\mathscr{E}}{4\mathscr{T}}$ per unit increment of the counter $t$, yet the value of $f$ cannot decrease by more than $f(x_0) - \flow$.
	
	Accordingly, assume $\PARGD(x_0)$ generates $x_{t_0}, \ldots, x_{t_{k+1}}$ and terminates there (returning $x_{t_k}$).
	The step from $x_{t_k}$ to $x_{t_{k+1}}$ necessarily falls in Case 3: $t_{k+1} - t_k = \mathscr{T}$. % and the value of $f$ may not decrease
	The step from $x_{t_{k-1}}$ to $x_{t_k}$ could be of any type.
	If it falls in Case 2, it could be that $f(x_{t_{k-1}}) - f(x_{t_k})$ is as small as zero, and that $t_{k} - t_{k-1} = \mathscr{T}$.
	(All other scenarios are better, in that the cost function decreases more, and the counter increases as much or less.)
	Moreover, for all steps prior to that, each unit increment of $t$ brings about an amortized decrease in $f$ of $\frac{\mathscr{E}}{4\mathscr{T}}$.
	Thus, $t_{k+1} \leq t_{k-1} + 2\mathscr{T}$ and
	\begin{align*}
		f(x_0) - \flow \geq f(x_{0}) - f(x_{t_{k-1}})
%			= \sum_{i = 0}^{k-2} f(x_{t_i}) - f(x_{t_{i+1}}) \geq \frac{\mathscr{E}}{4\mathscr{T}} \sum_{i = 0}^{k-2} t_{i+1} - t_i =
		\geq
		\frac{\mathscr{E}}{4\mathscr{T}} t_{k-1}.
	\end{align*}
	Combining, we find
	\begin{align*}
		t_k + \mathscr{T} = t_{k+1} \leq \left( 2 + 4\frac{f(x_0) - \flow}{\mathscr{E}} \right) \mathscr{T} \triangleq T_2.
	\end{align*}
	
	What can we say about the point that is returned, $x_{t_k}$?
	Deterministically, $f(x_{t_k}) \leq f(x_0)$ and $\|\grad f(x_{t_k})\| \leq \epsilon$ (notice that we cannot guarantee the same about $x_{t_{k+1}}$). %, which is why the algorithm returns $x_{t_k}$).% $f$ may increase from $x_{t_k}$ to $x_{t_{k+1}}$, but this is not an issue because $\PARGD$ returns $x_{t_k}$, not $x_{t_{k+1}}$).
	Let us now discuss the role of randomness.
	
	In any run of $\PARGD(x_0)$, there are at most $T_2/\mathscr{T}$ perturbations, that is, ``Case 3'' steps.
	By Proposition~\ref{prop:Case3}, the probability of any single one of those steps failing to prevent termination at a point where the smallest eigenvalue of the Hessian of the pullback at the origin is strictly less than $-\sqrt{\hat \rho \epsilon}$ is at most $\frac{\delta \mathscr{E}}{3\Delta_f}$.
	Thus, by a union bound, the probability of failure in any given run of $\PARGD(x_0)$ is at most
	(we use $\Delta_f \geq \max\!\left(f(x_0) - \flow, \sqrt{\frac{\epsilon^3}{\hat{\rho}}}\right) \geq \max\!\left( f(x_0) - \flow, 2^7\mathscr{E} \right)$ because $\chi \geq 1$ and $c \geq 2$):
	\begin{align*}
		\frac{T_2}{\mathscr{T}} \cdot \frac{\delta \mathscr{E}}{3\Delta_f} = \left( 2 + 4\frac{f(x_0) - \flow}{\mathscr{E}} \right) \frac{\delta \mathscr{E}}{3\Delta_f} \leq \left( \frac{2\mathscr{E}}{3\Delta_f} + \frac{4}{3}\right) \delta \leq 2\delta.
	\end{align*}
	In all other events, we have $\lambdamin(\nabla^2 \hat f_{x_{t_k}}(0)) \geq -\sqrt{\hat\rho \epsilon}$.
	
	For accounting of the maximal amount of work needed to run $\PARGD(x_0)$, use reasoning similar to that at the end of the proof of Theorem~\ref{thm:masterFOCP}, adding the cost of checking the condition ``$f(x_t) - f(x_{t+\mathscr{T}}) < \frac{1}{2}\mathscr{E}$'' after each perturbed call to $\TSS$.
	
	Note: the inequality $\frac{d^{1/2} \ell^{3/2} \sqrt{\epsilon^3 / \hat{\rho}}}{(\hat \rho \epsilon)^{1/4} \epsilon^2 \delta} \geq \theta^{-1}$ holds for all $d \geq 1$ and $\delta \in (0, 1)$ with $c \geq 4$.
\end{proof}
The next proposition corresponds mostly to~\citep[Lem.~13]{jin2018agdescapes}.
\begin{proposition}[Case 3] \label{prop:Case3}
	Fix parameters and assumptions as laid out in Section~\ref{sec:assuparams}, with $d = \dim \calM$, $\delta \in (0, 1)$, any $\Delta_f > 0$ and
	\begin{align*}
		\chi \geq \max\!\left(\log_2(\theta^{-1}), \log_2\!\left( \frac{d^{1/2} \ell^{3/2} \Delta_f}{(\hat \rho \epsilon)^{1/4} \epsilon^2 \delta} \right) \right) \geq 1.
	\end{align*}
	If $x \in \calM$ satisfies $\|\grad f(x)\| \leq \min(\epsilon, 2\ell\mathscr{M})$ and $\lambdamin(\nabla^2 \hat f_x(0)) \leq -\sqrt{\hat\rho \epsilon}$, and $\xi$ is sampled uniformly at random from the ball of radius $r$ around the origin in $\T_x\calM$,
	then $x_{\mathscr{T}} = \TSS(x, \xi)$ satisfies $f(x) - f(x_{\mathscr{T}}) \geq \mathscr{E}/2$ with probability at least $1 - \frac{\delta \mathscr{E}}{3\Delta_f}$ over the choice of $\xi$.
\end{proposition}
\begin{proof}[Proof of Proposition~\ref{prop:Case3}]
%	Using $\sqrt{\hat \rho \epsilon} \leq \frac{1}{2}\ell$ and $\epsilon \leq b^2 \hat \rho$ and by \aref{assu:epsilon}, we have
%	\begin{align*}
%		\|\grad f(x)\| \leq \epsilon = \sqrt{\epsilon} \cdot \sqrt{\epsilon} \leq \frac{1}{2}\ell \frac{1}{\sqrt{\hat\rho}} \cdot b \sqrt{\hat\rho} = \frac{1}{2} \ell b.
%		 %Simpler: "\|\grad f(x)\| \leq \epsilon \leq 2\ell\mathscr{M} < \frac{1}{2} \ell b" -- but this explicitly uses \epsilon \leq 2\ell\mathscr{M}, which we'd like to avoid here.
%	\end{align*}
%	Thus, the Lipschitz assumptions of~\aref{assu:Mandf} apply at $x$.
	By Lemma~\ref{lem:params}, $\|\grad f(x)\| \leq 2\ell\mathscr{M} < \frac{1}{2} \ell b$ and $\|\xi\| \leq r < b$.
	Thus, the strongest provisions of~\aref{assu:Mandf} apply at $x$, as do Lemmas~\ref{lem:TSSballs}, \ref{lem:TSSEj}, \ref{lem:TSStraveldist} and~\ref{lem:TSSNCE}.
	Let $u_j, s_j, v_j$ for $j = 0, 1, \ldots$ be the vectors generated by the computation of $x_{\mathscr{T}} = \TSS(x, \xi)$.
	Note that $s_0 = \xi$ and $v_0 = 0$.
	Owing to how $\TSS$ works, there are several cases to consider, based on how it terminates.
	We remark that cases 3a and 3b are deterministic (they only use the fact that $\|s_0\| \leq r$), % yes, by r, not by b: see details of the proof
	that there is no case 3c, and that case 3d is the only place where probabilities are involved.
	Throughout, it is useful to observe that, since $f(x) = \hat f_x(0)$,  $\|\grad f(x)\| \leq \epsilon$ and $\grad f(x) = \nabla \hat f_x(0)$, the first property of~\aref{assu:Mandf} ensures:
	\begin{align}
		\hat f_x(s_0) - f(x) \leq \inner{\grad f(x)}{s_0} + \frac{\ell}{2} \|s_0\|^2 \leq \epsilon r + \frac{\ell}{2} r^2 \leq \frac{1}{4} \mathscr{E}.
		\label{eq:hatfxszerominfxbound}
	\end{align}
	(Use Lemma~\ref{lem:params} to relate parameters.)
	Compare details below with Proposition~\ref{prop:Case2}. % (whose proof has more details).
	\begin{itemize}
		\item (Case 3a) The negative curvature condition \eqref{eq:NCC} triggers with $(x, s_j, u_j)$. % for some $j < \mathscr{T}$.
%		There are two cases to check.
		Either $\|s_j\| \leq \mathscr{L}$, in which case Lemma~\ref{lem:TSSEj} tells us $E_j \leq E_0 = \hat f_x(s_0)$ and, by Lemma~\ref{lem:TSSNCE},
		\begin{align*}
			f(x_{\mathscr{T}}) = \hat f_x(\NCE(x, s_j, v_j)) \leq E_j - 2\mathscr{E} \leq f(x) - 2\mathscr{E} + \hat f_x(s_0) - f(x).
		\end{align*}
		Or $\|s_j\| > \mathscr{L}$, in which case Lemma~\ref{lem:TSStraveldist} used with $q = j < \mathscr{T}$ and $\|s_q - s_0\| \geq \|s_q\| - \|s_0\| \geq \mathscr{L} - r \geq \frac{63}{64}\mathscr{L}$ implies
		\begin{align*}
			f(x_{\mathscr{T}}) \leq E_j \leq \hat f_x(s_0) - \frac{63^2}{64^2} \frac{\mathscr{L}^2}{16 \sqrt{\kappa} \eta \mathscr{T}} = f(x) - \frac{63^2}{64^2} \mathscr{E} + \hat f_x(s_0) - f(x).
		\end{align*}
		(We used Lemma~\ref{lem:params} to relate parameters.)
%		Owing to how $\NCE$ works, we always have $f(x_{\mathscr{T}}) = \hat f_x(\NCE(x, s_j, v_j)) \leq \hat f_x(s_j) \leq E_j$ (the last inequality is by definition of $E_j$~\eqref{eq:Ej}).
%		Therefore, $f(x_{\mathscr{T}}) \leq f(x) - \frac{63^2}{64^2} \mathscr{E} + \hat f_x(s_0) - f(x)$.
		Either way, bound $\hat f_x(s_0) - f(x)$ with~\eqref{eq:hatfxszerominfxbound}.
%		Requiring $c \geq 2$, we have:
%		By~\aref{assu:Mandf}, since $f(x) = \hat f_x(0)$ we have (using $\|\grad f(x)\| \leq \epsilon$ and again Lemma~\ref{lem:params}): % we really need to use that: if grad is bigger than epsilon, we can't quite say that the difference will be small in proportion to scriptE, which itself depends on epsilon.
%		\begin{align}
%			\hat f_x(s_0) - f(x) \leq \inner{\grad f(x)}{s_0} + \frac{\ell}{2} \|s_0\|^2 \leq \epsilon r + \frac{\ell}{2} r^2 \leq \frac{1}{4} \mathscr{E}.
%			\label{eq:hatfxszerominfxbound}
%		\end{align}
		Overall, we conclude that $f(x) - f(x_{\mathscr{T}}) \geq \frac{1}{2} \mathscr{E}$ (deterministically).
		
		\item (Case 3b) The iterate $s_{j+1}$ leaves the ball of radius $b$, that is, $\|s_{j+1}\| > b$. % for some $j < \mathscr{T}$.
		In this case, apply Lemma~\ref{lem:TSStraveldist} with $q = j + 1 \leq \mathscr{T}$ and
		\begin{align*}
			 \|s_{j+1} - s_0\| \geq \|s_{j+1}\| - \|s_0\| \geq b - r \geq 4\mathscr{L} - \frac{1}{64}\mathscr{L} \geq \mathscr{L}
		\end{align*}
		to claim (as always, we use Lemma~\ref{lem:params} repeatedly to relate parameters)
		\begin{align*}
			f(x_{\mathscr{T}}) = \hat f_x(s_{j+1}) \leq E_{j+1} \leq \hat f_x(s_0) - \frac{\|s_{j+1} - s_0\|^2}{16 \sqrt{\kappa} \eta \mathscr{T}} \leq \hat f_x(s_0) - \frac{\mathscr{L}^2}{16 \sqrt{\kappa} \eta \mathscr{T}} = \hat f_x(s_0) - \mathscr{E}.
		\end{align*}
		By~\eqref{eq:hatfxszerominfxbound}, it follows that $f(x) - f(x_{\mathscr{T}}) \geq \frac{3}{4} \mathscr{E}$ (deterministically).
		
		\item (Case 3d) None of the other events occur: $\TSS(x, s_0)$ runs its $\mathscr{T}$ iterations in full. %, with $\theta_j = \theta$ for all $j$.
		In this case, we apply the logic in the proof of~\citep[Lem.~13]{jin2018agdescapes}, as follows.
		Define the set $\Xstuck_x$ as containing exactly all tangent vectors $s^* \in B_x(r)$ such that
		\begin{enumerate}
			\item $\TSS(x, s^*)$ runs its $\mathscr{T}$ iterations in full, and
			\item $E_0^* - E_{\mathscr{T}}^* \leq 2\mathscr{E}$, where $E_j^*$ denotes the Hamiltonians associated to $\TSS(x, s^*)$.
		\end{enumerate}
		There are two cases.
		Either $s_0$ is not in $\Xstuck_x$, in which case $E_0 - E_{\mathscr{T}} > 2\mathscr{E}$:
		it is then easy to conclude (using~\eqref{eq:hatfxszerominfxbound}) that $f(x) - f(x_{\mathscr{T}}) > \frac{7}{4} \mathscr{E}$.
		Or $s_0$ is in $\Xstuck_x$, in which case we do not lower-bound $f(x) - f(x_{\mathscr{T}})$.
		The probability of this happening is
		\begin{align*}
			\Prob{\xi \in \Xstuck_x} & = \frac{\Vol{\Xstuck_x}}{\Vol{\mathbb{B}^d_r}},
		\end{align*}
		where $\Vol{\cdot}$ denotes the volume of a set, and $\Vol{\mathbb{B}^d_r}$ is the volume of a Euclidean ball of radius $r$ in a $d$-dimensional vector space.
		In order to upper-bound the volume of $\Xstuck_x$, we resort to Lemma~\ref{lem:socp}: this is where we use the assumption $\lambdamin(\nabla^2 \hat f_x(0)) \leq -\sqrt{\hat\rho \epsilon}$.
		
		Let $e_1$ denote an eigenvector of $\nabla^2 \hat f_x(0)$ with minimal eigenvalue, and let $s_0, s_0'$ be two arbitrary vectors in $\Xstuck_x$ such that $s_0 - s_0'$ is parallel to $e_1$.
		Lemma~\ref{lem:socp} implies that $\|s_0 - s_0'\| \leq \frac{\delta \mathscr{E}}{2\Delta_f} \frac{r}{\sqrt{d}}$.
		Now consider a point $a \in B_x(r)$ orthogonal to $e_1$, and let $\ell_a$ denote the line parallel to $e_1$ passing through $a$.
		The previous reasoning tells us that the intersection of $\ell_a$ with $\Xstuck_x$ is contained in a segment of $\ell_a$ of length at most $\frac{\delta \mathscr{E}}{2\Delta_f} \frac{r}{\sqrt{d}}$.
		Thus, with $\one$ denoting the indicator function,
		\begin{align*}
			\Vol{\Xstuck_x} & = \int_{B_x(r)} \one_{\Xstuck_x}(y) \dy \\
							& = \int_{a \in B_x(r) : a \perp e_1} \left[ \int_{\ell_a} \one_{\Xstuck_x}(z) \dz \right] \da \\
							& \leq \frac{\delta \mathscr{E}}{2\Delta_f} \frac{r}{\sqrt{d}} \Vol{\mathbb{B}^{d-1}_r}.
		\end{align*}
		With $\Gamma$ denoting the Gamma function, it follows that
		\begin{align*}
			\Prob{\xi \in \Xstuck_x} & \leq \frac{\delta \mathscr{E}}{2\Delta_f} \frac{r}{\sqrt{d}} \cdot \frac{\Vol{\mathbb{B}^{d-1}_r}}{\Vol{\mathbb{B}^d_r}} = \frac{\delta \mathscr{E}}{2\Delta_f} \frac{r}{\sqrt{d}} \cdot \frac{1}{r\sqrt{\pi}} \frac{\Gamma(1 + d/2)}{\Gamma(1 + (d-1)/2)}.
		\end{align*}
		% d = 1:100; plot(d, gamma(1+d/2)./gamma(1+(d-1)/2), d, sqrt(d))
		One can check (for example, using Gautschi's inequality) that the last fraction is upper-bounded by $\sqrt{d}$ for all $d \geq 1$.
		Thus,
		\begin{align*}
			\Prob{\xi \in \Xstuck_x} & \leq \frac{\delta \mathscr{E}}{2\sqrt{\pi}\Delta_f} \leq \frac{\delta \mathscr{E}}{3\Delta_f}.
		\end{align*}
		This limits the probability of the only bad event.
	\end{itemize}
This covers all possibilities.
\end{proof}

Mirroring Section~\ref{sec:focp}, the following final lemma supports Proposition~\ref{prop:Case3}.
The proof is in Appendix~\ref{app:SOCP}.
It corresponds to~\citep[Lem.~23]{jin2018agdescapes}.
The condition on $\chi$ originates in this lemma, and from here appears in Theorem~\ref{thm:masterSOCP} through Proposition~\ref{prop:Case3}.
It causes the occurrence of dimension in the polylogarithmic factor in the complexity of Theorem~\ref{thm:masterSOCPexp}, but note that the real reason why $d$ appears in the condition on $\chi$ here is so that dimension can be canceled out in the probabilistic argument in the proof of Proposition~\ref{prop:Case3}.
\begin{lemma} \label{lem:socp}
	Fix parameters and assumptions as laid out in Section~\ref{sec:assuparams}, with $d = \dim \calM$, $\delta \in (0, 1)$, any $\Delta_f > 0$ and
	\begin{align*}
		\chi \geq \max\!\left(\log_2(\theta^{-1}), \log_2\!\left( \frac{d^{1/2} \ell^{3/2} \Delta_f}{(\hat \rho \epsilon)^{1/4} \epsilon^2 \delta} \right) \right) \geq 1.
	\end{align*}
%	Let $\calS$ denote the linear subspace of $\T_x\calM$ spanned by the eigenvectors of $\nabla^2 \hat f_x(0)$ associated to eigenvalues larger than or equal to $\frac{\theta^2}{\eta(2-\theta)^2}$.
%	Let $P_\calS$ denote orthogonal projection to $\calS$.
	Let $s_0, s_0' \in B_x(r)$ be such that
	\begin{enumerate}
		\item $s_0 - s_0' = r_0 e_1$ where $e_1$ is an eigenvector of $\nabla^2 \hat f_x(0)$ associated to the smallest eigenvalue and $r_0 \geq \frac{\delta \mathscr{E}}{2\Delta_f} \frac{r}{\sqrt{d}}$, and
		\item $\TSS(x, s_0)$ and $\TSS(x, s_0')$ both run their $\mathscr{T}$ iterations in full, respectively generating vectors $u_j, s_j, v_j$ and $u_j', s_j', v_j'$, with corresponding Hamiltonians $E_j, E_j'$.
	\end{enumerate}
	If $\|\grad f(x)\| \leq \frac{1}{2} \ell b$ and $\lambdamin(\nabla^2 \hat f_x(0)) \leq -\sqrt{\hat\rho \epsilon}$, then $\max\!\left( E_0 - E_{\mathscr{T}}, E_0' - E_{\mathscr{T}}' \right) \geq 2\mathscr{E}$.
%	\begin{align*}
%		\max\!\left( E_0 - E_{\mathscr{T}}, E_0' - E_{\mathscr{T}}' \right) & \geq 2\mathscr{E}.
%	\end{align*}
\end{lemma}

%\TODO{Do we need a version of Lemma~\ref{lem:focpB} that applies to $\TSS(x, s_0)$ too? i think not, though we might need it in the proof of the equivalent of Lemma 5.21 (just above) -- no, we do not..}

\section{Discussion of the main results} \label{sec:discussion}

In this section, we discuss finer points of our main theorems and their construction. %, and we relate them to the literature.

\subsection*{About geometric results}

Our main geometric result, Theorem~\ref{thm:pullbacklipschitz}, departs from what one might ideally hope for in three ways: (a) it applies only at points where $\grad f$ is sufficiently small; (b) it does not provide full Lipschitzness for the Hessian: only a Lipschitz-like condition with respect to the origin of the tangent space; and (c) its conclusions are restricted to balls of some radius $b$. Here, we discuss these limitations.

On compact manifolds, the restrictions can be partly but not fully relaxed.
For example, consider the unit sphere $\Sn = \{ x \in \Rn : x\transpose x = 1\}$ as a Riemannian submanifold of $\Rn$ with the usual Euclidean metric ($\Klow = \Kup = K = 1$, $F = 0$).
Let $f \colon \Sn \to \reals$ have $L$-Lipschitz continuous gradient.
The pullback of $f$ through the exponential map at $x$ is denoted $\hat f_x = f \circ \Exp_x$.
%Through direct computation and bounding of $\nabla^2 \hat f_x$, it is easy to see
We show in Proposition~\ref{prop:lipschitzpullbacksphere} that $\nabla \hat f_x$ is $\frac{5}{2}L$-Lipschitz continuous on the whole tangent space, for all $x$.
If moreover the Hessian of $f$ is $\rho$-Lipschitz continuous, then $\|\nabla^2 \hat f_x(s) - \nabla^2 \hat f_x(0)\|_x \leq \hat\rho \|s\|_x$ with $\hat \rho = \rho + 3.1 \cdot L$ for all $x$ and $s$.
This secures the benefits of items 1.\ and 2.\ of Theorem~\ref{thm:pullbacklipschitz} with fewer restrictions.
However, for item~3.\  and still on the sphere, we do need a restriction to balls of some finite radius.
Indeed, the smallest singular value of $T_s = \D\Exp_x(s)$ drops from one to zero as $\|s\|_x$ increases from zero to $\pi$.
As we run an optimization algorithm in a tangent space, one aim is to find an approximate critical point $s$ of $\hat f_x$ which maps to an approximate critical point $y = \Exp_x(s)$ of $f$.
Since the norm of $\grad f(y)$ could be as large as $\|\nabla \hat f_x(s)\|_x / \sigmamin(T_s)$ by~\eqref{eq:gradhesspullback}, we must restrict $\|s\|_x$ to retain control.
In general, if $\Kup$ is positive, this last consideration forces us to consider only balls of some radius bounded in proportion to $1/\sqrt{\Kup}$.

In contrast, consider the hyperbolic manifold $\calM = \{ x \in \Rn : x_2^2 + \cdots + x_n^2 = x_1^2 - 1 \}$ with the Riemannian metric defined by restriction of the Minkowski semi-inner product $\inner{u}{v} = u_2v_2 + \cdots + u_nv_n - u_1v_1$ to the tangent spaces.
For this non-compact manifold, we have $\Klow = \Kup = -1$, hence, $K = 1$ and $F = 0$.
Owing to $\Klow \leq 0$, the singular values of $T_s$ are all at least one, for all $x$ and $s$.
Thus, securing item~3.\ in Theorem~\ref{thm:pullbacklipschitz} requires no particular restrictions.
However, as we show in Proposition~\ref{prop:lipschitzpullbackhyperbolic}, as soon as $f$ is non-constant,
%(and twice differentiable),
we cannot hope to find a finite $\ell \geq 0$ such that all pullbacks have $\ell$-Lipschitz gradient globally.

These considerations on the sphere and on the hyperbolic space suggest that some of the restrictions in Theorem~\ref{thm:pullbacklipschitz} are indeed necessary.
Since for both examples we have $F = 0$, we cannot conclude as to the necessity of the assumption regarding the covariant derivative of the Riemannan curvature endomorphism.
We suspect it is necessary.
Moreover, we suspect that by assuming a bound on the second covariant derivative of curvature it may be possible to improve item 2.\ in Theorem~\ref{thm:pullbacklipschitz} to offer a bound on $\|\nabla^2 \hat f_x(s_1) - \nabla^2 \hat f_x(s_2)\|_x$ (that is, full Lipschitz-continuity of pullback Hessians, in appropriate balls at appropriate points $x$).
% It's reasonable, because for grad we can also get a result of the type $\nabla \hat f_x(s) - \nabla \hat f_x(0)$ smaller than $\hat L \|s\|_x$ without assumption on $\nabla R$ (right?); so it's a story where you can go from ``Lipschitz-type'' to ``Lipschitz'' and to higher and higher derivatives by increasing control over derivatives of $R$.

\subsection*{About optimization results}
%The importance of computing approximate second-order critical points is highlighted by a number of recent works on non-convex optimization: see for example results cited in the references of the next paragraph.

In different papers, \citet{jin2017howtoescape,jin2019escape} also explore non-accelerated perturbed gradient descent for the purpose of finding approximate second-order critical points in the Euclidean case (in line with a number of other papers, e.g., \citep{agarwal2017finding}).
That work was extended to the Riemannian case by~\citet{sun2019prgd} and also by ourselves~\citep{criscitiello2019escapingsaddles} using different techniques.
The assumptions made in the latter left the role of curvature unclear: this role is now elucidated by Theorem~\ref{thm:pullbacklipschitz}.

Though $\PARGD$ is a generalization of Jin et al.'s PAGD in spirit, it does not reduce to PAGD when $\calM$ is a Euclidean space.
One of the reasons is that $\PARGD$ resets the momentum after each random perturbation but PAGD does not.
Our $\NCE$ procedure also works slightly differently.
At a more philosophical level, Jin et al.\ emphasize the single-loop aspect of PAGD, which we lose by working on a sequence of tangent spaces.

We only need \aref{assu:Mandf} to hold at the points generated by the algorithm.
Since $f$ decreases monotonically along (outer) iterations, these points remain in the sublevel set of $x_0$.
Stated differently: the value of $f$ can go up and down inside of $\TSS$, but not in the outer loops of $\ARGD, \PARGD$.
Thus, it is possible to relax \aref{assu:Mandf} somewhat, for example assuming the sublevel sets of $f$ are compact.
In addition, property 4 of \aref{assu:Mandf} is used only once, namely, in Proposition~\ref{prop:Case1}: it could also be relaxed in several ways.

\remove{
Notice that if $\epsilon$ is larger than $2\ell\mathscr{M}$ (that is, if $\epsilon > \frac{16 \ell^2}{c^4  \hat\rho}$), then $\ARGD$ reduces to vanilla Riemannian gradient descent with constant step-size.
The latter is known to produce an $\epsilon$-FOCP in $O(1/\epsilon^2)$ iterations, yet our result here announces this same outcome in $O(1/\epsilon^{7/4})$ iterations.
This is not a contradiction: when $\epsilon$ is large, $1/\epsilon^{7/4}$ can be worse than $1/\epsilon^2$.
In short: the rates are only meaningful for small $\epsilon$, in which case $\ARGD$ does use accelerated gradient descent steps.
}

\add{
Consider optimization in a Euclidean space $\reals^d$ under an equality constraint $h(x) = 0, h \colon \reals^d \rightarrow \reals^m$.
If $\calM = \{x \in \reals^d : h(x) = 0\}$ defines a smooth embedded submanifold of $\reals^d$,
%(see~\citep[Sec.~3.2]{boumal2020intromanifolds} for conditions under which this is true),
then we can consider applying our results to this optimization problem.
The requirement that the sectional curvatures at a point $x \in \calM$ are bounded by $K$ and $\nabla R$ is bounded by $F$ is a local condition on the regularity of $h$ and its higher-order derivatives---see for example~\citep{tamasrap} for an expression of sectional curvatures in terms of the gradient and Hessian of $h$.
By phrasing our results in terms of bounds on the curvature, there is the added benefit that these regularity conditions on $h$ are intrinsic rather than extrinsic.
}

\add{
In passing, we note the similarity of $\ARGD$ and $\PARGD$ with the Riemmanian trust-region method (RTR)~\citep{boumal2016globalrates,boumal2020intromanifolds}.
For example, we can view $\ARGD$ as a combination of gradient steps and subproblem steps.
Like RTR, each subproblem of $\ARGD$ consists of approximately minimizing a model function in a ball of finite radius in a fixed tangent space.
In RTR, each subproblem is usually minimized via the truncated conjugate gradient method, which can be viewed as a type of momentum method.
In $\ARGD$, each subproblem is minimized with a modification of AGD, another type of momentum method.
}

%\section{Particular cases of interest}
%
%\TODO{Specialize results to the model spaces? For products of the model spaces, O'Neill p221 says $F = 0$ still. Product of spheres should apply to Max-Cut SDP in BM format.}

\section{Conclusions and perspectives} \label{sec:conclusions}

Our main complexity results for $\ARGD$ and $\PARGD$ (Theorems~\ref{thm:masterFOCPexp} and~\ref{thm:masterSOCPexp}) recover known Euclidean results when $\calM$ is a Euclidean space.
In particular, they retain the important properties of scaling essentially with $\epsilon^{-7/4}$ and of being either dimension free (for $\ARGD$) or almost dimension free (for $\PARGD$).
Those properties extend as is to the Riemannian case.

However, our Riemannian results are negatively impacted by the Riemannian curvature of $\calM$, and also by the covariant derivative of the Riemann curvature endomorphism.
We do not know whether such a dependency on curvature is necessary to achieve acceleration.
In particular, the non-accelerated rates for Riemannian gradient descent, Riemannian trust-regions and Riemannian adaptive regularization with cubics under Lipschitz assumptions do not suffer from curvature~\citep{boumal2016globalrates,agarwal2018arcfirst}.

Curvature enters our complexity bounds through our geometric results (Theorem~\ref{thm:pullbacklipschitz}).
For the latter, we do believe that curvature must play a role.
Thus, it is natural to ask:
\begin{quote}
	\emph{Can we achieve acceleration for first-order methods on Riemannian manifolds with weaker (or without) dependency on the curvature of the manifold?}
\end{quote}
For the geodesically convex case, all algorithms we know of are affected by curvature~\citep{zhang2018estimatesequence,alimisis2019continuoustime,ahn2020nesterovs,alimisis2020practical}. 
\add{Additionally, \citet{hamilton2021nogo} show that curvature can significantly slow down convergence rates in the geodesically convex case with noisy gradients.}

Adaptive regularization with cubics (ARC) may offer insights in that regard.
ARC is a cubically-regularized approximate Newton method with optimal iteration complexity on the class of cost functions with Lipschitz continuous Hessian, assuming access to gradients and Hessians~\citep{nesterov2006cubic,cartis2011adaptivecubic}.
Specifically, assuming $f$ has $\rho$-Lipschitz continuous Hessian, ARC finds an $(\epsilon, \sqrt{\rho \epsilon})$-approximate second-order critical point in at most $\tilde O(\Delta_f \rho^{1/2} / \epsilon^{3/2})$ iterations, omitting logarithmic factors.
This also holds on complete Riemannian manifolds~\citep[Cor.~3, eqs~(16),~(26)]{agarwal2018arcfirst}.
Note that this is dimension free \emph{and} curvature free.
Each iteration, however, requires solving a separate subproblem more costly than a gradient evaluation.
\citet[\S3]{carmon2018analysis} argue that it is possible to solve the subproblems accurately enough so as to find $\epsilon$-approximate first-order critical points with $\sim 1/\epsilon^{7/4}$ Hessian-vector products overall, with randomization and a logarithmic dependency in dimension.
Compared to $\ARGD$, this has the benefit of being curvature free, at the cost of randomization, a logarithmic dimension dependency, and of requiring Hessian-vector products.
The latter could conceivably be approximated with finite differences of the gradients.
Perhaps that operation leads to losses tied to curvature?
If not, as it is unclear why there ought to be a trade-off between curvature dependency and randomization, this may be the indication that the curvature dependency is not necessary for acceleration.

On a distinct note and as pointed out in the introduction, $\ARGD$ and $\PARGD$ are theoretical constructs.
Despite having the theoretical upper-hand in worst-case scenarios, we do not expect them to be competitive against time-tested algorithms such as Riemannian versions of nonlinear conjugate gradients or the trust-region methods.
It remains an interesting open problem to devise a truly practical accelerated first-order method on manifolds.

%\TODO{Mention that Jin et al.\ simplified some of their proofs for the perturbed algorithm later on, and that one could try to get the same results using those simplified proofs too.}

In the Euclidean case, \citet{carmon2017convexguilty} showed that if one assumes not only the gradient and the Hessian of $f$ but also the third derivative of $f$ are Lipschitz continuous, then it is possible to find $\epsilon$-approximate first-order critical points in just $\tilde O(\epsilon^{-5/3})$ iterations. We suspect that our proof technique could be used to prove a similar result on manifolds, possibly at the cost of also assuming a bound on the second covariant derivative of the Riemann curvature endomorphism.

\clearpage
\appendix

\section{Parallel transport vs differential of exponential map} \label{app:PTvDExp}

In this section, we give a proof for Proposition~\ref{prop:TsminPsparticular} regarding the difference between parallel transport along a geodesic and the differential of the exponential map.
We use these families of functions parameterized by $\Klow \in \reals$:
\begin{align}
	h_\Klow(t) & = \begin{cases}
			t & \textrm{ if } \Klow = 0, \\
			r \sin(t/r) & \textrm{ if } \Klow = 1/r^2 > 0, \\
			r \sinh(t/r) & \textrm{ if } \Klow = -1/r^2 < 0.
			\end{cases}
		\label{eq:hklow}
\end{align}
\begin{align}
	g_{\Klow}(t) & = \int_{0}^{t} h_{\Klow}(\tau) \dtau = \begin{cases}
			\frac{t^2}{2} & \textrm{ if } \Klow = 0, \\
			r^2\left( 1 - \cos(t/r) \right) & \textrm{ if } \Klow = 1/r^2 > 0, \\
			r^2\left( \cosh(t/r) - 1 \right) & \textrm{ if } \Klow = -1/r^2 < 0.
		\end{cases}
		\label{eq:gklow}
\end{align}
\begin{align}
	f_{\Klow}(t) & = \frac{1}{t} \int_{0}^{t} g_{\Klow}(\tau) \dtau = \begin{cases}
			\frac{t^2}{6} & \textrm{ if } \Klow = 0, \\
			r^2\left( 1 - \frac{\sin(t/r)}{t/r} \right) & \textrm{ if } \Klow = 1/r^2 > 0, \\
			r^2\left( \frac{\sinh(t/r)}{t/r} - 1 \right) & \textrm{ if } \Klow = -1/r^2 < 0.
		\end{cases}
		\label{eq:fklow}
\end{align}
Under the assumptions we make below, these functions are only ever evaluated at points where they are nonnegative.
In all cases, functions are dominated by the case $\Klow < 0$; formally, for all $\Klow \in \reals$, all $K \geq |\Klow|$ and all $t \geq 0$:
% plot {t, sin(t), sinh(t)} for t in [-7, 7]  -- for t >= 0
% plot {t^2/2, 1-cos(t), cosh(t)-1} for t in [-7, 7] -- for all t
% plot {t^2/6, 1-sin(t)/t, sinh(t)/t-1} for t in [-7, 7] -- for all t
\begin{align}
	h_\Klow(t) & \leq h_{-K}(t), & g_\Klow(t) & \leq g_{-K}(t), & f_\Klow(t) & \leq f_{-K}(t).
\end{align}
If $\Klow \geq 0$ and $t \geq 0$, then
\begin{align}
	h_\Klow(t) & \leq t, & g_\Klow(t) & \leq \frac{1}{2} t^2, & f_\Klow(t) & \leq \frac{1}{6} t^2.
\end{align}
Independently of the sign of $\Klow$, if $0 \leq t \leq \pi / \sqrt{|\Klow|}$, then
\begin{align*}
	h_\Klow(t) & \leq t + 0.2712 \cdot \Klow t^3 \leq 3.6761 \cdot t, & g_\Klow(t) & \leq 1.0732 \cdot t^2, &  f_\Klow(t) & \leq 0.2712 \cdot t^2.
	% I re-checked on Feb 13, 2020: the bound is .5 t^2 for Klow >= 0, and it's (cosh(pi)-1)/pi^2  *  t^2  <= 1.07.... t^2 for Klow < 0.
	% Do it by dividing the inequality by t^2 then change of variable t/r = x, and plot for x in [0, pi]. In the third case, you get a growing function of x, starting at 0.5 for x = 0 then growing with x.
	%
	% for f: (sinh(pi)/pi - 1)/pi^2
%	\label{eq:boundsfghKlowgeneral}
\end{align*}
For $t$ bounded as indicated, this last line shows that up to constants the sign of $\Klow$ does not substantially affect bounds.

To state our result, we need the notion of conjugate points along geodesics on Riemannian manifolds.
The following definition is equivalent to the standard one~\citep[Prop.~10.20 and p303]{lee2018riemannian}.
We are particularly interested in situations where there are no conjugate points on some interval: we discuss that event in a remark.
\begin{definition}
	Let $\calM$ be a Riemannian manifold.
	Consider $(x, s) \in \T\calM$ and the geodesic $\gamma(t) = \Exp_x(ts)$ defined on an open interval $I$ around zero.
	For $t \in I$, we say $\gamma(t)$ is \emph{conjugate to $x$ along $\gamma$} if $\D\Exp_x(ts)$ is rank deficient.
	We say $\gamma$ \emph{has an interior conjugate point on $[0, \bar t] \subset I$} if $\gamma(t)$ is conjugate to $x$ along $\gamma$ for some $t \in (0, \bar t)$.
\end{definition}
\begin{remark} \label{rem:conjugatepoints}
	Let $\gamma$ be a geodesic on a Riemannian manifold $\calM$.
	If $\gamma$ is minimizing on the interval $[0, \bar t]$, then it has no interior conjugate point on that interval~\citep[Thm.~10.26]{lee2018riemannian}.
	Assume the sectional curvatures of $\calM$ are in the interval $[\Klow, \Kup]$. Then:
	\begin{enumerate}
		\item If $\Kup \leq 0$, $\gamma$ has no conjugate points at all~\citep[Pb.~10-7]{lee2018riemannian};
		\item If $\Kup > 0$, $\gamma$ has no interior conjugate points on $[0, \pi / \sqrt{\Kup}]$~\citep[Thm.~11.9a]{lee2018riemannian}; and
		\item If $\Klow > 0$ and $\gamma$ has no interior conjugate point on $[0, \bar t]$, then $\bar t \leq \pi / \sqrt{\Klow}$~\citep[p298 and Thm.~11.9b]{lee2018riemannian}. This will be why, under our assumptions, $h_\Klow$~\eqref{eq:hklow} is only ever evaluated at points where it is nonnegative.
	\end{enumerate}
\end{remark}

%\TODO{See ``Tripuraneni Flammarion Bach Jordan - Averaging stochastic gradient descent on Riemannian manifolds - 2018'' Lemma 6: they talk exactly about this, where $\Lambda$ is the differentiated retraction and $\Gamma$ is the parallel transport. They do all of this along retractions. They work with O(.) notation though: not explicit at all about constants. They reference Waldmann 2012 Thm. A.2.9 which recovers the -1/6 times curvature. The important point is that Waldmann works out a (high degree) Taylor expansion of $\D\Exp_x(s)$ around $s = 0$; but we will be interested in $s \neq 0$ in the next section.}
We now state and prove the main result of this section.
A similar result appears in~\citep[Lem.~6]{tripuraneni2018averagingriemannian} for general retractions.
Constants there are not explicit (they are absorbed in $O(\cdot)$ notation).
Their proof is based on Taylor expansions of the differential of the exponential map as they appear in~\citep[Thm.~A.2.9]{waldmann2012geometric}, namely, for $s \mapsto \D\Exp_x(s)$ around $s = 0$.
In the next section, we investigate a situation around $s \neq 0$.
In appendices, we typically omit subscripts for inner products and norms.
\begin{proposition} \label{prop:DExpminPT}
	Let $\calM$ be a Riemannian manifold whose sectional curvatures are in the interval $[\Klow, \Kup]$, and let $K = \max(|\Klow|, |\Kup|)$. Consider $(x, s) \in \T\calM$ and the geodesic $\gamma(t) = \Exp_x(ts)$. If $\gamma$ is defined and has no interior conjugate point on the interval $[0, 1]$, then
%	no point in $\gamma([0, 1))$ is conjugate to $\gamma(0)$ \TODO{Should be true up to $\pi/\sqrt{\Kup}$ if $\Kup > 0$, and infinity otherwise? Other phrasing is: $\gamma$ is minimizing on $[0, 1]$, see~\citep[Thm.~10.26]{lee2018riemannian}---is the notion of being conjugate dependent on the choice of curve? yes, see p298. Also, Prop.~10.20 states $\gamma(t)$ is conjugate to $\gamma(0)$ along $\gamma$ if and only if $\D\Exp_x(ts)$ is singular; I believe this is notably the case if $\|ts\| < \inj(x)$; so, we could allow $\|s\| \leq \inj(x)$ (non-strict); then call upon lower-bounds of inj given by $\Kup$.}, then
	\begin{align}
		\forall \dot s \in \T_x\calM, && \|(T_s - P_s)[\dot s]\| \leq K \cdot f_{\Klow}(\|s\|) \cdot \|\dot s_\perp\|,
		\label{eq:DExpPgammaone}
	\end{align}
	where $\dot s_\perp = \dot s - \frac{\inner{s}{\dot s}}{\inner{s}{s}}s$ is the component of $\dot s$ orthogonal to $s$, $T_s = \D\Exp_x(s)$ and $P_{ts}$ denotes parallel transport along $\gamma$ from $\gamma(0)$ to $\gamma(t)$.
	(The inequality holds with equality if $\Klow = \Kup$.)
%	In particular,
	%\TODO{factor out the particulars into statements about the special functions?}
%	\begin{enumerate}
%		\item If $\Klow \geq 0$, then
%				\begin{align}
%					\|(\D\Exp_x(s) - P^\gamma_1)[\dot s]\| & \leq K \|\dot s_\perp\| \frac{\|s\|^2}{6}.
%				\end{align}
%		\item If $\Klow \leq 0$, then
%				\begin{align}
%					\|(\D\Exp_x(s) - P^\gamma_1)[\dot s]\| & \leq K \|\dot s_\perp\| \left( \frac{\|s\|^2}{6} + O(|\Klow| \|s\|^4) \right).
%				\end{align}
%				and the conditions on $\gamma$ always hold \TODO{well, no; need $\Kup \leq 0$ too..}.
%		\item
		If it also holds that $\|s\| \leq \pi / \sqrt{|\Klow|}$, then
		\begin{align}
			\forall \dot s \in \T_x\calM, && \|(T_s - P_s)[\dot s]\| \leq \frac{1}{3} K \|s\|^2 \|\dot s_\perp\|.
		\end{align}
%	\end{enumerate}
%	\begin{enumerate}
%		\item Provided $\gamma$ is defined in $[0, 1]$, the condition that it be minimizing can be omitted if $\Kup \leq 0$, and it can be replaced by $\|s\| \leq \pi / \sqrt{\Kup}$ if $\Kup > 0$.
		% Notice: we're not saying that the conditions on gamma are satisfied if so; only that they can be replaced. That's because the actual condition on gamma is this thing with conjugate points, and that one is implied by the bound on s. Whether or not it implies that gamma is minimizing would require some thinking. Life is too short though.
%		
%		\item If $\Klow = \Kup$ (constant sectional curvature), then~\eqref{eq:DExpPgammaone} holds with equality.
		%	\TODO{Add a comment about constant sectional curvature $K$: we should then get equality; should follow from $M(t)$ being equal to $K$ times identity, and, more importantly, from $\D\Exp_x(s)$ being a multiple of $P^\gamma_1$. It's basically all in~\citep[Prop.~10.12]{lee2018riemannian}, and you can also get it using polarization on $M(t)$ here: no need for too many details.}
%	\end{enumerate}
\end{proposition}
\begin{proof}
%Consider $(x, s) \in \T\calM$ and the geodesic $\gamma(t) = \Exp_x(ts)$.
For convenience, we consider $\|s\| = 1$: the result follows by a simple rescaling of $t$.
Given any tangent vector $\dot s \in \T_x\calM$, consider the following smooth vector field along $\gamma$:
\begin{align}
	J(t) & = \D\Exp_x(ts)[t \dot s].
	\label{eq:DExpPTJ}
\end{align}
By~\citep[Prop.~10.10]{lee2018riemannian}, this is the unique Jacobi field satisfying the initial conditions
\begin{align}
	J(0) & = 0 & \textrm{ and } & & \Ddt J(0) & = \dot s,
\end{align}
where $\Ddt$ is the covariant derivative along curves induced by the Riemannian connection.
%(Explicitly, $\Gamma(q, t) = \Exp_x(t(s+q\dot s))$ is a variation through geodesics, $\Gamma(0, t) = \gamma(t)$ and $J(t) = \left. \frac{\partial}{\partial q} \Gamma(q, t) \right|_{q = 0}$.)
Thus, $J$ is smooth and obeys the ordinary differential equation (ODE) known as the Jacobi equation:
\begin{align}
	\Ddttwo J(t) + R(J(t), \gamma'(t))\gamma'(t) = 0,
\end{align}
where $R$ denotes Riemannian curvature.
Fix $e_d = s$ and pick $e_1, \ldots, e_{d-1}$ so that $e_1, \ldots, e_d$ form an orthonormal basis for $\T_x\calM$.
%With $P^\gamma_t$ denoting parallel transport along $\gamma$ from $0$ to $t$, define
Parallel transport this basis along $\gamma$ as
\begin{align}
	E_i(t) & = P_{ts}(e_i), & i & = 1, \ldots, d,
\end{align}
so that $E_1(t), \ldots, E_d(t)$ form an orthonormal basis for $\T_{\gamma(t)}\calM$. Expand $J$ as
\begin{align}
	J(t) & = \sum_{i = 1}^{d} a_i(t) E_i(t)
\end{align}
with uniquely defined smooth, real functions $a_1, \ldots, a_d$. Plugging this expansion into the Jacobi equation yields the ODE
\begin{align}
	\sum_{i = 1}^{d} a_i''(t) E_i(t) + \sum_{i = 1}^{d} a_i(t) R(E_i(t), E_d(t)) E_d(t) = 0,
\end{align}
where we used the Leibniz rule on $\Ddt$, the fact that $\Ddt E_i = 0$, linearity of the Riemann curvature endomorphism in its inputs, and the fact that
\begin{align*}
	\gamma'(t) = P_{ts}(\gamma'(0)) = E_d(t).
\end{align*}
Taking an inner product of this ODE against each one of the fields $E_j(t)$ yields $d$ ODEs:
\begin{align}
	a_j''(t) & = - \sum_{i = 1}^{d} a_i(t) \inner{R(E_i(t), E_d(t)) E_d(t)}{E_j(t)}, & j & = 1, \ldots, d.
	\label{eq:ODEcoordinatesJacobi}
\end{align}
Furthermore, the initial conditions fix $a_i(0) = 0$ and $a_i'(0) = \inner{\dot s}{e_i}$ for $i = 1, \ldots, d$.
%$a_1(0) = \cdots = a_d(0) = 0$ and $\dot s = a_1'(0) e_1 + \cdots + a_d'(0)e_d$.

Owing to symmetries of Riemannian curvature, the summation above can be restricted to the range $1, \ldots, d-1$. For the same reason, $a_d''(t) = 0$, so that
\begin{align}
	a_d(t) = a_d(0) + t a_d'(0) = t \inner{\dot s}{s}.
\end{align}
It remains to solve for the first $d-1$ coefficients (they are decoupled from $a_d$). This effectively splits the solution $J$ into two fields: one tangent  (aligned with $\gamma'$), and one normal (orthogonal to $\gamma'$):
\begin{align}
	J(t) & = t \inner{\dot s}{s} \gamma'(t) + J_\perp(t), & J_\perp(t) & = \sum_{i = 1}^{d-1} a_i(t) E_i(t).
	\label{eq:Jsplit}
\end{align}
The normal part is the Jacobi field with initial conditions $J_\perp(0) = 0$ and $\Ddt J_\perp(0) = \dot s_\perp$, where $\dot s_\perp = \dot s - \inner{\dot s}{s}s$ is the component of $\dot s$ orthogonal to $s$.

Introducing vector notation for the first $d-1$ ODEs, let $a(t) \in \reals^{d-1}$ have components $a_1(t), \ldots, a_{d-1}(t)$, and let $M(t) \in \reals^{(d-1)\times(d-1)}$ have entries
\begin{align}
	M_{ji}(t) = \inner{R(E_i(t), E_d(t)) E_d(t)}{E_j(t)}.
	\label{eq:Mjit}
\end{align}
Then, equations in~\eqref{eq:ODEcoordinatesJacobi} for $j = 1, \ldots, d-1$ can be written succinctly as
\begin{align}
	a''(t) & = - M(t)a(t).
\end{align}
%with $a(0) = 0$ and $a'(0)$ being the vector of coordinates of the component of $\dot s$ orthogonal to $s$, in the basis $e_1, \ldots, e_{d-1}$.
Since $a(t)$ is smooth, it holds that
\begin{align}
	a(t) & = a(0) + \int_{0}^{t} a'(\tau) \dtau = a(0) + t a'(0) + \int_{0}^{t} \int_{0}^{\tau} a''(\theta) \dtheta \dtau.
\end{align}
Initial conditions specify $a(0) = 0$, so that (with $\|\cdot\|$ also denoting the standard Euclidean norm and associated operator norm in real space):
\begin{align}
	\|a(t) - t a'(0)\| & \leq \int_{0}^{t} \int_{0}^{\tau} \|M(\theta)\| \|a(\theta)\| \dtheta \dtau.
	\label{eq:DExpPTboundintegral}
\end{align}
The left-hand side is exactly what we seek to control.
Indeed, initial conditions ensure $\dot s = a_1'(0) e_1 + \cdots + a_d'(0) e_d$, and:
\begin{align*}
	\|(\D\Exp_x(ts) - P_{ts})[t\dot s]\| & = \|J(t) - P_{ts}(t\dot s)\| \\
											 & = \left\| \sum_{i = 1}^{d} \left[ a_i(t) E_i(t) - t a_i'(0) E_i(t) \right] \right\| \\
											 & = \sqrt{\|a(t) - t a'(0)\|^2 + |a_d(t) - t a_d'(0)|^2} \\
											 & = \|a(t) - t a'(0)\|.
\end{align*}
%where we used linearity of parallel transport.
For the right-hand side of~\eqref{eq:DExpPTboundintegral}, first note that $M(t)$ is a symmetric matrix owing to the symmetries of $R$.

Additionally, for any unit-norm $z \in \reals^{d-1}$,
\begin{align}
	z\transpose M(t) z & = \sum_{i,j = 1}^{d-1} z_i z_j \inner{R(E_i(t), E_d(t)) E_d(t)}{E_j(t)} = \inner{R(v, \gamma'(t)) \gamma'(t)}{v},
\end{align}
where $v = z_1 E_1(t) + \cdots + z_{d-1} E_{d-1}(t)$ is a tangent vector at $\gamma(t)$: it is orthogonal to $\gamma'(t)$ and also has unit norm.
By definition of sectional curvature $K(\cdot, \cdot)$~\eqref{eq:sectionalcurvature}, it follows that
\begin{align}
	z \transpose M(t) z & = K(v, \gamma'(t)).
%	\begin{cases}
%		(\|z\|^2 - z_1^2) K(v, \gamma'(t)) & \textrm{ if $v$ is linearly independent of $\gamma'(t)$, and} \\
%		0 & \textrm{ otherwise.}
%	\end{cases}
\end{align}
By symmetry of $M(t)$, we conclude that
\begin{align}
	\|M(t)\| & = \max_{z \in \reals^{d-1}, \|z\| = 1} |z\transpose M(t) z| \leq K,
\end{align}
where $K \geq 0$ is such that all sectional curvatures of $\calM$ along $\gamma$ are in the interval $[-K, K]$. Going back to~\eqref{eq:DExpPTboundintegral}, we have so far shown that
\begin{align}
	\|(\D\Exp_x(ts) - P_{ts})[t\dot s]\| & \leq K \int_{0}^{t} \int_{0}^{\tau} \|a(\theta)\| \dtheta \dtau.
	\label{eq:DExpPTboundintegralfurther}
\end{align}
It remains to bound $\|a(\theta)\|$.
%To this end, let $\dot s_\perp = \dot s - \inner{s}{\dot s}s$ denote the component of $\dot s$ which is orthogonal to $s$, and let $J_\perp(t)$ be defined as $J$~\eqref{eq:DExpPTJ}, only with $\dot s$ replaced by $\dot s_\perp$. Owing to the separation of $a_d(t)$ above in the system of ODEs, it is clear that
%\begin{align}
%	J(t) & = J_\perp(t) + a_d(t) E_d(t).
%	J_\perp(t) & = \sum_{i = 1}^{d-1} a_i(t) E_i(t).
%\end{align}
By~\eqref{eq:Jsplit}, we see that $\|a(t)\| = \|J_\perp(t)\|$. By the Jacobi field comparison theorem~\citep[Thm.~11.9b]{lee2018riemannian} and our assumed lower-bound on sectional curvature, we can now claim that, for $t \geq 0$, with $h_\Klow(t)$ as defined by~\eqref{eq:hklow},
\begin{align}
	\|a(t)\| = \|J_\perp(t)\| & \leq  h_\Klow(t)\|\dot s_\perp\|,
	\label{eq:UBJperp}
\end{align}
provided $\gamma$ has no interior conjugate point on $[0, t]$.
%$\gamma$ does not cross a conjugate point to $\gamma(0)$ on the interval $[0, t]$ or $\gamma(t)$ is the first such point on that interval,
%\TODO{there are no points conjugate to $\gamma(0)$ on the set $\gamma([0, t))$},
%and all sectional curvatures of $\calM$ are lower-bounded by $\Klow$.
Combining with~\eqref{eq:DExpPTboundintegralfurther} and with the definitions of $h_\Klow$~\eqref{eq:hklow}, $g_\Klow$~\eqref{eq:gklow} and $f_\Klow$~\eqref{eq:fklow}, we find
\begin{align}
	\|(\D\Exp_x(ts) - P_{ts})[t\dot s]\| & \leq K \|\dot s_\perp\| \int_{0}^{t} \int_{0}^{\tau} h_\Klow(\theta) \dtheta \dtau \nonumber \\
		& = K \|\dot s_\perp\| \int_{0}^{t} g_\Klow(\tau) \dtau \nonumber \\
		& = K \|\dot s_\perp\| \cdot t f_\Klow(t).
%		& = K \|\dot s_\perp\| \cdot \begin{cases}
%		\frac{t^3}{6} & \textrm{ if } \Klow = 0, \\
%		r^2\left( t - r\sin(t/r) \right) & \textrm{ if } \Klow = 1/r^2 > 0, \\
%		r^2\left( r\sinh(t/r) - t \right) & \textrm{ if } \Klow = -1/r^2 < 0.
%		\end{cases}
\end{align}
It only remains to divide through by $t$, and to rescale $s$ so that $t$ plays the role of $\|s\|$.
%All three expressions right of the brace behave like $t^3/6 + O(t^5)$.

%The bound~\eqref{eq:PminTboundgeneral} clearly holds for $\Klow \geq 0$. For $\Klow < 0$, consider
%\begin{align*}
%	f(t) = r^2 (\sinh(t/r)/(t/r) - 1).
%\end{align*}
%Then, $f(r\tau) = r^2(\sinh(\tau)/\tau - 1) \leq r^2 \tau^2 / 3$ for $\tau \in [0, \pi]$, so that $f(t) \leq t^2 / 3$ for $t \in [0, \pi r]$.

For the special case where $\Kup = \Klow = \pm K$ (constant sectional curvature), one can show (for example by polarization) that $M(t) = \pm K I_{d-1}$, that is, $M(t)$ is a multiple of the identity matrix.
As a result, the ODEs separate and are easily solved (see also~\citep[Prop.~10.12]{lee2018riemannian}).
Explicitly, with $\|s\| = 1$,
\begin{align}
	\D\Exp_x(ts)[t \dot s] & = J(t) = h_{\pm K}(t) P_{ts}(\dot s_\perp) + t P_{ts}(\dot s_\parallel),
	\label{eq:DExpxtsconstantseccurv}
\end{align}
where $\dot s_\parallel = \inner{\dot s}{s} s$ is the component of $\dot s$ parallel to $s$. Hence,
\begin{align}
	\D\Exp_x(ts)[t \dot s] - P_{ts}(t \dot s) & = (h_{\pm K}(t) - t) P_{ts}(\dot s_\perp),
\end{align}
and the claim follows easily after dividing through by $t$ and rescaling.
\end{proof}

As a continuation of the previous proof and in anticipation of our needs in Appendix~\ref{app:controllingcprimeprime}, we provide a lemma controlling the Jacobi field $J$ and its covariant derivative, assessing both the full field and its normal component.
\begin{lemma} \label{lem:Jbounds}
	Let $\calM$ be a Riemannian manifold whose sectional curvatures are in the interval $[\Klow, \Kup]$, and let $K = \max(|\Klow|, |\Kup|)$.
	Consider $(x, s) \in \T\calM$ with $\|s\| = 1$ and the geodesic $\gamma(t) = \Exp_x(ts)$.
	Given a tangent vector $\dot s \in \T_x\calM$, consider the Jacobi field $J$ defined by~\eqref{eq:Jsplit}:
	\begin{align*}
		J(t) & = t \inner{\dot s}{s} \gamma'(t) + J_\perp(t),
	\end{align*}
	where $J_\perp$ is the Jacobi field along $\gamma$ with initial conditions $J_\perp(0) = 0$ and $\Ddt J_\perp(0) = \dot s_\perp$, and $\dot s_\perp = \dot s - \inner{\dot s}{s}s$ is the component of $\dot s$ orthogonal to $s$.
	For $t \geq 0$ such that $\gamma$ is defined and has no interior conjugate point on the interval $[0, t]$, the following inequalities hold:
	\begin{align}
		\left\| J(t) \right\| & \leq \max(t, h_\Klow(t)) \|\dot s\|,  &  \left\| \Ddt J(t) \right\| & \leq \left( 1 + K g_{\Klow}(t) \right) \|\dot s\|, \nonumber \\
		\left\| J_\perp(t) \right\| & \leq  h_\Klow(t)\|\dot s_\perp\|, &  \left\| \Ddt J_\perp(t) \right\| & \leq \left( 1 + K g_{\Klow}(t) \right) \|\dot s_\perp\|, \nonumber
%		\label{eq:Jboundnorm}
	\end{align}
	where $h_\Klow(t)$ and $g_\Klow(t)$ are defined by~\eqref{eq:hklow} and~\eqref{eq:gklow}.
\end{lemma}
\begin{proof}
	The proof is a continuation of that of Proposition~\ref{prop:DExpminPT}. Using notation as in there,
	\begin{align*}
		J_\perp(t) & = \sum_{i = 1}^{d-1} a_i(t) E_i(t).
	\end{align*}
	Since $J_\perp$ and $\Ddt J_\perp$ are orthogonal to $\gamma' = E_d$, we know that
	\begin{align*}
		\|J\|^2 & = t^2 \inner{\dot s}{s}^2 + \|J_\perp\|^2 & \textrm{ and } &&
		\left\|\Ddt J\right\|^2 & = \inner{\dot s}{s}^2 + \left\|\Ddt J_\perp\right\|^2.
	\end{align*}
	The bound $\|J_\perp(t)\| \leq  h_\Klow(t)\|\dot s_\perp\|$ appears explicitly as~\eqref{eq:UBJperp}. With $\alpha$ denoting the angle between $s$ and $\dot s$, we may write $\inner{\dot s}{s}^2 = \cos(\alpha)^2\|\dot s\|^2$ and $\|\dot s_\perp\|^2 = \sin(\alpha)^2\|\dot s\|^2$, so that
	\begin{align*}
		\|J\|^2 & \leq t^2  \left( \cos(\alpha)^2 + \left(\frac{h_\Klow(t)}{t}\right)^2 \sin(\alpha)^2 \right)\|\dot s\|^2.
	\end{align*}
	Since the maximum of $\alpha \mapsto \cos(\alpha)^2 + q \sin(\alpha)^2$ with $q \in \reals$ is $\max(1, q)$, we find for $t \geq 0$ that
	% see notes Sep 24, 2019: it's just a matter of getting the critical points and evaluating f there.
	\begin{align*}
		\|J\| & \leq \max(t, h_\Klow(t)) \|\dot s\|.
	\end{align*}
		
%
% --- This was a statement valid for any $\dot s$, even not a perpendicular one; but no need.
%
%\begin{align*}
%	\|J(t)\|^2 & = (a_d(t))^2 + \|a(t)\|^2  \\
%			   & \leq t^2 \|\dot s_\parallel\|^2 + (h_\Klow(t))^2 \|\dot s_\perp\|^2 \\
%			   & = t^2 \|\dot s\|^2 \left( \cos(\alpha)^2 + \left(\frac{h_\Klow(t)}{t}\right)^2 \sin(\alpha)^2 \right),
%\end{align*}
%where $\alpha$ is the angle between $\dot s$ and $s$. Since the maximum of $\alpha \mapsto \cos(\alpha)^2 + q \sin(\alpha)^2$ with $q \in \reals$ is $\max(1, q)$, we find for $t \geq 0$ that % see notes Sep 24, 2019: it's just a matter of getting the critical points and evaluating f there.
%\begin{align}
%	\|J(t)\| & \leq \|\dot s\| \cdot \begin{cases}
%	t & \textrm{ if } \Klow = 0, \\
%	t & \textrm{ if } \Klow = 1/r^2 > 0, \\
%	r \sinh(t/r) & \textrm{ if } \Klow = -1/r^2 < 0.
%	\end{cases}
%\end{align}
%Since $\sinh(\tau) \leq 4\tau$ for $\tau \in [0, \pi]$, we conclude that regardless of the sign of $\Klow$ it holds
%\begin{align}
%	\|J(t)\| & \leq 4 t \|\dot s\|
%	\label{eq:Jboundnorm}
%\end{align}
%for $0 \leq t \leq \pi/\sqrt{|\Klow|}$.
%
%
%\TODO{About $\Ddt J$}
%

	With the same tools, we may also bound $\Ddt J = \inner{\dot s}{s}\gamma' + \Ddt J_\perp$.
	Indeed, its coordinates in the frame $E_1, \ldots, E_d$ are given by $a_1', \ldots, a_d'$ with $a_d'(t) = \inner{\dot s}{s}$, so that
	\begin{align*}
		\left\| \Ddt J(t) \right\|^2 & = \inner{\dot s}{s}^2 + \left\| \Ddt J_\perp(t) \right\|^2 = \inner{\dot s}{s}^2 + \|a'(t)\|^2,
	\end{align*}	
	where $a(t) \in \reals^{d-1}$ collects the $d-1$ first coordinates.
	Moreover,
	\begin{align*}
		a'(t) = a'(0) + \int_{0}^{t} a''(\tau) \dtau = a'(0) - \int_{0}^{t} M(\tau) a(\tau) \dtau.
	\end{align*}
	Note that
	\begin{align*}
		\left\| \int_{0}^{t} M(\tau) a(\tau) \dtau \right\| \leq K \|\dot s_\perp\| \int_{0}^{t} h_{\Klow}(\tau) \dtau = K \|\dot s_\perp\| g_{\Klow}(t).
	\end{align*}
	%where $g_{\Klow}(t)$ is defined by the integral.
	Combining with the fact that $\|a'(0)\| = \|\dot s_\perp\|$, we get
	\begin{align*}
		\left\| \Ddt J_\perp(t) \right\| & \leq \left( 1 + K g_{\Klow}(t) \right) \|\dot s_\perp\|,
	\end{align*}
	as announced. We now conclude along the same lines as above with
	\begin{align*}
		\left\| \Ddt J(t) \right\|^2 & \leq \left( \cos(\alpha)^2 + \left( 1 + K g_{\Klow}(t) \right)^2 \sin(\alpha)^2 \right)\|\dot s\|^2.
	\end{align*}
	Since $\max(1, 1 + K g_{\Klow}(t)) = 1 + K g_{\Klow}(t)$, we reach the desired conclusion.
\end{proof}

%\begin{corollary}
%	Have a corollary with neat bounds for $\|s\| \leq r$.
%\end{corollary}

\section{Controlling the initial acceleration $c''(0)$} \label{app:controllingcprimeprime}

In this section, we build a proof for Proposition~\ref{prop:cprimeprimecontrol}, whose aim is to control the initial intrinsic acceleration $c''(0)$ of the curve $c(t) = \Exp_x(s + t \dot s)$.
Since $c'(t) = \D\Exp_x(s + t \dot s)[\dot s]$, we can think of this result as giving us access to a second derivative of the exponential map $\Exp_x$ away from the origin.
As a first step, we build an ODE whose solution encodes $c''(0)$.
\begin{proposition} \label{prop:WODE}
	Let $\calM$ be a Riemannian manifold with Riemannian connection $\nabla$ and Riemann curvature endomorphism $R$.
	Consider $(x, s) \in \T\calM$ with $\|s\| = 1$ and the geodesic $\gamma(t) = \Exp_x(ts)$.
	Furthermore, consider a tangent vector $\dot s \in \T_x\calM$ and the curve
	\begin{align*}
		c_{ts, \dot s}(q) & = \Exp_x(ts + q\dot s)
	\end{align*}
	defined for some fixed $t$.
	Let $J$ be the Jacobi field along $\gamma$ with initial conditions $J(0) = 0$ and $\Ddt J(0) = \dot s$.
%	(It is defined on the same domain as $\gamma$.)
	We use it to define a new vector field $H$ along $\gamma$:
	\begin{align*}
		H & = 4 R(\gamma', J) \Ddt J + (\nabla_J R)(\gamma', J) \gamma' + \left(\nabla_{\gamma'} R\right)(\gamma', J) J.
	\end{align*}
	The smooth vector field $W$ along $\gamma$ defined by the linear ODE
	\begin{align*}
		\Ddttwo W + R(W, \gamma') \gamma'& = H
	\end{align*}
	with initial conditions $W(0) = 0$ and $\Ddt W(0) = 0$ is also defined on the same domain as $\gamma$. This vector field is related to the initial intrinsic acceleration of the curve $c_{ts, \dot s}$ as follows:
	\begin{align*}
		W(t) & = t^2 c_{ts, \dot s}''(0).
	\end{align*}
	Furthermore, the vector field $H$ is equivalently defined as
	\begin{align*}
		H & = 4 R(\gamma', J_\perp) \Ddt J + (\nabla_J R)(\gamma', J_\perp) \gamma' + \left(\nabla_{\gamma'} R\right)(\gamma', J_\perp) J,
	\end{align*}
	where $J_\perp$ the Jacobi field along $\gamma$ with initial conditions $J_\perp(0) = 0$ and $\Ddt J_\perp(0) = \dot s_\perp = \dot s - \inner{\dot s}{s}s$.
\end{proposition}
\begin{proof}
	Define
	\begin{align*}
		\Gamma(q, t) & = \Exp_x(t(s + q \dot s)),
	\end{align*}
	a variation through geodesics of the geodesic
	\begin{align*}
		\gamma(t) & = \Gamma(0, t) = \Exp_x(ts).
	\end{align*}
	Then,
	\begin{align*}
		J(t) & = \partial_q \Gamma(0, t) = \left. \D\Exp_x(t(s + q\dot s))[t \dot s] \right|_{q = 0} = \D\Exp_x(ts)[t\dot s]
	\end{align*}
	is the Jacobi field along $\gamma$ with initial conditions $J(0) = 0$ and $\Ddt J(0) = \dot s$: the same field we considered in the proof of Proposition~\ref{prop:DExpminPT}.
	Further consider
	\begin{align}
		W(t) & = \left(\Ddq \partial_q \Gamma\right)(0, t), % = \left. \Ddq \D\Exp_x(t(s + q\dot s))[t \dot s] \right|_{q = 0},
	\end{align}
	another smooth vector field along $\gamma$.
	This field is related to acceleration of curves of the form
	\begin{align*}
		c_{s, \dot s}(q) & = \Exp_x(s + q\dot s),
	\end{align*}
	because $c_{ts, t\dot s}(q) =  \Gamma(q, t)$.
	Specifically,
	\begin{align}
		W(t) & = \left(\Ddq \partial_q \Gamma\right)(0, t) = c''_{ts, t\dot s}(0) = t^2 c''_{ts, \dot s}(0).
		\label{eq:Wcprimeprimelink}
	\end{align}
	To verify the last equality, differentiate the identity $c_{ts, t\dot s}(q) = c_{ts, \dot s}(tq)$ twice with respect to $q$, with the chain rule.
	This shows in particular that
%	$W(t)$ is exactly the vector field we wish to bound, and also that
	\begin{align}
		W(0) & = 0 & \textrm{ and } & & \Ddt W(0) & = 0.
	\end{align}
	Our goal is to derive a second-order ODE for $W$.
	In so doing, we repeatedly use the two following results from Riemannian geometry which allow us to commute certain derivatives:
	\begin{itemize}
		\item \citep[Prop.~7.5]{lee2018riemannian} For every smooth vector field $V$ along $\Gamma$ (meaning $V(q, t)$ is tangent to $\calM$ at $\Gamma(q, t)$),
		\begin{align}
			\Ddt \Ddq V - \Ddq \Ddt V = R(\partial_t \Gamma, \partial_q \Gamma) V,
		\end{align}
		where $R$ is the Riemann curvature endomorphism.
		\item \citep[Lem.~6.2]{lee2018riemannian} The \emph{symmetry lemma} states
		\begin{align}
			\Ddq \partial_t \Gamma = \Ddt \partial_q \Gamma.
		\end{align}
	\end{itemize}
	With the link between $W$ and $\Ddq \partial_q \Gamma$ in mind, we compute a first derivative with respect to $t$:
	\begin{align*}
		\Ddt \Ddq \partial_q \Gamma = \Ddq \Ddt \partial_q \Gamma + R(\partial_t \Gamma, \partial_q \Gamma) \partial_q \Gamma,
	\end{align*}
	then a second derivative:
	\begin{align*}
		\Ddt \Ddt \Ddq \partial_q \Gamma & = \Ddt \Ddq \Ddt \partial_q \Gamma + \Ddt \left\{R(\partial_t \Gamma, \partial_q \Gamma) \partial_q \Gamma\right\}. % \\
%			& = \Ddt \Ddq \Ddt \partial_q \Gamma + \left(\nabla_{\partial_t \Gamma} R\right)(\partial_t \Gamma, \partial_q \Gamma) \partial_q \Gamma + R\left(\Ddt \partial_t \Gamma, \partial_q \Gamma\right) \partial_q \Gamma \\
%			& \qquad  + R\left(\partial_t \Gamma, \Ddt \partial_q \Gamma\right) \partial_q \Gamma + R(\partial_t \Gamma, \partial_q \Gamma) \Ddt \partial_q \Gamma,
	\end{align*}
%	where in the last step we used the chain rule for tensors~\citep[pp95--103]{lee2018riemannian}. % This is not a great reference actually; doesn't give exactly what we need
%	If we evaluate this identity at $q = 0$, the left-hand side becomes $\Ddttwo W(t)$
	Our goal is to evaluate this expression for $q = 0$, in which case the left-hand side yields $\Ddttwo W$. However, it is unclear how to evaluate the first term on the right-hand side at $q = 0$.
	Focusing on that term for now, apply the commutation rule on the first two derivatives:
	\begin{align*}
		\Ddt \Ddq \Ddt \partial_q \Gamma & = \Ddq \Ddt \Ddt \partial_q \Gamma + R(\partial_t \Gamma, \partial_q \Gamma) \Ddt \partial_q \Gamma.
	\end{align*}
	Focusing on the first term once more, apply the symmetry lemma then the commutation rule:
	\begin{align*}
		\Ddq \Ddt \Ddt \partial_q \Gamma & = \Ddq \left\{ \Ddt \Ddq \partial_t \Gamma \right\} = \Ddq \left\{ \Ddq \Ddt \partial_t \Gamma + R(\partial_t \Gamma, \partial_q \Gamma) \partial_t \Gamma\right\} = \Ddq \left\{R(\partial_t \Gamma, \partial_q \Gamma) \partial_t \Gamma\right\}.
	\end{align*}
	To reach the last equality, we used that $\Ddt \partial_t \Gamma$ vanishes identically since $t \mapsto \Gamma(q, t)$ is a geodesic for every fixed $q$.
	Combining, we find
	\begin{align}
		\Ddt \Ddt \Ddq \partial_q \Gamma & = R(\partial_t \Gamma, \partial_q \Gamma) \Ddt \partial_q \Gamma + \Ddt \left\{R(\partial_t \Gamma, \partial_q \Gamma) \partial_q \Gamma\right\} + \Ddq \left\{R(\partial_t \Gamma, \partial_q \Gamma) \partial_t \Gamma\right\}.
	\end{align}
	Using the chain rule for tensors as in~\eqref{eq:nablaUR} (see also \citep[pp95--103]{lee2018riemannian} or \citep[Def.~3.17]{oneill}), we can further expand the right-most term:
	\begin{align*}
		\Ddq \left\{R(\partial_t \Gamma, \partial_q \Gamma) \partial_t \Gamma\right\} & = \left(\nabla_{\partial_q \Gamma} R\right)(\partial_t \Gamma, \partial_q \Gamma) \partial_t \Gamma + R\left(\Ddq \partial_t \Gamma, \partial_q \Gamma\right) \partial_t \Gamma \\ & \qquad + R\left(\partial_t \Gamma, \Ddq \partial_q \Gamma\right) \partial_t \Gamma + R\left(\partial_t \Gamma, \partial_q \Gamma\right) \Ddq \partial_t \Gamma.
	\end{align*}
	It is now easier to evaluate the whole expression at $q = 0$: using
	\begin{align*}
		\partial_q \Gamma(0, t) & = J(t),  & \partial_t \Gamma(0, t) & = \gamma'(t) & & \textrm{ and } & \left(\Ddq \partial_q \Gamma\right)(0, t) & = W(t)
	\end{align*}
	repeatedly, and also $\Ddq \partial_t \Gamma = \Ddt \partial_q \Gamma$ twice so that it evaluates to $\Ddt J$ at $q = 0$, we find
	\begin{align*}
		\Ddttwo W & = 2 R(\gamma', J) \Ddt J + \Ddt \left\{ R(\gamma', J) J \right\} + (\nabla_J R)(\gamma', J) \gamma' + R\left(\Ddt J, J\right) \gamma' + R(\gamma', W) \gamma'. % + R(\gamma', J) \Ddt J.
	\end{align*}
	This is now an ODE in the single variable $t$, involving smooth vector fields $J$, $W$ and $\gamma'$ along the geodesic $\gamma$.
	We may apply the chain rule for tensors again (we could just as well have done this earlier too):
	\begin{align*}
		\Ddt \left\{ R(\gamma', J) J \right\} & = \left(\nabla_{\gamma'} R\right)(\gamma', J) J + R\left(\gamma', \Ddt J\right) J + R(\gamma', J) \Ddt J,
	\end{align*}
	here too simplifying one term since $\gamma''$ vanishes. The algebraic Bianchi identity~\citep[p203]{lee2018riemannian} states $R(X, Y)Z + R(Y, Z)X + R(Z, X) Y = 0$, so that in particular
	\begin{align*}
		R\left(\Ddt J, J\right) \gamma' + R\left(\gamma', \Ddt J\right) J = - R(J, \gamma') \Ddt J = R(\gamma', J) \Ddt J.
	\end{align*}
	(We also used anti-symmetry of $R$). Overall, $\Ddttwo W + R(W, \gamma') \gamma' = H$ with
	\begin{align*}
		H & = 4 R(\gamma', J) \Ddt J + (\nabla_J R)(\gamma', J) \gamma' + \left(\nabla_{\gamma'} R\right)(\gamma', J) J.
	\end{align*}
	The Jacobi field $J$ splits into its tangent and normal parts~\eqref{eq:Jsplit}:
	\begin{align*}
		J(t) & = t\inner{\dot s}{s}\gamma'(t) + J_\perp(t).
	\end{align*}
	Since $R(\gamma', \gamma') = 0$ by anti-symmetry of $R$, and since for the same reason $(\nabla_\cdot R)(\gamma', \gamma') = 0$ as well, by linearity, we may simplify $H$ to:
	\begin{align*}
		H & = 4 R(\gamma', J_\perp) \Ddt J + (\nabla_J R)(\gamma', J_\perp) \gamma' + \left(\nabla_{\gamma'} R\right)(\gamma', J_\perp) J.
	\end{align*}
	This concludes the proof.
\end{proof}

To reach our main result, it remains to bound the solutions of the ODE in $W$.
In order to do so, we notably need to bound the inhomogeneous term $H$.
For that reason, we require a bound on the covariant derivative of Riemannian curvature.
%By definition, that derivative is zero for so-called \emph{locally symmetric spaces}.
%As one would expect, manifolds with constant sectional curvature are locally symmetric~\citep[Prop.~8.10, Cor.~8.11]{oneill}.
%For such manifolds (which include Euclidean spaces, spheres and hyperbolic spaces), we can set $F = 0$.
\begin{theorem} \label{thm:cprimeprime}
	Let $\calM$ be a Riemannian manifold whose sectional curvatures are in the interval $[\Klow, \Kup]$, and let $K = \max(|\Klow|, |\Kup|)$.
	Also assume $\nabla R$---the covariant derivative of the Riemann curvature endomorphism---is bounded by $F$ in operator norm.
%	\TODO{These conditions on curvature only need to hold along $\gamma$ I imagine? Would need to check the comparison theorems we called upon in the previous theorem too.}
%	There exists $\bar A \geq 0$, \TODO{polynomial? rational? non-decreasing and such that $\bar A(0, 0) = 0$?} function of $K$ and $F$ only, such that the following holds.
	Pick any $(x, s) \in \T\calM$ such that the geodesic $\gamma(t) = \Exp_x(ts)$ is defined for all $t \in [0, 1]$, and such that
	\begin{align*}
		\|s\| & \leq \min\!\left( C \frac{1}{\sqrt{K}}, C' \frac{K}{F} \right)
	\end{align*}
	with some constants $C \leq \pi$ and $C'$.
	For any $\dot s \in \T_x\calM$,
	the curve
	\begin{align*}
		 c(t) = \Exp_x(s + t \dot s)
	\end{align*}
	has initial acceleration bounded as
	\begin{align*}
		\|c''(0)\| & \leq \bar{\bar W} K \|s\| \|\dot s\| \|\dot s_\perp\|,
	\end{align*}
	where $\dot s_\perp = \dot s - \frac{\inner{s}{\dot s}}{\inner{s}{s}} s$ is the component of $\dot s$ orthogonal to $s$ and $\bar{\bar{W}} \in \reals$ is only a function of $C$ and $C'$.
	In particular, for $C, C' \leq \frac{1}{4}$, we have $\bar{\bar W} \leq \frac{3}{2}$.
\end{theorem}
\begin{proof}
	By Remark~\ref{rem:conjugatepoints}, since $C \leq \pi$ we know that $\gamma$ has no interior conjugate point on $[0, 1]$.
	Since the claim is clear for either $s = 0$ or $\dot s = 0$, assume $\|s\| = 1$ for now---we rescale at the end---and $\dot s \neq 0$.
	We also assume $K > 0$: the case $K = 0$ follows easily by inspection of the proof below.
	% K = 0 is not necessarily just the Euclidean case: could also be 1D manifolds; it's easier to state things this way.

	Following Proposition~\ref{prop:WODE}, the goal is to bound $W$: the solution of an ODE with right-hand side given by the vector field $H$.
	As we did in earlier proofs, pick an orthonormal basis $e_1, \ldots, e_d$ for $\T_x\calM$ with $e_d = s$ and transport it along $\gamma$ as $E_i(t) = P_{ts}(e_i)$.
	We expand $W$ and $H$ as
	\begin{align}
		W(t) & = \sum_{i = 1}^{d} w_i(t) E_i(t), & H(t) & = \sum_{i = 1}^{d} h_i(t) E_i(t).
	\end{align}
	%and we still have $J(t) = \sum_{i = 1}^{d} a_i(t) E_i(t)$ and $\gamma'(t) = E_d(t)$.
	This allows us to write the ODE in coordinates:
	\begin{align}
		w''(t) + M(t)w(t) & = h(t),
	\end{align}
	where $M(t)$ is as in~\eqref{eq:Mjit} but defined in $\Rdd$ (thus, it has an extra row and column of zeros), and $w(t), h(t) \in \Rd$ are vectors containing the coordinates of $W(t)$ and $H(t)$.
	Since $W(0) = \Ddt W(0) = 0$, we have $w(0) = w'(0) = 0$ and we deduce
	\begin{align*}
		w(t) & = w(0) + t w'(0) + \int_{0}^{t} \int_{0}^{\tau} w''(\theta) \dtheta \dtau = \int_{0}^{t} \int_{0}^{\tau} -M(\theta)w(\theta) + h(\theta) \, \dtheta \dtau.
	\end{align*}
	Thus,
	\begin{align}
		\|W(t)\| = \|w(t)\| \leq \int_{0}^{t} \int_{0}^{\tau} K \|w(\theta)\| + \|h(\theta)\| \, \dtheta \dtau.
		\label{eq:WboundFirstForm}
	\end{align}
	To proceed, we need a bound on $\|H(t)\| = \|h(t)\|$ and a first bound on $\|W(t)\|$.
	We will then improve the latter by bootstrapping.
	
	Let us first bound $H$.
	%The expressions $R$, $\nabla_{\gamma'} R$ and $\nabla_J R$ denote tensor fields along $\gamma$. -- not sure what was meant here 
	Following~\citep[eq.~(9)]{karcher1970curvatureestimates}, we know that $R$ is bounded (as an operator) as follows:
	\begin{align}
		\|R(X, Y)Z\| & \leq K_0 \|X\| \|Y\| \|Z\| & \textrm{ with } & & K_0 & = \sqrt{K^2 + (25/36)(\Kup - \Klow)^2} \leq 2K,
		\label{eq:boundR}
	\end{align}
	where $X, Y, Z$ are arbitrary vector fields along $\gamma$.
	We further assume that
	\begin{align}
		\|(\nabla_U R) (X, Y) Z\| & \leq F \|U\| \|X\| \|Y\| \|Z\|
	\end{align}
	for some finite $F \geq 0$. Then,
	\begin{align}
		\|H\| & \leq 4 K_0 \|\gamma'\| \left\| \Ddt J \right\| \|J_\perp\| + 2 F \|\gamma'\|^2 \|J\| \|J_\perp\|.
	\end{align}
	Since $\|\gamma'(t)\| = \|s\| = 1$ for all $t$, this expression simplifies somewhat.
	Using Lemma~\ref{lem:Jbounds}, we can also bound all terms involving $J$ and $J_\perp$, so that, also using $K_0 \leq 2K$,
	\begin{align}
		\|H(t)\| &\leq  h_\Klow(t) \Big( 8K \left( 1 + K g_{\Klow}(t) \right)  +  2F \max(t, h_\Klow(t))  \Big) \|\dot s\| \|\dot s_\perp\|.
	\end{align}
	Since $h_\Klow(t) \leq h_{-K}(t) = t \frac{\sinh(\sqrt{K}t)}{\sqrt{K}t}$ and likewise $K g_\Klow(t) \leq K g_{-K}(t) = \cosh(\sqrt{K}t) - 1$, and since $h_{-K}(t) \geq t$, we find
	\begin{align}
		\|H(t)\| & \leq t \frac{\sinh(\sqrt{K}t)}{\sqrt{K}t} \Big( 8K \cosh(\sqrt{K}t)  +  2F t \frac{\sinh(\sqrt{K}t)}{\sqrt{K}t}  \Big) \|\dot s\| \|\dot s_\perp\| % \nonumber \\
%				 & = \sinh(\sqrt{K}t) \Big( 8\sqrt{K}\cosh(\sqrt{K}t)  +  2\frac{F}{K} \sinh(\sqrt{K}t)  \Big) \|\dot s\| \|\dot s_\perp\|
	\end{align}
	for all $t \geq 0$. % need t >= 0 because we took norm of H, so some absolute values should have appeared, but no need if t >= 0.
	Assuming $0 \leq \sqrt{K} t \leq C$ for some $C > 0$, we find
	%use $\frac{\sinh(\sqrt{K}t)}{\sqrt{K}t} \leq 1 + 0.2712 \cdot K t^2 \leq 3.6767$ and $\cosh(\sqrt{K}t) \leq 1 + 1.0732 \cdot Kt^2$ to claim:
	\begin{align}
		\|H(t)\| & \leq \left( aK +  b F t \right) t \|\dot s\| \|\dot s_\perp\|
		\label{eq:boundH}
	\end{align}
	with $a = 8\frac{\sinh(C)\cosh(C)}{C}$ and $b = 2 \frac{\sinh(C)^2}{C^2}$.
%	For example, with $C = \pi$ we have $a \leq 341$ and $b \leq 28$.
	Let us further assume that $0 \leq t \leq C' \frac{K}{F}$. Then, $Ft \leq C' K$ and we write:
	\begin{align}
		\frac{\|H(t)\|}{\|\dot s\| \|\dot s_\perp\|} & \leq \left( a +  b C' \right) K t \triangleq \bar H Kt.
		\label{eq:boundHbis}
	\end{align}
	Let us now obtain a first crude bound on $\|W(t)\|$.
	To this end, introduce
	\begin{align*}
		u(t) & = w'(t) / \sqrt{K}, & y(t) & = \|\dot s\| \|\dot s_\perp\| / \sqrt{K}, & z(t) & = \begin{bmatrix}
		u(t) \\ w(t) \\ y(t)
		\end{bmatrix}.
	\end{align*}
	Then,
	\begin{align*}
		z'(t) & = A(t) z(t), & \textrm{ with } & & A(t) & = \begin{bmatrix}
		0 & -M(t)/\sqrt{K} & h(t) / (\|\dot s\| \|\dot s_\perp\|) \\ \sqrt{K} I & 0 & 0 \\ 0 & 0 & 0
		\end{bmatrix}.
	\end{align*}
	Let $g(t) = \|z(t)\|^2$.
	Then, $g(0) = \|\dot s\|^2 \|\dot s_\perp\|^2 / K$ and
	\begin{align*}
		g'(t) & = 2\inner{z(t)}{z'(t)} = 2 \inner{z(t)}{A(t) z(t)} \leq 2 \|A(t)\| \|z(t)\|^2 = 2 \|A(t)\| g(t).
	\end{align*}
	Gr\"onwall's inequality states that
	\begin{align*}
		g(t) & \leq g(0) \exp\!\left( 2 \int_{0}^{t} \|A(\tau)\| \dtau \right).
	\end{align*}
	By triangle inequality and using $\|M(t)\| \leq K$, we have $\|A(t)\| \leq 2\sqrt{K} + \|h(t)\|/(\|\dot s\| \|\dot s_\perp\|)$.
	Thus, $\|z(t)\|^2$ can be bounded above and below:
	% https://math.stackexchange.com/questions/3361631/bound-on-operator-norm-of-block-matrices
	\begin{align}
		\|w(t)\|^2 + \frac{ \|\dot s\|^2 \|\dot s_\perp\|^2}{K} & \leq \|z(t)\|^2 \leq \frac{ \|\dot s\|^2 \|\dot s_\perp\|^2}{K} \exp\!\left( 4\sqrt{K}t + 2 \int_{0}^{t} \|h(\tau)\| / (\|\dot s\| \|\dot s_\perp\|) \dtau \right).
	\end{align}
	Using our bound on $H(t)$~\eqref{eq:boundHbis}, we find
	\begin{align*}
		\exp\!\left( 4\sqrt{K}t + 2 \int_{0}^{t} \|h(\tau)\| / (\|\dot s\| \|\dot s_\perp\|) \dtau \right) & \leq %\exp\!\left(  4\sqrt{K}t + 2 \int_{0}^{t} 16 K_0 \tau + 32 F \tau^2 + 18 K_0 K \tau^3 \dtau \right) \\
			%& =
			\exp\!\left( 4\sqrt{K}t + \bar H K t^2 \right).
	\end{align*}
	Using $\sqrt{K}t \leq C$ again we deduce this crude bound:
	\begin{align}
		\frac{\|W(t)\|}{\|\dot s\| \|\dot s_\perp\|} &
		%\leq \sqrt{\exp\!\left( 4\sqrt{K}t + \left( a + \frac{2}{3} b C' \right) K t^2 \right) - 1}
		\leq \frac{1}{\sqrt{K}} \sqrt{\exp\!\left( 4C + \bar H C^2 \right) - 1} \triangleq \frac{1}{\sqrt{K}} \bar W.
		\label{eq:boundWconstant}
	\end{align}
	We now return to~\eqref{eq:WboundFirstForm} and plug in our bounds for $H$~\eqref{eq:boundHbis} and $W$~\eqref{eq:boundWconstant} to get an improved bound on $W$: assuming $t$ satisfies the stated conditions,
	\begin{align*}
		\frac{\|W(t)\|}{\|\dot s\| \|\dot s_\perp\|} & \leq \int_{0}^{t} \int_{0}^{\tau} \bar W \sqrt{K} + \bar H K \theta \, \dtheta \dtau = \frac{1}{2} \bar W \sqrt{K} t^2 + \frac{1}{6} \bar H K t^3.
	\end{align*}
	Plug this new and improved bound on $W$ in~\eqref{eq:WboundFirstForm} once again to get:
	\begin{align*}
		\frac{\|W(t)\|}{\|\dot s\| \|\dot s_\perp\|} & \leq \int_{0}^{t} \int_{0}^{\tau} K \left(\frac{1}{2} \bar W \sqrt{K} \theta^2 + \frac{1}{6} \bar H K \theta^3\right) + \bar H K \theta \, \dtheta \dtau \\
			& = \frac{1}{24} \bar W K^{3/2} t^4 + \frac{1}{120} \bar H K^2 t^5 + \frac{1}{6} \bar H K t^3 \\
			& = \left( \frac{1}{6} \bar H + \frac{1}{24} \bar W \sqrt{K}t + \frac{1}{120} \bar H Kt^2 \right) K t^3.
%			& \leq \left( \frac{1}{6} \bar H + \frac{1}{24} \bar W C + \frac{1}{120} \bar H C^2 \right) K t^3 \\
%			& \triangleq \bar{\bar W} K t^3.
	\end{align*}
	We could now bound $\sqrt{K}t$ and $Kt^2$ by $C$ and $C^2$ respectively and stop here.
	However, this yields a constant which depends on $\bar W$: this can be quite large.
	Instead, we plug our new bound in~\eqref{eq:WboundFirstForm} again, repeatedly.
	Doing so infinitely many times, we obtain a sequence of upper bounds, all of them valid.
	The limit of these bounds exists, and is hence also a valid bound.
	It is tedious but not difficult to check that this reasoning leads to the following:
	\begin{align*}
		\frac{\|W(t)\|}{\|\dot s\| \|\dot s_\perp\|} & \leq \frac{1}{6} \bar H \left( 1 + \frac{C^2}{6 \cdot 7} + \frac{C^4}{6 \cdot 7 \cdot 8 \cdot 9} + \frac{C^6}{6 \cdots 11} + \cdots \right) Kt^3.
	\end{align*}
	It is clear that the series converges.
	Let $z$ be the value it converges to; then:
	\begin{align*}
		z & = 1 + \frac{C^2}{6 \cdot 7} + \frac{C^4}{6 \cdot 7 \cdot 8 \cdot 9} + \frac{C^6}{6 \cdots 11} + \cdots \\
		  & = 1 + \frac{C^2}{6 \cdot 7} \left( 1 + \frac{C^2}{8 \cdot 9} + \frac{C^4}{8 \cdots 11} + \cdots \right) \leq 1 + \frac{C^2}{42}z.
		  % This happens to yield a very good bound for that series.
	\end{align*}
	Thus, $z \leq \frac{1}{1 - \frac{C^2}{42}}$.
	All in all, we conclude that
	\begin{align*}
		\frac{\|W(t)\|}{\|\dot s\| \|\dot s_\perp\|} & \leq \bar{\bar W} Kt^3 & \textrm{ with } & & \bar{\bar{W}} & = \frac{1}{6} \bar H \frac{1}{1 - \frac{C^2}{42}} \textrm{ and } \\ & & & & \bar H & = 8\frac{\sinh(C)\cosh(C)}{C} + 2 \frac{\sinh(C)^2}{C^2} C'.
		% Yes: everything is monotonically growing with C, C': I checked numerically.
	\end{align*}
	For example, with $C, C' \leq \frac{1}{4}$, we have $\bar{\bar{W}} \leq \frac{3}{2}$.
	% Note to self: The behavior of barbarW is clearly monotonic in C' because we have monotonicity in barbarH and the coefficient of C' is positive (in the considered range anyway).
	% So it's just a question of setting C' = .25 and checking that the formula for barbarW is monotonic in C.
	% This is the case over the selected interval: plot (1/6) * (1/(1 - C^2/42)) * (8*sinh(C)*cosh(C)/C + (2*(sinh(C)/C)^2)*.25) for C in [0, .25]
	% Max value is ~1.47687: evaluate (1/6) * (1/(1 - C^2/42)) * (8*sinh(C)*cosh(C)/C + (2*(sinh(C)/C)^2)*.25) at C = .25
	
	From Proposition~\ref{prop:WODE}, we know that for the curve
	\begin{align*}
		c_{ts, \dot s}(q) = \Exp_x(ts + q\dot s)
	\end{align*}
	(recall that $s$ has unit norm) it holds that $W(t) = t^2 c_{ts, \dot s}''(0)$.
	Thus,
	\begin{align*}
		\|c_{ts, \dot s}''(0)\| & \leq \bar{\bar W} K \|\dot s\| \|\dot s_\perp\| t.
	\end{align*}
	Allowing $s$ to have norm different from one and rescaling $t$, we conclude that for the curve
	\begin{align*}
		c(t) & = \Exp_x(s + t\dot s)
	\end{align*}
	we have
	\begin{align*}
		\|c''(0)\| & \leq \bar{\bar W} K \|s\| \|\dot s\| \|\dot s_\perp\|,
	\end{align*}
	provided $\|s\| \leq C \frac{1}{\sqrt{K}}$ with $C \leq \pi$ and $\|s\| \leq C' \frac{K}{F}$ and $\gamma(t) = \Exp_x(ts)$ is defined $[0, 1]$.
	%\TODO{wouldn't it actually necessarily be defined? No: think of an incomplete manifold, e.g., punctured plane.}
\end{proof}

We now argue that Theorem~\ref{thm:cprimeprime} is sharp for manifolds with constant sectional curvature.
The claim is clear for flat manifolds ($K = F = 0$), hence we consider $K \neq 0, F = 0$.
For $C, C' > 0$ very small, we can lower $\bar{\bar W}$ arbitrarily close to $\frac{4}{3}$.
Using that sectional curvatures are constant, we have $K_0 = K$~\eqref{eq:boundR} and $F = 0$ so that $\bar{H} = 4\frac{\sinh(C)\cosh(C)}{C}$ yields a valid bound (see~\eqref{eq:boundH} and~\eqref{eq:boundHbis}).
As a result, $\bar{\bar W}$ can be lowered arbitrarily close to $\frac{2}{3}$ by taking $C > 0$ small enough.
As it turns out, $\frac{2}{3}$ is the right constant for manifolds with constant nonzero sectional curvature.

Indeed, consider  the unit sphere $\{ x \in \Rn : x_1^2 + \cdots + x_n^2 = 1 \}$ with Riemannian metric defined by restriction of the Euclidean inner product $\inner{u}{v} = u_1v_1 + \cdots + u_nv_n$ to its tangent spaces.
This manifold has constant positive curvature: $\Klow = \Kup = K = 1$ and $F = 0$. It is tedious but not hard to show that %,
%assuming $\dot s$ is orthogonal to $s$, we have
% see notes Sep 12, 2019 (NB40) + March 3, 2020 (NB42) + code Exp_sphere_acceleration_away_from_origin.m
%\begin{align}
%	\frac{c''(0)}{\|\dot s\|^2} & = \frac{\|s\| - \sin(\|s\|)\cos(\|s\|)}{\|s\|^3} P_s s = \left( \frac{2}{3} + O(\|s\|^2) \right) P_s s.
%	\frac{\sin(\|s\|)^3 + \|s\|\cos(\|s\|) - \sin(\|s\|)}{\|s\|^3}s + \frac{\cos(\|s\|)\sin(\|s\|)^2 - \|s\|\sin(\|s\|)}{\|s\|^2} x \\
%								& = \left(\frac{2}{3} + O(\|s\|^2)\right) s + \left( -\frac{2}{3} \|s\|^2 + O(\|s\|^4) \right)x.
%\end{align}
\begin{align}
	c''(0) & = \frac{\|s\| - \sin(\|s\|) \cos(\|s\|)}{\|s\|^3} \|\dot s_\perp\|^2 \cdot P_s(s) - 2 \frac{\sin(\|s\|) - \cos(\|s\|)\|s\|}{\|s\|^3} \innersmall{s}{\dot s} \cdot \dot s_\perp,
	\label{eq:cprimeprimesphere}
\end{align}
where $\dot s_\perp = \dot s - \frac{\innersmall{\dot s}{s}}{\innersmall{s}{s}}s$ is the component of $\dot s$ orthogonal to $s$.
Using this expression, it follows that $\|c''(0)\| \leq \frac{2}{3} \|s\| \|\dot s\| \|\dot s_\perp\|$ for all $x, s, \dot s$,
%and that for small $s$ we have $\|c''(0)\| \approx \frac{2}{3} \|s\| \|\dot s_\perp\| \|\dot s\|$.
with equality up to $O(\|s\|^3)$ terms.
% NB42, March 19, 2020 + code Exp_hyperbolic_acceleration_away_from_origin.m

Likewise, consider the hyperbolic manifold $\{ x \in \Rn : x_2^2 + \cdots + x_n^2 = x_1^2 - 1 \}$ with Riemannian metric defined by restriction of the Minkowski semi-inner product $\inner{u}{v} = u_2v_2 + \cdots + u_nv_n - u_1v_1$ to its tangent spaces.
This manifold has constant negative curvature: $\Klow = \Kup = -1$ (hence $K = 1$) and $F = 0$. It is tedious but not hard to show that %,
%assuming $\dot s$ is orthogonal to $s$, we have
%\begin{align}
%	\frac{c''(0)}{\|\dot s\|^2} & = \frac{\|s\|\cosh(\|s\|) - \sinh(\|s\|) - \sinh(\|s\|)^3}{\|s\|^3} s + \frac{\|s\| \sinh(\|s\|) - \cosh(\|s\|) \sinh(\|s\|)^2}{\|s\|^2} x \\
%								& = \left(-\frac{2}{3} + O(\|s\|^2)\right) s + \left( -\frac{2}{3} \|s\|^2 + O(\|s\|^4) \right)x.
%	\frac{c''(0)}{\|\dot s\|^2} & = \frac{\|s\| - \sinh(\|s\|)\cosh(\|s\|)}{\|s\|^3} P_s s = \left( -\frac{2}{3} + O(\|s\|^2) \right) P_s s
%\end{align}
\begin{align}
	c''(0) & = -\frac{\sinh(\|s\|) \cosh(\|s\|) - \|s\|}{\|s\|^3} \|\dot s_\perp\|^2 \cdot P_s(s) + 2 \frac{\cosh(\|s\|)\|s\| - \sinh(\|s\|)}{\|s\|^3} \inner{s}{\dot s} \cdot \dot s_\perp,
	\label{eq:cprimeprimehyperbolic}
\end{align}
where $\|\cdot\|$ is the Minkowski semi-norm (which is an actual norm on the tangent spaces)
and $\dot s_\perp = \dot s - \frac{\innersmall{\dot s}{s}}{\innersmall{s}{s}}s$ is the component of $\dot s$ orthogonal to $s$.
%Again, for small $s$ it follows that $\|c''(0)\| \approx \frac{2}{3} \|s\| \|\dot s\|^2$.
This time, one can deduce that $\|c''(0)\| \geq \frac{2}{3} \|s\| \|\dot s\| \|\dot s_\perp\|$ for all $x, s, \dot s$,
still with equality up to $O(\|s\|^3)$ terms.

We close with fairly direct consequences of eqs.~\eqref{eq:cprimeprimesphere} and~\eqref{eq:cprimeprimehyperbolic}.
\begin{proposition} \label{prop:lipschitzpullbacksphere}
	Let $f$ be a real function on the unit sphere, with pullbacks $\hat f_x = f \circ \Exp_x$.
	Assume $f$ is twice differentiable.
	If $f$ has $L$-Lipschitz continuous gradient, then $\hat f_x$ has $\frac{5}{2}L$-Lipschitz continuous gradient on the whole tangent space, for all $x$.
	If moreover $f$ has $\rho$-Lipschitz continuous Hessian, then $\|\nabla^2 \hat f_x(s) - \nabla^2 \hat f_x(0)\| \leq \hat\rho \|s\|$, for all $x, s$, with $\hat\rho = \rho + 3.1 \cdot L$.
\end{proposition}
\begin{proof}
	% April 8, 2020: first pages of NB43.
	From Lemma~\ref{lem:derivativespullback}, we have this expression for all $s, \dot s$ tangent at an arbitrary point $x$:
	\begin{align*}
		\innersmall{\dot s}{\nabla^2 \hat f_x(s)[\dot s]} & = \innersmall{T_s(\dot s)}{\Hess f(y)[T_s(\dot s)]} + \inner{\grad f(y)}{c''(0)},
	\end{align*}
	where $y = \Exp_x(s)$ and $c(t) = \Exp_x(s+t\dot s)$.
	Split $\dot s = \dot s_\parallel + \dot s_\perp$ with $\dot s_\parallel = \frac{\innersmall{\dot s}{s}}{\innersmall{s}{s}}s$.
	It is not difficult to check that $T_s(\dot s) = P_s\!\left(\dot s_\parallel + \frac{\sin(\|s\|)}{\|s\|} \dot s_\perp\right)$.
	Therefore,
%	\begin{align*}
%		\innersmall{\dot s}{\nabla^2 \hat f_x(s)[\dot s]} & = \inner{\dot s_\parallel + \frac{\sin(\|s\|)}{\|s\|} \dot s_\perp}{(P_s^* \circ \Hess f(y) \circ P_s)\left[\dot s_\parallel + \frac{\sin(\|s\|)}{\|s\|} \dot s_\perp\right]} + \inner{\grad f(y)}{c''(0)}.
%	\end{align*}
	\begin{align*}
		|\innersmall{\dot s}{\nabla^2 \hat f_x(s)[\dot s]}| & \leq \left\| P_s^* \circ \Hess f(y) \circ P_s \right\| \left\| \dot s_\parallel + \frac{\sin(\|s\|)}{\|s\|} \dot s_\perp \right\|^2 + \|\grad f(y)\| \|c''(0)\|.
	\end{align*}
	Since parallel transport $P_s$ is an isometry, the operator norm of $P_s^* \circ \Hess f(y) \circ P_s$ is the same as that of $\Hess f(y)$. Moreover, the operator norm of $\Hess f(y)$ is bounded by $L$ since $f$ has $L$-Lipschitz gradient.
	Additionally, since the sphere is compact, there exists a point $z$ such that $\grad f(z) = 0$.
	Say $v$ is such that $\Exp_y(v) = z$: we can arrange to have $\|v\| \leq \pi$.
	Using Lipschitz continuity again then reveals that
	\begin{align}
		\|\grad f(y)\| = \|\grad f(y) - P_v^* \grad f(z)\| \leq L \|v\| \leq L\pi.
	\end{align}
	(Note that this implies $f$ is $\pi L$-Lipschitz continuous.)
	Thus,
	\begin{align*}
		|\innersmall{\dot s}{\nabla^2 \hat f_x(s)[\dot s]}| & \leq L \left( \|\dot s_\parallel\|^2 + \frac{\sin(\|s\|)^2}{\|s\|^2} \|\dot s_\perp \|^2 + \pi \|c''(0)\| \right).
	\end{align*}
	With $\alpha$ representing the angle between $s$ and $\dot s$, we have $\|\dot s_\parallel\|^2 = \cos(\alpha)^2\|\dot s\|^2$ and $\|\dot s_\perp\|^2 = \sin(\alpha)^2 \|\dot s\|^2$.
	Combining also with~\eqref{eq:cprimeprimesphere}, it follows that
	\begin{align*}
		\frac{|\innersmall{\dot s}{\nabla^2 \hat f_x(s)[\dot s]}|}{\|\dot s\|^2} & \leq L \left( \cos(\alpha)^2 + \frac{\sin(\|s\|)^2}{\|s\|^2} \sin(\alpha)^2 + \pi \frac{\|c''(0)\|}{\|\dot s\|^2} \right), \textrm{ with} \\
		\frac{\|c''(0)\|}{\|\dot s\|^2} & = \frac{|\sin(\alpha)|}{\|s\|^2} \sqrt{\sin(\alpha)^2 \Big(\|s\| - \sin(\|s\|) \cos(\|s\|) \Big)^2 + 4 \cos(\alpha)^2 \Big(\cos(\|s\|)\|s\| - \sin(\|s\|) \Big)^2 }.
	\end{align*}
	% See code: April7_2020_cprimeprime_sphere_general
	The right-hand side of the first expression now depends only on two scalar parameters, namely, $\alpha$ and $\|s\|$.
	By inspection, it is easy to see that it is uniformly bounded so that $\frac{|\innersmall{\dot s}{\nabla^2 \hat f_x(s)[\dot s]}|}{\|\dot s\|^2} \leq  \frac{5}{2} L$.
	This immediately implies that $\nabla \hat f_x$ is $\frac{5}{2}L$-Lipschitz continuous.
	
	For the second part of the claim, consider
	\begin{align*}
		\inner{\dot s}{\left( \nabla^2 \hat f_x(s) - \nabla^2 \hat f_x(0) \right)\![\dot s]} & = \innersmall{T_s(\dot s)}{\Hess f(y)[T_s(\dot s)]} + \inner{\grad f(y)}{c''(0)} - \inner{\dot s}{\Hess f(x)[\dot s]}.
	\end{align*}
	Introduce $q = 1 - \frac{\sin(\|s\|)}{\|s\|}$ so that $T_s(\dot s) = P_s(\dot s - q \dot s_\perp)$.
	Plugging this into the first term above (and using $P_s(\dot s_\perp) = \dot s_\perp$) yields
	\begin{align*}
		\inner{\dot s}{\left( \nabla^2 \hat f_x(s) - \nabla^2 \hat f_x(0) \right)\![\dot s]} & = \innersmall{\dot s}{(P_s^* \circ \Hess f(y) \circ P_s - \Hess f(x))[\dot s]} \\ & \quad + q^2 \inner{\dot s_\perp}{\Hess f(y)[\dot s_\perp]} - 2q\inner{\dot s_\perp}{\Hess f(y)[P_s(\dot s)]} \\ & \quad + \inner{\grad f(y)}{c''(0)}.
	\end{align*}
	We bound the first line using that $\Hess f$ is $\rho$-Lipschitz continuous.
	We bound the second line using that $\Hess f$ is bounded by $L$ everywhere since $\grad f$ is $L$-Lipschitz continuous.
	Finally, we bound the third line using $\|\grad f(y)\| \leq \pi L$ as above and $\|c''(0)\| \leq \frac{2}{3} \|s\| \|\dot s_\perp\| \|\dot s\|$ (which can be deduced from~\eqref{eq:cprimeprimesphere}).
	Thus,
	\begin{align*}
		\|\nabla^2 \hat f_x(s) - \nabla^2 \hat f_x(0)\| & \leq \rho \|s\| + q^2 L + 2|q|L + \frac{2\pi}{3} L \|s\|.
	\end{align*}
	It remains to check that $\frac{q^2}{\|s\|} + \frac{2|q|}{\|s\|} + \frac{2\pi}{3}$ (function of $\|s\|$ only) is bounded by $3.1$.
	% plot (1-sin(t)/t)^2 / t + 2*abs(1-sin(t)/t) / t + 2*pi/3, for t from 0 to 100
	% plot (1-sin(t)/t)^2 / t + 2*abs(1-sin(t)/t) / t + 2*pi/3, for t from 3.53 to 3.54 ----  3.0691 bound
\end{proof}

\begin{proposition} \label{prop:lipschitzpullbackhyperbolic}
	Let $f$ be a real function on hyperbolic space $\calM$ of dimension at least two, with pullbacks $\hat f_x = f \circ \Exp_x$.
	Assume $f$ is twice differentiable.
	If $f$ is not constant, then for all $\ell \geq 0$ there exists $(x, s) \in \T\calM$ such that $\|\nabla^2 \hat f_x(s)\| > \ell$.
%	Thus, the gradients of the pullbacks are not uniformly $\ell$-Lipschitz continuous for any finite $\ell$.
	Thus, there does not exist a finite $\ell$ such that $\nabla \hat f_y$ is $\ell$-Lipschitz continuous for all $y \in \calM$.
\end{proposition}
\begin{proof}
	Since $f$ is not constant, there exists a point $y \in \calM$ such that $\grad f(y) \neq 0$.
	Define $v = \frac{1}{\|\grad f(y)\|} \grad f(y)$.
	Since $\calM$ has dimension at least two, we can pick $\dot s \in \T_y\calM$ orthogonal to $v$ with $\|\dot s\| = 1$.
	Consider the geodesic $\gamma(t) = \Exp_y(tv)$ and its velocity $\gamma'(t) = P_{tv} \gamma'(0) = P_{tv} v$.
	It is easy to check that $\dot s$ is tangent with unit norm at $\gamma(t)$ and orthogonal to $\gamma'(t)$ for all $t$.
	For some $t \neq 0$ to be determined, let $x = \gamma(t)$.
	Notice that $y = \Exp_x(s)$ with $s = -t\gamma'(t) = -P_{tv}(tv)$, tangent at $x$.
	Moreover, $P_s = P_{tv}^{-1}$, so that $P_s s = -tv$.
	Lemma~\ref{lem:derivativespullback} and eq.~\eqref{eq:cprimeprimehyperbolic} then provide the following expression:
	\begin{align*}
		\innersmall{\dot s}{\nabla^2 \hat f_x(s)[\dot s]}
		        & = \innersmall{T_s(\dot s)}{\Hess f(y)[T_s(\dot s)]} - \frac{\sinh(\|s\|) \cosh(\|s\|) - \|s\|}{\|s\|^3} \inner{\grad f(y)}{P_s(s)} \\
				& = \frac{\sinh(t)^2}{t^2} \innersmall{\dot s}{\Hess f(y)[\dot s]} + t \frac{\sinh(t) \cosh(t) - t}{t^3} \|\grad f(y)\|,
	\end{align*}
	where to reach the second line we use $T_s(\dot s) = \D\Exp_x(s)[\dot s] = \frac{\sinh(\|s\|)}{\|s\|} \dot s$ owing to orthogonality of $s$ and $\dot s$, and we also use $\|s\| = |t|$ and the fact that the two fractions are (positive) even functions of $t$.
	Notice that $\innersmall{\dot s}{\Hess f(y)[\dot s]}$ is independent of our choice of $t$.
	If $\innersmall{\dot s}{\Hess f(y)[\dot s]}$ is nonzero, let $t$ have the same sign; otherwise, the sign of $t$ is free.
	Then, we deduce the following bound:
	\begin{align*}
		\|\nabla^2 \hat f_x(s)\| & \geq
		%\frac{\sinh(t)^2}{t^2} |\innersmall{\dot s}{\Hess f(y)[\dot s]}| +
		|t| \frac{\sinh(t) \cosh(t) - t}{t^3} \|\grad f(y)\|.
	\end{align*}
	The right-hand side grows unbounded with $|t|$: for any $\ell \geq 0$, it is possible to pick $t$ (with appropriate sign) so that the right-hand side exceeds $\ell$.
	This choice of $t$ identifies a pair $(x, s) \in \T\calM$ as announced. 
\end{proof}

\section{Proof from Section~\ref{sec:assuparams} about parameter relations} \label{app:params}

As a general comments: here and throughout, constants are not optimized at all.
In part, this is so that there is leeway in the precise definition of parameters.
For example, the step-size $\eta$ does not need to be exactly equal to $1/4\ell$, but it is convenient to assume equality to simplify many tedious computations.
\begin{lemma} \label{lem:params}
	With parameters and assumptions as laid out in Section~\ref{sec:assuparams}, the following hold:
	\begin{multicols}{2}
		\begin{enumerate}
			\item $\kappa \geq 2$ and $\log_2(\theta^{-1}) \geq \frac{5}{2}$,
			\item $\epsilon \leq \frac{1}{2} \ell b$ and $2\ell\mathscr{M} < \frac{1}{2} \ell b$,
			\item $r \leq \frac{1}{64} \mathscr{L}$ and $\mathscr{L} \leq s \leq \frac{1}{32} b$,
			%		\item the interval $\left[ \frac{\theta^2}{\eta(2-\theta)^2}, \ell \right]$ is not empty, \TODO{Might no longer be necessary}
			%		\item $\sqrt{8\sqrt{\kappa}\eta\mathscr{E} \mathscr{T}} = \frac{1}{\sqrt{2}} \mathscr{L}$ \TODO{ever used?},
			\item $\ell \mathscr{M}^2 \geq \frac{64\mathscr{E}}{\mathscr{T}}$ and $\theta \ell \mathscr{M}^2 \geq \frac{4\mathscr{E}}{\mathscr{T}}$,
			\item $\epsilon r + \frac{\ell}{2} r^2 \leq \frac{1}{4} \mathscr{E}$,
			\item $\frac{\mathscr{L}^2}{16\sqrt{\kappa}\eta\mathscr{T}} = \mathscr{E}$,
			\item $\frac{s^2}{2\eta} \geq 2\mathscr{E}$ and $\frac{(\gamma - 4\hat\rho s)s^2}{2} \geq 2\mathscr{E}$,
			\item $\hat\rho (\mathscr{L} + \mathscr{M}) \leq \sqrt{\hat\rho \epsilon}$.
		\end{enumerate}
	\end{multicols}
\end{lemma}
\begin{proof} %[Proof of Lemma~\ref{lem:params}]
	We require $c \geq 5$ and use~\aref{assu:epsilon} for $\epsilon > 0$, namely: $\sqrt{\hat \rho \epsilon} \leq \frac{1}{2}\ell$ and $\epsilon \leq b^2 \hat \rho$.
	\begin{enumerate}
		\item The assumption $\sqrt{\hat \rho \epsilon} \leq \frac{1}{2}\ell$ is equivalent to $\kappa \geq 2$ and to $\log_2(\theta^{-1}) = \log_2(4\sqrt{\kappa}) \geq \frac{5}{2}$.
		\item Using both $\sqrt{\hat \rho \epsilon} \leq \frac{1}{2}\ell$ and $\epsilon \leq b^2 \hat \rho$, we have $\epsilon = \sqrt{\epsilon} \sqrt{\epsilon} \leq \frac{1}{2}\ell \frac{1}{\sqrt{\hat\rho}} \cdot b \sqrt{\hat\rho} = \frac{1}{2}\ell b$.
		Thus,
		\begin{align*}
		\mathscr{M} = c^{-1}\frac{\epsilon \sqrt{\kappa}}{\ell} = c^{-1}\frac{\epsilon \sqrt{\frac{\ell}{\sqrt{\hat \rho \epsilon}}}}{\ell} = c^{-1}\sqrt{\frac{\epsilon }{\ell} \cdot \sqrt{\frac{\epsilon}{\hat\rho}}} \leq c^{-1} \sqrt{\frac{1}{2}b \cdot \sqrt{b^2}} = \frac{1}{\sqrt{2}c} b.
		\end{align*}
		We have $\mathscr{M} < \frac{1}{4} b$ with $c \geq 3$.
		%		      This covers $2\ell\mathscr{M} < \frac{1}{2} \ell b$. To verify $2b\eta\ell + 2\eta\ell\mathscr{M} < b$, use $\eta \ell = \frac{1}{4}$.
		\item %Using $\epsilon \leq b^2 \hat \rho$ and $\chi \geq 1$, we see that $4\mathscr{L} = 8\sqrt{\frac{\epsilon}{\hat{\rho}}} \chi^{-2}c^{-3} \leq 8bc^{-3} < b$ with $c > 2$.
		Compute: $\frac{r}{\mathscr{L}} = \eta \epsilon \chi^{-5}c^{-8} \cdot \sqrt{\frac{\hat{\rho}}{4\epsilon}} \chi^{2}c^{3} = \frac{1}{8} \frac{\epsilon}{\ell} \sqrt{\frac{\hat\rho}{\epsilon}} \chi^{-3} c^{-5} = \frac{1}{8} \frac{\sqrt{\hat\rho \epsilon}}{\ell} \chi^{-3} c^{-5} \leq \frac{1}{16} c^{-5}$, where we used $\sqrt{\hat \rho \epsilon} \leq \frac{1}{2}\ell$ and $\chi \geq 1$.
		The claim follows with $c \geq 2$.
		The last claim is direct: $s = \frac{1}{32} \sqrt{\frac{\epsilon}{\hat \rho}} \leq \frac{1}{32} b$ since $\epsilon \leq b^2 \hat \rho$, and also $\mathscr{L} = \sqrt{\frac{4\epsilon}{\hat{\rho}}} \chi^{-2}c^{-3} = 64 s \chi^{-2}c^{-3} \leq s$ with $c \geq 4$.
		\item For the first identity, check that $\ell\mathscr{M}^2 = \frac{\mathscr{E}}{\mathscr{T}} (\chi c)^6 \sqrt{\kappa}$, then use $\kappa \geq 1$, $\chi \geq 1$ and $c \geq 2$.
		For the second identity, check that $\theta\ell\mathscr{M}^2 = \frac{\mathscr{E}}{4\mathscr{T}} (\chi c)^6$, then use $\chi \geq 1$ and $c \geq 2$.
		\item Consider both $\frac{\epsilon r}{\mathscr{E}} = \eta \epsilon^2 \sqrt{\frac{\hat{\rho}}{\epsilon^3}} c^{-1} = \frac{\sqrt{\hat\rho\epsilon}}{4\ell} c^{-1} \leq \frac{1}{8c} \leq \frac{1}{8}$
		and $\frac{\ell r^2}{2\mathscr{E}} = \frac{1}{2} \ell \eta^2 \sqrt{\hat{\rho} \epsilon}\chi^{-5}c^{-9} = \frac{1}{32} \frac{\sqrt{\hat{\rho} \epsilon}}{\ell} \chi^{-5}c^{-9} \leq \frac{1}{64c^9} \leq \frac{1}{64}$, both with $c \geq 1$ and $\chi \geq 1$.
		\item This is a direct computation. % Yes, already verified this a thousand times.
		\item Use $2\sqrt{\hat \rho \epsilon} \leq \ell$ to check $\frac{s^2}{2\eta} = \frac{\ell}{512} \frac{\epsilon}{\hat\rho} \geq \frac{\sqrt{\hat \rho \epsilon}}{256} \frac{\epsilon}{\hat\rho} =  \frac{1}{256} \sqrt{\frac{\epsilon^3}{\hat\rho}} \geq 2\mathscr{E}$ with $c \geq 3$, and
		$(\gamma - 4\hat\rho s)\frac{s^2}{2} = \left( \frac{\sqrt{\hat\rho \epsilon}}{4} - \frac{\sqrt{\hat\rho \epsilon}}{8} \right) \frac{1}{2048} \frac{\epsilon}{\hat\rho} = \frac{1}{16384} \sqrt{\frac{\epsilon^3}{\hat\rho}} \geq 2\mathscr{E}$ with $c \geq 5$.
		\item Compute: $\hat\rho (\mathscr{L} + \mathscr{M}) = \sqrt{4\hat{\rho}\epsilon} \chi^{-2}c^{-3} + \frac{\hat\rho \epsilon \sqrt{\kappa}}{\ell}c^{-1} \cdot \frac{\sqrt{\kappa}}{\sqrt{\kappa}} = \sqrt{\hat\rho\epsilon}\left( 2\chi^{-2}c^{-3} + \frac{1}{\sqrt{\kappa}} c^{-1} \right).$ Now reach the desired bound using $\chi \geq 1$, $\sqrt{\kappa} \geq 1$ and $c \geq 2$.
		\qedhere
	\end{enumerate}
\end{proof}

\section{Proofs from Section~\ref{sec:agdTxM} about AGD in a ball of a tangent space} \label{app:agdTxM}

We give a proof of the lemma which states that iterates generated by $\TSS$ remain in certain balls.
Such a lemma is not necessary in the Euclidean case.
\begin{proof}[Proof of Lemma~\ref{lem:TSSballs}]
	Because of how $\TSS$ works, if it defines $u_j$ for some $j$, then $s_j$ must have already been defined.
	Moreover, if $\|s_{j+1}\| > b$, then the algorithm terminates before defining $u_{j+1}$.
	It follows that if $u_0, \ldots, u_q$ are defined then $\|s_0\|, \ldots, \|s_q\|$ are all at most $b$.
	Also, $\TSS$ ensures $\|u_0\|, \ldots, \|u_q\|$ are all at most $2b$ by construction.
	
	Recall that $\theta = \frac{1}{4\sqrt{\kappa}}$.
	From Lemma~\ref{lem:params} we know $\kappa \geq 2$ so that $\theta \leq 1$.
	Moreover, $2\eta\gamma = \frac{1}{8\kappa} = \frac{1}{2\sqrt{\kappa}} \theta \leq \theta$.
	It follows that $\theta_j$ as presented in~\eqref{eq:thetaj} is well defined in the interval $[\theta, 1]$.
	Indeed, either $\|s_j + (1-\theta) v_j\| \leq 2b$, in which case $\theta_j = \theta$; or the line segment connecting $s_j$ to $s_j + (1-\theta) v_j$ intersects the boundary of the sphere of radius $2b$ at exactly one point.
	By definition, this happens at $s_j + (1-\theta_j) v_j$ with $1-\theta_j$ chosen in the interval $[0, 1-\theta]$, that is, $\theta_j \in [\theta, 1]$.
	
	Now assume that $\|\grad f(x)\| \leq \frac{1}{2} \ell b$. Then, for all $0 \leq j \leq q$ we have
	\begin{align*}
		\|\eta \nabla \hat f_x(u_j)\| \leq \eta \left( \|\nabla \hat f_x(u_j) - \nabla \hat f_x(0)\| + \|\nabla \hat f_x(0)\| \right) \leq \eta \left( \ell \|u_j\| + \frac{1}{2}\ell b \right) \leq \frac{5}{2}\eta\ell b = \frac{5}{8} b < b,
	\end{align*}
	where we used the fact that $\|u_j\| \leq 2b$ and that $\nabla \hat f_x$ is $\ell$-Lipschitz continuous in the ball of radius $3b$ around the origin (by~\aref{assu:Mandf}), the fact that $\grad f(x) = \nabla \hat f_x(0)$, and the fact that $\eta \ell = \frac{1}{4}$ by definition of $\eta$.
	Consequently, if $s_{q+1}$ is defined, then
	\begin{align*}
		\|s_{q+1}\| = \|u_q - \eta \nabla \hat f_x(u_q)\| \leq \|u_q\| + \|\eta \nabla \hat f_x(u_q)\| \leq 3b.
	\end{align*}
	If additionally it holds that $\|u_q\| = 2b$, then
	\begin{align*}
		\|s_{q+1}\| = \|u_q - \eta \nabla \hat f_x(u_q)\| \geq \|u_q\| - \|\eta \nabla \hat f_x(u_q)\| > b.
	\end{align*}
	(Mind the strict inequality: this one will matter.)
\end{proof}

Lemma~\ref{lem:TSSballs} applies under the assumptions of Lemmas~\ref{lem:TSSEj}, \ref{lem:TSStraveldist} and~\ref{lem:TSSNCE}. This ensures all vectors $u_j, s_j$ remain in $B_x(3b)$, hence the strongest provisions of \aref{assu:Mandf} apply: we use this often in the proofs below.

We give a proof of the lemma which states that the Hamiltonian is monotonically decreasing along iterations.
%\TODO{The conclusions of the proof hold if we require $\ell \eta \leq \frac{1}{2}$ instead of $\ell \eta = \frac{1}{4}$, which gives some lee-way in choosing step-sizes.--True for all lemmas in this section? Unlikely (for example, last claim in lemma~\ref{lem:TSSEj} seems to require $\eta$ large enough). Should just remove this remark...}
\begin{proof}[Proof of Lemma~\ref{lem:TSSEj}]
	This follows almost exactly~\citep[Lem.~9 and~20]{jin2018agdescapes}, with one modification to allow $\theta_j$~\eqref{eq:thetaj} to be larger than $1/2$: this is necessary in our setup because we need to cap $u_j$ to the ball of radius $2b$, requiring values of $\theta_j$ which can be arbitrarily close to~$1$.
	
	Since $\nabla \hat f_x$ is $\ell$-Lipschitz continuous in $B_x(3b)$ and $u_j, s_{j+1} \in B_x(3b)$, standard calculus and the identity $s_{j+1} = u_j - \eta \nabla \hat{f}_x(u_j)$ show that
	\begin{align*}
		\hat f_x(s_{j+1})  \leq \hat f_x(u_j) + \innersmall{s_{j+1} - u_j}{\nabla \hat f_x(u_j)} + \frac{\ell}{2} \|s_{j+1} - u_j\|^2 
						   = \hat f_x(u_j) - \eta \left( 1 - \frac{\ell\eta}{2} \right) \|\nabla \hat f_x(u_j)\|^2.
	\end{align*}
	Since $\ell \eta = \frac{1}{4} \leq \frac{1}{2}$, it follows that
	\begin{align*}
		\hat f_x(s_{j+1}) & \leq \hat f_x(u_j) - \frac{3 \eta}{4} \|\nabla \hat f_x(u_j)\|^2.
	\end{align*}
	Turning to $E_{j+1}$ as defined by~\eqref{eq:Ej} and with the identity $v_{j+1} = s_{j+1} - s_j$, we compute:
	\begin{align*}
		E_{j+1} = \hat f_x(s_{j+1}) + \frac{1}{2\eta} \|v_{j+1}\|^2 \leq \hat f_x(u_j) - \frac{3 \eta}{4} \|\nabla \hat f_x(u_j)\|^2 + \frac{1}{2\eta} \|s_{j+1} - s_j\|^2.
	\end{align*}
	Notice that
	\begin{align*}
		\|s_{j+1} - s_j\|^2 = \|u_j - \eta \nabla \hat{f}_x(u_j) - s_j\|^2 = \|u_j - s_j\|^2 - 2\eta \innersmall{u_j - s_j}{\nabla \hat{f}_x(u_j)} + \eta^2 \|\nabla \hat{f}_x(u_j)\|^2.
	\end{align*}
	Moreover, the fact that $s_{j+1}$ is defined means that~\eqref{eq:NCC} does not trigger with $(x, s_j, u_j)$; in other words:
	\begin{align*}
		\hat{f}_x(s_j) \geq \hat{f}_x(u_j) - \innersmall{u_j - s_j}{\nabla \hat{f}_x(u_j)} - \frac{\gamma}{2}\norm{u_j - s_j}^2.
	\end{align*}
	Combining, we find that
	\begin{align*}
		E_{j+1} & \leq \hat f_x(u_j) - \frac{3 \eta}{4} \|\nabla \hat f_x(u_j)\|^2 - \innersmall{u_j - s_j}{\nabla \hat{f}_x(u_j)} + \frac{1}{2\eta} \|u_j - s_j\|^2 + \frac{\eta}{2} \|\nabla \hat{f}_x(u_j)\|^2 \\
				& \leq \hat{f}_x(s_j) + \left(\frac{\gamma}{2} + \frac{1}{2\eta}\right) \norm{u_j - s_j}^2 - \frac{\eta}{4} \|\nabla \hat f_x(u_j)\|^2.
	\end{align*}
	Using the identities $u_j - s_j = (1-\theta_j) v_j$ and $E_j = \hat f_x(s_j) + \frac{1}{2\eta} \|v_j\|^2$, we can further write:
	\begin{align*}
		E_{j+1} & \leq \hat{f}_x(s_j) + \left(\frac{\gamma}{2} + \frac{1}{2\eta}\right) (1-\theta_j)^2 \norm{v_j}^2 - \frac{\eta}{4} \|\nabla \hat f_x(u_j)\|^2 \\
				& = E_j + \left( \frac{\gamma (1-\theta_j)^2}{2} + \frac{(1-\theta_j)^2 - 1}{2\eta} \right) \|v_j\|^2  - \frac{\eta}{4} \|\nabla \hat f_x(u_j)\|^2 \\
				& = E_j + \frac{1}{2\eta} \left( \eta \gamma (1-\theta_j)^2 + (1-\theta_j)^2 - 1 \right) \|v_j\|^2  - \frac{\eta}{4} \|\nabla \hat f_x(u_j)\|^2.
	\end{align*}
	From Lemma~\ref{lem:TSSballs} we know that $\eta\gamma \leq \frac{1}{2} \theta_j$ and that $\theta_j$ is in the interval $[0, 1]$.
	It is easy to check that the function $\theta_j \mapsto \frac{1}{2} \theta_j (1-\theta_j)^2 + (1-\theta_j)^2 - 1$ is upper-bounded by $-\theta_j$ over the interval $[0, 1]$.
	% https://www.wolframalpha.com/input/?i=plot+-x+and+%28.5+*+x+*+%281-x%29%5E2+%2B+%281-x%29%5E2+-+1%29+for+x+from+0+to+1
	Thus,
	\begin{align*}
		E_{j+1} & \leq E_j - \frac{\theta_j}{2\eta} \|v_j\|^2  - \frac{\eta}{4} \|\nabla \hat f_x(u_j)\|^2 \leq E_j,
	\end{align*}
	as announced.
	
	In closing, note that if $\|v_j\| \geq \mathscr{M}$ then Lemma~\ref{lem:params} shows
	\begin{align*}
		E_j - E_{j+1} \geq \frac{\theta_j}{2\eta} \|v_j\|^2 \geq \frac{\theta}{2\eta} \mathscr{M}^2 = 2 \theta \ell\mathscr{M}^2 \geq \frac{4\mathscr{E}}{\mathscr{T}},
	\end{align*}
	which concludes the proof.
%	
%	We now consider two special cases.
%	First, assume $\|v_j\| \geq \mathscr{M}$.
%	Then, Lemma~\ref{lem:params} shows
%	\begin{align*}
%		E_j - E_{j+1} \geq \frac{\theta_j}{2\eta} \|v_j\|^2 \geq \frac{\theta}{2\eta} \mathscr{M}^2 = 2 \theta \ell\mathscr{M}^2 \geq \frac{4\mathscr{E}}{\mathscr{T}}.
%	\end{align*}
%	Second, assume $\|\nabla \hat f_x(s_j)\| \geq 2\ell\mathscr{M}$.
%	Then, using that $\nabla \hat f_x$ is $\ell$-Lipschitz continuous on $B_x(3b)$,
%	\begin{align*}
%		\|\nabla \hat f_x(u_j)\| \geq \|\nabla \hat f_x(s_j)\| - \|\nabla \hat f_x(u_j) - \nabla \hat f_x(s_j)\| \geq 2\ell\mathscr{M} - \ell \|u_j - s_j\| = 2\ell\mathscr{M} - (1-\theta_j) \ell \|v_j\|.
%	\end{align*}
%	Assuming additionally that $\|v_j\| \leq \mathscr{M}$ (as otherwise we already know the result to be true), we conclude with $(1-\theta_j) \ell \|v_j\| \leq \ell \mathscr{M}$ that
%	\begin{align*}
%		E_j - E_{j+1} \geq \frac{\eta}{4} \|\nabla \hat f_x(u_j)\|^2 \geq \frac{\eta}{4} (\ell\mathscr{M})^2 = \frac{1}{16} \ell \mathscr{M}^2 \geq \frac{4\mathscr{E}}{\mathscr{T}},
%	\end{align*}
%	where we used $\ell \mathscr{M}^2 \geq \frac{64\mathscr{E}}{\mathscr{T}}$ as provided by Lemma~\ref{lem:params}.
\end{proof}

We give a proof of the improve-or-localize lemma.
\begin{proof}[Proof of Lemma~\ref{lem:TSStraveldist}]
	This follows from~\citep[Cor.~11]{jin2018agdescapes}, with some modifications for variable $\theta_j$ and because we allow $\theta_j > \frac{1}{2}$.
	By triangular inequality then Cauchy--Schwarz, we have
	\begin{align*}
		\|s_q - s_{q'}\|^2 = \left\| \sum_{j = q'}^{q-1} s_{j+1} - s_j \right\|^2 \leq \left( \sum_{j = q'}^{q-1} \|s_{j+1} - s_j\| \right)^2 \leq (q - q') \sum_{j = q'}^{q-1} \|s_{j+1} - s_j\|^2.
	\end{align*}
	Now use the inequality $\|a + b\|^2 \leq (1+C) \|a\|^2 + \frac{1+C}{C} \|b\|^2$ (valid for all vectors $a, b$ and reals $C > 0$) with $C = 2\sqrt{\kappa} - 1$ (positive owing to $\kappa \geq 1$ by Lemma~\ref{lem:params}) to see that
	\begin{align*}
		\|s_{j+1} - s_j\|^2 = \|(s_{j+1} - u_j) + (u_j - s_j)\|^2 \leq 2\sqrt{\kappa} \|s_{j+1} - u_j\|^2 + \frac{2\sqrt{\kappa}}{2\sqrt{\kappa} - 1} \|u_j - s_j\|^2.
	\end{align*}
	By construction, we have $s_{j+1} = u_j - \eta \nabla \hat{f}_x(u_j)$ and $u_j = s_j + (1-\theta_j) v_j$. Thus:
	\begin{align*}
		\|s_{j+1} - s_j\|^2 & \leq 2\sqrt{\kappa} \eta^2 \|\nabla \hat{f}_x(u_j)\|^2 + \frac{2\sqrt{\kappa}(1-\theta_j)^2}{2\sqrt{\kappa} - 1} \|v_j\|^2 \\
			                & = 16\sqrt{\kappa} \eta \left( \frac{\eta}{8} \|\nabla \hat{f}_x(u_j)\|^2 + \frac{1}{2\eta}\frac{(1-\theta_j)^2}{4(2\sqrt{\kappa} - 1)} \|v_j\|^2 \right).
	\end{align*}
	We focus on the second term: recall from Lemma~\ref{lem:TSSballs} that $\theta_j \in [\theta, 1]$ with $\theta = \frac{1}{4\sqrt{\kappa}}$, and notice that $(1-t)^2 \leq 4(2\sqrt{\kappa} - 1)t$ for all $t$ in the interval defined by $\frac{1-\theta\pm\sqrt{1-2\theta}}{\theta}$.
	% https://www.wolframalpha.com/input/?i=solve+%281-x%29%5E2+%3D+4*%281%2F%282*theta%29+-+1%29*x
	This holds a fortiori for all $t$ in $[\theta, 1]$ because $\theta \leq \frac{1}{4}$ owing to $\kappa \geq 1$.
	% plot (1-t+sqrt(1-2t))/t for t from 0 to 1/2
	% evaluate (1-t+sqrt(1-2t))/t for t = 1/2
	% plot (1-t-sqrt(1-2t))/t and t for t from 0 to 1/4
	It follows that
	\begin{align*}
		\|s_{j+1} - s_j\|^2 & \leq 16\sqrt{\kappa} \eta \left( \frac{\eta}{8} \|\nabla \hat{f}_x(u_j)\|^2 + \frac{\theta_j}{2\eta} \|v_j\|^2 \right).
	\end{align*}
	Apply Lemma~\ref{lem:TSSEj} to the parenthesized expression to deduce that
	\begin{align*}
		\|s_{j+1} - s_j\|^2 & \leq 16\sqrt{\kappa} \eta \left( E_j - E_{j+1} \right).
	\end{align*}
	Plug this into the first inequality of this proof to conclude with a telescoping sum.
\end{proof}

We give a proof of the lemma which states that, upon witnessing significant non-convexity, it is possible to exploit that observation to drive significant decrease in the cost function value.
\begin{proof}[Proof of Lemma~\ref{lem:TSSNCE}]
	This follows almost exactly~\citep[Lem.~10 and 17]{jin2018agdescapes}.
	We need a slight modification because the Hessian $\nabla^2 \hat f_x$ may not be Lipschitz continuous in all of $B_x(3b)$: our assumptions only guarantee a type of Lipschitz continuity with respect to the origin of $\T_x\calM$.
	Interestingly, even if the last momentum step was capped (that is, if $\theta_j \neq \theta$)---something which does not happen in the Euclidean case---the result goes through.
	%\TODO{See also Remark 5.13 in Chris' write-up: nice argument for why even if $u_j$ was capped to $B_x(2b)$ we are still fine. This may be good to point out explicitly in main text to highlight differences with Jin's.}
	
	First, consider the case $\|v_j\| \geq s$, where $s$ is a parameter set in Section~\ref{sec:assuparams}.
	Then, $\NCE(x, s_j, v_j) = s_j$.
	It follows from the definition of $E_j$~\eqref{eq:Ej} that
	\begin{align*}
		\hat f_x(\NCE(x, s_j, v_j)) & = \hat f_x(s_j) = E_j - \frac{1}{2\eta} \|v_j\|^2 \leq E_j - \frac{s^2}{2\eta}.
	\end{align*}
	
	Second, consider the case $\|v_j\| < s$.
	We know that $v_j \neq 0$ as otherwise $u_j = s_j + (1-\theta_j) v_j = s_j$: this would contradict the assumption that~\eqref{eq:NCC} triggers with $(x, s_j, u_j)$.
	Expand $\hat f_x$ around $u_j$ in a truncated Taylor series with Lagrange remainder to see that
	\begin{align*}
		\hat f_x(s_j) = \hat f_x(u_j) + \innersmall{\nabla \hat f_x(u_j)}{s_j - u_j} + \frac{1}{2}\innersmall{\nabla^2 \hat f_x(\zeta_j)[s_j - u_j]}{s_j - u_j}
	\end{align*}
	with $\zeta_j = ts_j + (1-t)u_j$ for some $t \in [0, 1]$.
	Since~\eqref{eq:NCC} triggers with $(x, s_j, u_j)$, we also know that
	\begin{align*}
		\hat{f}_x(s_j) < \hat{f}_x(u_j) + \innersmall{\nabla \hat{f}_x(u_j)}{s_j - u_j} - \frac{\gamma}{2}\norm{s_j - u_j}^2.
	\end{align*}
	The last two claims combined yield:
	\begin{align}
		\innersmall{\nabla^2 \hat f_x(\zeta_j)[s_j - u_j]}{s_j - u_j} < - \gamma \norm{s_j - u_j}^2.
		\label{eq:NCEpfnablatwo}
	\end{align}
	Consider $\dot{v} = s \frac{v_j}{\norm{v_j}}$ as defined in the call to $\NCE$.
	Let $\tilde v$ be either $\dot v$ or $-\dot v$, chosen so that $\innersmall{\nabla \hat f_x(s_j)}{\tilde v} \leq 0$ (at least one of the two choices satisfies this condition).
	By construction, $\NCE(x, s_j, v_j)$ is the element of the triplet $\{s_j, s_j+\dot v, s_j-\dot v\}$ where $\hat f_x$ is minimized.
	Since $s_j + \tilde v$ belongs to this triplet, it follows through another truncated Taylor series with Lagrange remainder (this time around $s_j$) that
	\begin{align}
		\hat f_x(\NCE(x, s_j, v_j)) \leq \hat f_x(s_j + \tilde v) & = \hat f_x(s_j) + \innersmall{\nabla \hat f_x(s_j)}{\tilde v} + \frac{1}{2}\innersmall{\nabla^2 \hat f_x(\zeta_j')[\tilde v]}{\tilde v} \nonumber \\
				& \leq \hat f_x(s_j) + \frac{1}{2}\innersmall{\nabla^2 \hat f_x(\zeta_j')[\tilde v]}{\tilde v}
				\label{eq:NCEpfhatfxNCEfoo}
	\end{align}
	with $\zeta_j' = s_j + t'\tilde v$ for some $t' \in [0, 1]$.
	Since $\tilde v$ is parallel to $v_j$ which itself is parallel to $s_j - u_j$ (by definition of $u_j$), we deduce from~\eqref{eq:NCEpfnablatwo} that
	\begin{align*}
		\innersmall{\nabla^2 \hat f_x(\zeta_j)[\tilde v]}{\tilde v} < -\gamma \|\tilde v\|^2 = -\gamma s^2.
	\end{align*}
	We aim to use this to work on~\eqref{eq:NCEpfhatfxNCEfoo}, but notice that $\nabla^2 \hat f_x$ is evaluated at two possibly distinct points, namely, $\zeta_j$ and $\zeta_j'$: we need to use the Lipschitz properties of the Hessian to relate them.
	To this end, notice that $\zeta_j$ and $\zeta_j'$ both live in $B_x(3b)$.
	Indeed, $\|\tilde v\| = \|\dot v\| = s \leq b$ by Lemma~\ref{lem:params} and $\|s_j\| \leq b, \|u_j\| \leq 2b$ by Lemma~\ref{lem:TSSballs}.
	Thus, $\|\zeta_j\| \leq \|s_j\| + \|u_j\| \leq b + 2b = 3b$ and $\|\zeta_j'\| \leq \|s_j\| + \|\tilde v\| \leq b + b = 2b$.
	In contrast to the proof in~\citep{jin2018agdescapes}, we have no Lipschitz guarantee for $\nabla^2 \hat f_x$ along the line segment connecting $\zeta_j$ and $\zeta_j'$, but \aref{assu:Mandf} still offers such guarantees along the line segments connecting the origin of $\T_x\calM$ to each of $\zeta_j$ and $\zeta_j'$. Thus, we can write:
	\begin{align*}
		\innersmall{\nabla^2 \hat f_x(\zeta_j')[\tilde v]}{\tilde v} & = \innersmall{\nabla^2 \hat f_x(\zeta_j)[\tilde v]}{\tilde v} + \innersmall{(\nabla^2 \hat f_x(\zeta_j') - \nabla^2 \hat f_x(0))[\tilde v]}{\tilde v} - \innersmall{(\nabla^2 \hat f_x(\zeta_j) - \nabla^2 \hat f_x(0))[\tilde v]}{\tilde v} \\
				& \leq -\gamma s^2 + \left( \|\nabla^2 \hat f_x(\zeta_j') - \nabla^2 \hat f_x(0)\| + \|\nabla^2 \hat f_x(\zeta_j) - \nabla^2 \hat f_x(0)\| \right) \|\tilde v\|^2 \\
				& \leq \left( -\gamma + \hat\rho (\|\zeta_j'\| + \|\zeta_j\|)  \right) s^2 \\
				& \leq \left( -\gamma + 2\hat\rho(s + \|s_j\|) \right) s^2,
	\end{align*}
	where on the last line we used $\zeta_j = ts_j + (1-t)u_j$,  $u_j = s_j + (1-\theta_j) v_j$, $\theta_j \in [0, 1]$ and $\|v_j\| \leq s$ to claim that $\|\zeta_j\| = \|s_j + (1-t)(1-\theta_j)v_j\| \leq \|s_j\| + \|v_j\| \leq \|s_j\| + s$, and also (more directly) that $\|\zeta_j'\| \leq \|s_j\| + \|\tilde v\| = \|s_j\| + s$.
	Plugging our findings into~\eqref{eq:NCEpfhatfxNCEfoo}, it follows that
	\begin{align}
		\hat f_x(\NCE(x, s_j, v_j)) \leq \hat f_x(s_j) - \frac{1}{2}\left(\gamma - 2\hat\rho(s + \|s_j\|) \right) s^2.
	\end{align}
	Since $\hat f_x(s_j) \leq E_j$ by definition~\eqref{eq:Ej}, the main part of the lemma's claim is now proved.
	
	We now turn to the last part of the lemma's claim, for which we further assume $\|s_j\| \leq \mathscr{L}$.
	Recall from Lemma~\ref{lem:params} that $\mathscr{L} \leq s$.
	We deduce from the main claim that
	\begin{align*}
		\hat f_x(\NCE(x, s_j, v_j)) \leq E_j - \min\!\left( \frac{s^2}{2\eta}, \frac{(\gamma - 4\hat\rho s)s^2}{2} \right).
	\end{align*}
	To conclude, use Lemma~\ref{lem:params} anew to bound the right-most term.
\end{proof}

\section{Supporting lemmas} \label{sec:supportlemmas}

In this section, we state and prove three additional lemmas about accelerated gradient descent in balls of tangent spaces that are useful for proofs in subsequent sections.
%The proofs involve only simple symbolic manipulations.
The statements apply more broadly than the setup of parameters and assumptions in Section~\ref{sec:assuparams}, but of course it is under those provisions that the conclusions are useful to us.

Throughout this section, we use the following notation.
For some $x \in \calM$, let $\mathcal{H} = \nabla^2 \hat{f}_x(0)$.
Given $s_0 \in \T_x\calM$, set $v_0 = 0$ and define for $j = 0, 1, 2, \ldots$:
\begin{align}\label{updateq}
	u_j & = s_j + (1-\theta)v_j, & s_{j+1} & = u_j - \eta \nabla \hat{f}_x(u_j) & \textrm{ and } & & v_{j+1} & = s_{j+1} - s_j
\end{align}
with some arbitrary $\theta \in [0, 1]$ and $\eta > 0$.
Also define $s_{-1} = s_0 - v_0$ for convenience and
\begin{align}
	\delta_k & = \nabla \hat{f}_x(u_k) - \nabla \hat{f}_x(0) - \mathcal{H}u_k, \nonumber\\
	\delta_k' & = \nabla \hat{f}_x(u_k) - \nabla \hat{f}_x(s_\tau) - \mathcal{H}(u_k-s_\tau),
	\label{eq:deltakdeltakprime}
\end{align}
where $\tau \geq 0$ is a fixed index.
Notice that iterates generated by $\TSS(x, s_0)$ with parameters and assumptions as laid out in Section~\ref{sec:assuparams} conform to this notation so long as $\theta_j = \theta$.
Owing to Lemma~\ref{lem:TSSballs}, the latter condition holds in particular if $\TSS$ runs all its iterations in full because if at any point $\theta_j \neq \theta$ then $\|s_{j+1}\| > b$ and $\TSS$ terminates early.
This is the setting in which we call upon lemmas from this section.

The first lemma is a variation on~\citep[Lem.~18]{jin2018agdescapes}.
%\TODO{Describe Jin's subspace problem in more depth.}
\begin{lemma}\label{Lemma13}
	With notation as above, for all $j \geq 0$ we can write 
	\begin{align}
		\begin{pmatrix} s_{\tau+j} \\ s_{\tau+j-1} \end{pmatrix} & = A^j \begin{pmatrix} s_{\tau} \\ s_{\tau-1} \end{pmatrix} - \eta \sum_{k = 0}^{j-1} A^{j-1-k} \begin{pmatrix} \nabla \hat f_x(0) + \delta_{\tau + k} \\ 0 \end{pmatrix}
		\label{eq:first_2}
	\end{align}
%	\begin{equation}
%	\begin{split}
%	\left( {\begin{array}{c}
%		s_{\tau+j}\\
%		s_{\tau+j-1} \\
%		\end{array} } \right)
%	= 
%	A^j \left( {\begin{array}{c}
%		s_{\tau}\\
%		s_{\tau-1} \\
%		\end{array} } \right)
%	- \eta \sum_{k=0}^{j-1} A^{j-1-k} 
%	\left( {\begin{array}{c}
%		\nabla \hat{f}_x(0) + \delta_{\tau+k}\\
%		0 \\
%		\end{array} } \right)
%	\end{split}
%	\end{equation}
	and
	\begin{align}
		\begin{pmatrix}
			s_{\tau+j} - s_{\tau} \\
			s_{\tau+j-1} - s_{\tau}
		\end{pmatrix} & = A^j \begin{pmatrix}
		0 \\
		-v_{\tau}
		\end{pmatrix} - \eta \sum_{k=0}^{j-1} A^{j-1-k} \begin{pmatrix}
			\nabla \hat{f}_x(s_{\tau}) + \delta_{\tau+k}' \\
			0
		\end{pmatrix}
		\label{eq:second}
	\end{align}
	where
	\begin{align}
		A & = \begin{pmatrix}
					(2-\theta)(I-\eta\mathcal{H}) & -(1-\theta)(I-\eta\mathcal{H}) \\
					I & 0
			  \end{pmatrix}.
	\label{matrixA}
	\end{align}
	%I think variable $\theta$ is going to cause a problem.  I.e., when $u_j$ escapes ball the current version of $\PARGD$ does not move back to the manifold but just increases $\theta$ to $\hat{\theta}$.  So at each iteration $j$ we have variable $\theta_j \in [\theta, 1]$.  Therefore,  the matrix $A_j$ might be different at each iteration, and unfolding is not as neat.  So instead, if $u_j$ is outside ball then try to just move back to the manifold.  
\end{lemma}
\begin{proof}
	By definition of $\delta_{\tau+j-1}$, we have $\nabla \hat{f}_x(u_{\tau+j-1}) = \nabla \hat{f}_x(0) + \mathcal{H}u_{\tau+j-1} + \delta_{\tau+j-1}$.
	Thus,
	\begin{align*}
		s_{\tau+j} & = u_{\tau+j-1} - \eta\nabla\hat{f}_x(u_{\tau+j-1}) \\
		           & = u_{\tau+j-1} - \eta\nabla\hat{f}_x(0) -\eta \mathcal{H}u_{\tau+j-1} - \eta \delta_{\tau+j-1} \\
			       & = (I -\eta \mathcal{H})u_{\tau+j-1} -\eta (\nabla \hat{f}_x(0) + \delta_{\tau+j-1}).
	\end{align*}
	Use the definitions of $u_{k}$ and $v_{k}$ to verify that $u_{k} = (2-\theta)s_{k} - (1-\theta)s_{k-1}$ (we use this several times in subsequent proofs).
	Plug this in the previous identity to see that
	\begin{align*}
		s_{\tau+j} & = (2-\theta) (I -\eta \mathcal{H}) s_{\tau+j-1} - (1-\theta) (I -\eta \mathcal{H}) s_{\tau+j-2} -\eta (\nabla \hat{f}_x(0) + \delta_{\tau+j-1}).
	\end{align*}
	Equivalently in matrix form, then reasoning by induction, it follows that
	\begin{equation*}
	\begin{split}
	\left( {\begin{array}{c}
		s_{\tau+j}\\
		s_{\tau+j-1} \\
		\end{array} } \right)
	&=
	\left( {\begin{array}{cc}
		(2-\theta)(I-\eta\mathcal{H}) & -(1-\theta)(I-\eta\mathcal{H}) \\
		I & 0 \\
		\end{array} } \right)
	\left( {\begin{array}{c}
		s_{\tau+j-1 }\\
		s_{\tau+j-2} \\
		\end{array} } \right) - \eta
	\left( {\begin{array}{c}
		\nabla \hat{f}_x(0) + \delta_{\tau+j-1}\\
		0 \\
		\end{array} } \right) \\
	&= 
	A^j \left( {\begin{array}{c}
		s_{\tau}\\
		s_{\tau-1} \\
		\end{array} } \right)
	- \eta \sum_{k=0}^{j-1} A^{j-1-k} 
	\left( {\begin{array}{c}
		\nabla \hat{f}_x(0) + \delta_{\tau+k}\\
		0
		\end{array} } \right).
	\end{split}
	\end{equation*}
	This verifies eq.~\eqref{eq:first_2}.  To prove eq.~\eqref{eq:second}, observe that \eqref{eq:first_2} together with
	\begin{align*}
		\delta_{\tau+k} & = \delta_{\tau+k}' + \nabla \hat f_x(s_\tau) - \nabla \hat f_x(0) - \mathcal{H}s_\tau
	\end{align*}
	and $s_{\tau-1} = s_\tau - v_\tau$ imply
	\begin{equation*}
	\begin{split}
	\left( {\begin{array}{c}
		s_{\tau+j}-s_{\tau}\\
		s_{\tau+j-1} -s_{\tau}
		\end{array} } \right)
	&= 
	A^j \left( {\begin{array}{c}
		0\\
		-v_{\tau} \\
		\end{array} } \right)
	- \eta \sum_{k=0}^{j-1} A^{j-1-k} 
	\left( {\begin{array}{c}
		\nabla \hat{f}_x(s_{\tau}) + \delta_{\tau+k}'\\
		0 \\
		\end{array} } \right) \\
	& \quad + (A^j - I) \left( {\begin{array}{c}
		s_{\tau}\\
		s_{\tau}
		\end{array} } \right)
	+ \sum_{k=0}^{j-1} A^{j-1-k} 
	\left( {\begin{array}{c}
		\eta\mathcal{H} s_{\tau}\\
		0 \\
		\end{array} } \right).
	\end{split}
	\end{equation*}
	The last two terms cancel.
	Indeed, let $M \triangleq \sum_{k=0}^{j-1} A^{j-1-k} = A^0 + \cdots + A^{j-1}$.
	Notice that $M(A-I) = MA - M = A^j - I$.
	Thus,
	\begin{equation*}
	\begin{split}
	&\sum_{k=0}^{j-1} A^{j-1-k} 
	\left( {\begin{array}{c}
		\eta\mathcal{H} s_{\tau}\\
		0 \\
		\end{array} } \right)
	+ (A^j - I) \left( {\begin{array}{c}
		s_{\tau}\\
		s_{\tau} \\
		\end{array} } \right)\\
	& \quad =
	M\Bigg[
	\left( {\begin{array}{cc}
		\eta\mathcal{H} & 0\\
		0 & \eta\mathcal{H} \\
		\end{array} } \right)
	\left( {\begin{array}{c}
		s_{\tau}\\
		0 \\
		\end{array} } \right)
	+ (A - I)\left( {\begin{array}{c}
		s_{\tau}\\
		s_{\tau} \\
		\end{array} } \right)\Bigg] \\ 
	& \quad =
	M\Bigg[
	\left( {\begin{array}{cc}
		0 & - \eta\mathcal{H}\\
		0 & \eta\mathcal{H} \\
		\end{array} } \right)
	\left( {\begin{array}{c}
		s_{\tau}\\
		0 \\
		\end{array} } \right)
	+ (A-I)\left( {\begin{array}{cc}
		-I & -I\\
		-I & -I \\
		\end{array} } \right)
	\left( {\begin{array}{c}
		s_{\tau}\\
		0 \\
		\end{array} } \right)
	+ (A - I)\left( {\begin{array}{c}
		s_{\tau}\\
		s_{\tau} \\
		\end{array} } \right)\Bigg]\\
	& \quad =
	M\Bigg[0\Bigg] = 0.
	\end{split}
	\end{equation*}
	To reach the second-to-last line, verify that $(A-I)\begin{pmatrix}
	-I & -I \\ -I & -I
	\end{pmatrix} = \begin{pmatrix}
	\eta\mathcal{H} & \eta\mathcal{H} \\ 0 & 0
	\end{pmatrix}$ using~\eqref{matrixA}.
	The last line follows by direct calculation.
\end{proof}

The lemma below is a direct continuation from the lemma above.
We use it only for the proof of Lemma~\ref{lem:socp}. % yes, I checked
\begin{lemma}\label{ForLemma18}
	Use notation from Lemma \ref{Lemma13}.  Given $s_0, s_0' \in \T_x\calM$, define two sequences $\{s_j, u_j, v_j\}$ and $\{s_j', u_j', v_j'\}$ by the update equations~\eqref{updateq}.
	Let $w_j = s_j - s_j'$.
	Then,
	\begin{align*}
		\begin{pmatrix} w_j \\ w_{j-1} \end{pmatrix} & = A^j \begin{pmatrix} w_0 \\ w_{-1} \end{pmatrix} - \eta \sum_{k = 0}^{j-1} A^{j-1-k} \begin{pmatrix} \delta_k'' \\ 0 \end{pmatrix}
	\end{align*}
	where $\delta_k'' = \nabla \hat{f}_x(u_k)-\nabla \hat{f}_x(u_k') - \mathcal{H} (u_k - u_k')$.
\end{lemma}
\begin{proof}
	By Lemma \ref{Lemma13} with $\tau = 0$, both of these identities hold:
	\begin{align*}
		\begin{pmatrix} s_{j} \\ s_{j-1} \end{pmatrix} & = A^j \begin{pmatrix} s_{0} \\ s_{-1} \end{pmatrix} - \eta \sum_{k = 0}^{j-1} A^{j-1-k} \begin{pmatrix} \nabla \hat f_x(u_k) - \mathcal{H} u_{k} \\ 0 \end{pmatrix}, \\
		\begin{pmatrix} s_{j}' \\ s_{j-1}' \end{pmatrix} & = A^j \begin{pmatrix} s_{0}' \\ s_{-1}' \end{pmatrix} - \eta \sum_{k = 0}^{j-1} A^{j-1-k} \begin{pmatrix} \nabla \hat f_x(u_k') - \mathcal{H} u_{k}' \\ 0 \end{pmatrix}.
	\end{align*}
	Taking the difference of these two equations reveals that
	\begin{align*}
		\begin{pmatrix} w_{j} \\ w_{j-1} \end{pmatrix} & = A^j \begin{pmatrix} w_{0} \\ w_{-1} \end{pmatrix} - \eta \sum_{k = 0}^{j-1} A^{j-1-k} \begin{pmatrix} \nabla \hat{f}_x(u_k)-\nabla \hat{f}_x(u_k') - \mathcal{H} (u_k - u_k') \\ 0 \end{pmatrix}.
	\end{align*}
	Conclude with the definition of $\delta_k''$.
\end{proof}

The next lemma corresponds to~\citep[Prop.~19]{jin2018agdescapes}.
The claim applies in particular to iterates generated by $\TSS$ with parameters and assumptions as laid out in Section~\ref{sec:assuparams} and $R \leq b$, so long as $\theta_j = \theta$ and the $s_j$ remain in the appropriate balls.
There are a few changes related to indexing and to the fact that our Lipschitz assumptions are limited to balls.
\begin{lemma}\label{Lemma14}
	%\TODO{also need to assume $\theta \in [0, 1]$}
	%\TODO{note: this lemma assumes constant $\theta$}
	Use notation from Lemma \ref{Lemma13}.
	Assume $\|\nabla^2\hat{f}_x(s) - \nabla^2\hat{f}_x(0)\| \leq \hat{\rho}\norm{s}$ for all $s \in B_x(3R)$ with some $R > 0$, $\hat \rho > 0$.
	Also assume $\norm{s_k} \leq R$ for all $k = q'-1, \ldots, q$.
	Then for all $k = q', \ldots, q$ we have $\norm{\delta_k} \leq 5\hat{\rho} R^2$.
	Moreover, for all $k = q'+1, \ldots, q$ we have
	\begin{align*}
		\norm{\delta_{k} - \delta_{k-1}} & \leq 6\hat{\rho} R\big(\norm{s_{k} - s_{k-1}} + \norm{s_{k-1} - s_{k-2}}\big).
	\end{align*}
%	\TODO{If additionally $\|v_{q'}\| \leq \|v_{q}\|$ then} we also have
	Additionally, we can bound their sum as:
	\begin{align*}
		\sum_{k=q'+1}^q \norm{\delta_k - \delta_{k-1}}^2 & \leq 144 \hat{\rho}^2R^2\sum_{k=q'}^q \norm{s_{k} - s_{k-1}}^2.
	\end{align*}
	(Mind the different ranges of summation.)
%	\TODO{Presumably, we could shift one of the two sums and that could have some beneficial effects. Maybe I missed a step, but it seems to match the notes.}
\end{lemma}
\begin{proof}
	Recall that $u_k = (2-\theta) s_k - (1-\theta)s_{k-1}$.
	In particular,
	\begin{align*}
		\|u_k\| & \leq |2-\theta| \|s_k\| + |1-\theta| \|s_{k-1}\| \leq 3R & \textrm{ for } && k & 	= q', \ldots, q.
	\end{align*}
	We use this to establish each of the three inequalities.
	
	First, by definition of $\mathcal{H} = \nabla^2 \hat f_x(0)$ and of $\delta_k$, we know that
	\begin{align*}
		\delta_k = \nabla \hat{f}_x(u_k) - \nabla \hat{f}_x(0) - \mathcal{H} u_k = \int_0^1 \nabla^2 \hat{f}_x(\phi u_k)[u_k] - \nabla^2 \hat f_x(0)[u_k]\dphi.
	\end{align*}
	Owing to $\|u_k\| \leq 3R$, we can use the Lipschitz properties of $\nabla^2 \hat f_x$ to find
	\begin{align*}
		\|\delta_k\| & \leq \int_0^1\norm{\nabla^2 \hat{f}_x(\phi u_k)-\nabla^2 \hat{f}_x(0)} \dphi\norm{u_k}\leq \frac{1}{2} \hat{\rho}\norm{u_k}^2 \leq \frac{9}{2} \hat\rho R^2.
	\end{align*}
%	\begin{equation*}
%	\begin{split}
%	\norm{\delta_k} & \leq \int_0^1\norm{\nabla^2 \hat{f}_x(\phi u_k)-\nabla^2 \hat{f}_x(0)} \dphi\norm{u_k} \\ & \leq \frac{1}{2} \hat{\rho}\norm{u_k}^2 = \frac{1}{2} \hat{\rho}\norm{(2-\theta)s_k- (1-\theta)s_{k-1}}^2 \\
%	& \leq \hat{\rho}\big(\norm{2s_k}^2 + \norm{s_{k-1}}^2\big) \leq 5 \hat{\rho}R^2.
%	\end{split}
%	\end{equation*}
	This shows the first inequality for $k = q', \ldots, q$.
	
	For the second inequality, first verify that
	\begin{align*}
		\norm{\delta_k - \delta_{k-1}} & = \norm{\nabla \hat{f}_x(u_k) - \nabla \hat{f}_x(u_{k-1}) - \nabla^2 \hat{f}_x(0) [u_k - u_{k-1}]} \\
					& = \norm{\bigg(\int_0^1\nabla^2 \hat{f}_x((1-\phi)u_{k-1} + \phi u_k) - \nabla^2 \hat{f}_x(0)\dphi\bigg)[u_k - u_{k-1}]}.
	\end{align*}
	Note that the distance between $(1-\phi)u_{k-1} + \phi u_k$ and the origin is at most $\max\{\norm{u_k},\norm{u_{k-1}}\}$ for all $\phi \in [0, 1]$.
	Since for $k = q'+1, \ldots q$ we have both $\|u_k\| \leq 3R$ and $\|u_{k-1}\| \leq 3R$, it follows that $\|(1-\phi)u_{k-1} + \phi u_k\| \leq 3R$ for all $\phi \in [0, 1]$.
	As a result, we can use the Lipschitz-like properties of $\nabla^2 \hat f_x$ and write:
	\begin{align*}
		\norm{\delta_k - \delta_{k-1}} & \leq 3\hat{\rho}R \norm{u_k - u_{k-1}}.
	\end{align*}
	Combine $u_k = (2-\theta) s_k - (1-\theta)s_{k-1}$ and $u_{k-1} = (2-\theta) s_{k-1} - (1-\theta)s_{k-2}$ to find $u_{k} - u_{k-1} = (2-\theta)(s_k - s_{k-1}) - (1-\theta) (s_{k-1} - s_{k-2})$.
	From there, it follows that
	\begin{align*}
		\norm{\delta_k - \delta_{k-1}} & \leq 3\hat{\rho} R \norm{(2-\theta)(s_k - s_{k-1}) - (1-\theta) (s_{k-1} - s_{k-2})} \\
									   & \leq 3\hat{\rho} R \left( 2 \|s_k - s_{k-1}\| + \|s_{k-1} - s_{k-2}\| \right) \\
									   & \leq 6\hat{\rho} R \left( \|s_k - s_{k-1}\| + \|s_{k-1} - s_{k-2}\| \right).
	\end{align*}
	This establishes the second inequality for $k = q'+1, \ldots q$.
	
	The third inequality follows from the second one through squaring and a sum, notably using $(a+b)^2 \leq 2(a^2 + b^2)$ for $a, b \geq 0$:
	\begin{align*}
		\sum_{k=q'+1}^q \norm{\delta_k - \delta_{k-1}}^2 & \leq 36 \hat{\rho}^2 R^2 \sum_{k=q'+1}^q \left( \|s_k - s_{k-1}\| + \|s_{k-1} - s_{k-2}\| \right)^2 \\
													     & \leq 72 \hat{\rho}^2 R^2 \sum_{k=q'+1}^q \left( \|s_k - s_{k-1}\|^2 + \|s_{k-1} - s_{k-2}\|^2 \right) \\
													     & = 72 \hat{\rho}^2 R^2 \left( \sum_{k=q'+1}^q \|s_k - s_{k-1}\|^2   +   \sum_{k=q'}^{q-1} \|s_k - s_{k-1}\|^2   \right).
%													     & = 144 \hat{\rho}^2 R^2 \sum_{k=q'+1}^q \|s_k - s_{k-1}\|^2 + 72\hat{\rho}^2 R^2\left(\|s_{q'} - s_{q'-1}\|^2 - \|s_q - s_{q-1}\|^2\right).
	\end{align*}
	To conclude, extend the ranges of both sums to $q', \ldots, q$.
\end{proof}

We close this supporting section with important remarks about the matrix $A$~\eqref{matrixA}, still following~\citep{jin2018agdescapes}.
Recall the notation $\mathcal{H} = \nabla^2 \hat f_x(0)$: this is an operator on $\T_x\calM$,  self-adjoint with respect to the Riemannian inner product on $\T_x\calM$.
Let $e_1, \ldots, e_d \in \T_x\calM$ form an orthonormal basis of eigenvectors of $\mathcal{H}$ associated to ordered eigenvalues $\lambda_1 \leq \cdots \leq \lambda_d$.
We think of $A$ as a linear operator to and from $\T_x\calM \times \T_x\calM$.
Conveniently, the eigenvectors of $\mathcal{H}$ reveal how to block-diagonalize $A$.
Indeed,
from
\begin{align*}
	A \begin{pmatrix}
	e_m \\ 0
	\end{pmatrix} & = \begin{pmatrix}
	(2-\theta)(I-\eta\mathcal{H}) & -(1-\theta)(I-\eta\mathcal{H}) \\
	I & 0
	\end{pmatrix} \begin{pmatrix}
	e_m \\ 0
	\end{pmatrix} = \begin{pmatrix}
	(2-\theta)(1-\eta\lambda_m) e_m \\ e_m
	\end{pmatrix}
\end{align*}
and
\begin{align*}
A \begin{pmatrix}
0 \\ e_m
\end{pmatrix} & = \begin{pmatrix}
(2-\theta)(I-\eta\mathcal{H}) & -(1-\theta)(I-\eta\mathcal{H}) \\
I & 0
\end{pmatrix} \begin{pmatrix}
0 \\ e_m
\end{pmatrix} = \begin{pmatrix}
-(1-\theta)(1-\eta\lambda_m) e_m \\ 0
\end{pmatrix}
\end{align*}
it is a simple exercise to check that
%\TODO{What follows in between here and the final statements of the section is only to handle powers of $A$; there might be a nice shortcut;}
\begin{align}
	J^* A J & = \diag\!\left( A_1, \ldots, A_d \right) & \textrm{ with } && J & = \begin{pmatrix}
	e_1 & 0   & e_2 & 0   & \cdots & e_d & 0   \\
	0   & e_1 & 0   & e_2 & \cdots & 0   & e_d
	\end{pmatrix} \nonumber\\
	& & \textrm{ and } && A_m & = \begin{pmatrix}
	(2-\theta)(1-\eta\lambda_m) & -(1-\theta)(1-\eta\lambda_m) \\
	1 & 0
	\end{pmatrix}. \label{eq:Am} % \textrm{ for } m = 1, \ldots d.
\end{align}
% Wolfram code to study the 2x2 matrix: eigenvalues (((2-\theta)(1-\eta\lambda), -(1-\theta)(1-\eta\lambda)), (1, 0))
Here, $J$ is a unitary operator from $\reals^{2d}$ (equipped with the standard Euclidean metric) to $\T_x\calM \times \T_x\calM$, and $J^*$ denotes its adjoint (which is also its inverse).
In particular, it becomes straightforward to investigate powers of $A$:
\begin{align}
	A^k & = \left( J \diag\!\left( A_1, \ldots, A_d \right) J^* \right)^k = J \diag\!\left( A_1^k, \ldots, A_d^k \right) J^*.
	\label{eq:Ak}
\end{align}
For $m, m'$ in $\{1, \ldots, d\}$ we have the useful identities
\begin{align}
	\inner{\begin{pmatrix} e_{m'} \\ 0 \end{pmatrix}}{A^k \begin{pmatrix} e_m \\ 0 \end{pmatrix}} & = \begin{cases}
		(A_m^k)_{11} & \textrm{ if } m = m', \\  %\begin{pmatrix} 1 & 0 \end{pmatrix} A_m^k \begin{pmatrix} 1 \\ 0 \end{pmatrix}
		0 & \textrm{ if } m \neq m',
	\end{cases}
	\label{eq:Akmm}
\end{align}
where $(A_m^k)_{11}$ is the top-left entry of the $2 \times 2$ matrix $(A_m)^k$.
Likewise,
\begin{align}
	\inner{\begin{pmatrix} e_{m'} \\ 0 \end{pmatrix}}{A^k \begin{pmatrix} 0 \\ e_m \end{pmatrix}} & = \begin{cases}
	(A_m^k)_{12} & \textrm{ if } m = m', \\  %\begin{pmatrix} 1 & 0 \end{pmatrix} A_m^k \begin{pmatrix} 1 \\ 0 \end{pmatrix}
	0 & \textrm{ if } m \neq m'.
	\end{cases}
	\label{eq:Akmmbis}
\end{align}
Additionally, one can also check that~\citep[Lem.~24]{jin2018agdescapes}:
\begin{align}
	\inner{\begin{pmatrix} 0 \\ e_{m'} \end{pmatrix}}{A^k \begin{pmatrix} e_m \\ 0 \end{pmatrix}} & = \begin{cases}
	(A_m^{k-1})_{11} & \textrm{ if } m = m', \\
	0 & \textrm{ if } m \neq m'.
	\end{cases}
	\label{eq:Akmmter}
\end{align}

\section{Proofs from Section~\ref{sec:focp} about $\ARGD$} \label{app:FOCP}

We include fulls proofs for the analogues of~\citep[Lem.~21 and~22]{jin2018agdescapes} because we need small but important changes for our setting (as is the case for the other similar results we prove in full), and because of (ultimately inconsequential) small issues with some arguments pertaining to the subspace $\mathcal{S}$ in the original proofs.
(Specifically, the subspace $\calS$ is defined with respect to the Hessian of the cost function at a specific reference point, which for notational convenience in \citet{jin2018agdescapes} is denoted by $0$; however, this same convention is used in several lemmas, on at least one occasion referring to distinct reference points; the authors easily proposed a fix, and we use a different fix below; to avoid ambiguities, we keep all iterate references explicit.)
Up to those minor changes, the proofs of the next two lemmas are due to Jin et al.

As a general heads-up for this and the next section: we call upon several lemmas from~\citep{jin2018agdescapes} which are purely algebraic facts about the entries of powers of the $2\times 2$ matrices $A_m$~\eqref{eq:Am}: they do not change at all for our context, hence we do not include their proofs.
We only note that Lemma~33 in~\citep{jin2018agdescapes} may not hold for all $x \in \left( \frac{\theta^2}{(2-\theta)^2}, \frac{1}{4} \right]$ as stated (there are some issues surrounding their eq.~(17)), but it is only used twice, both times with $x \in \left( \frac{2\theta^2}{(2-\theta)^2}, \frac{1}{4} \right]$: in that interval the lemma does hold.

\begin{proof}[Proof of Lemma~\ref{lem:focpA}]
	For contradiction, assume $E_{\tau-1} - E_{\tau + \mathscr{T}/4} < \mathscr{E}$.
	Then, Lemma~\ref{lem:TSSEj} implies that $E_{\tau-1} - E_{\tau + j} < \mathscr{E}$ for all $-1 \leq j \leq \mathscr{T}/4$.
	Over that range, Lemmas~\ref{lem:TSStraveldist} and~\ref{lem:params} tell us that % $\frac{\mathscr{L}^2}{16\sqrt{\kappa}\eta\mathscr{T}} = \mathscr{E}$
	\begin{align}
		\|s_{\tau + j} - s_{\tau}\|^2 \leq 16 \sqrt{\kappa} \eta |j| |E_{\tau} - E_{\tau+j}| < 4 \sqrt{\kappa} \eta \mathscr{T} \mathscr{E} = \frac{1}{4}\mathscr{L}^2.
		\label{eq:staujminstaucontradiction}
	\end{align}
	The remainder of the proof consists in showing that $\|s_{\tau + \mathscr{T}/4} - s_{\tau}\|$ is in fact larger than $\frac{1}{2}\mathscr{L}$.
	
	Starting now, consider $j = \mathscr{T}/4$. %We need this throughout so that the special scalars are well defined
	From~\eqref{eq:second} in Lemma~\ref{Lemma13}, we know that
	\begin{align*}
	\begin{pmatrix}
	s_{\tau+j} - s_{\tau} \\
	s_{\tau+j-1} - s_{\tau}
	\end{pmatrix} & = A^j \begin{pmatrix}
	0 \\
	-v_{\tau}
	\end{pmatrix} - \eta \sum_{k=0}^{j-1} A^{j-1-k} \begin{pmatrix}
	\nabla \hat{f}_x(s_{\tau}) + \delta_{\tau+k}' \\
	0
	\end{pmatrix}.
	\end{align*}
	Let $e_1, \ldots, e_d$ form an orthonormal basis of eigenvectors for $\mathcal{H} = \nabla^2 \hat f_x(0)$ with eigenvalues $\lambda_1 \leq \cdots \leq \lambda_d$.
	Expand $v_\tau$, $\nabla \hat{f}_x(s_{\tau})$ and $\delta_{\tau+k}'$ in that basis as:
	\begin{align}
		v_\tau & = \sum_{m = 1}^{d} v^{(m)} e_m, & \nabla \hat{f}_x(s_{\tau}) & = \sum_{m = 1}^{d} g^{(m)} e_m, & \delta_{\tau+k}' & = \sum_{m = 1}^{d} (\delta_{\tau+k}')^{(m)} e_m.
		\label{eq:expansionsfocpA}
	\end{align}
	Then,
	\begin{align*}
		\begin{pmatrix}
		s_{\tau+j} - s_{\tau} \\
		s_{\tau+j-1} - s_{\tau}
		\end{pmatrix} & = \sum_{m = 1}^{d} \left[ -v^{(m)} A^j \begin{pmatrix} 0 \\ e_m \end{pmatrix} - \eta \sum_{k=0}^{j-1} (g^{(m)} + (\delta_{\tau+k}')^{(m)}) A^{j-1-k} \begin{pmatrix}
		e_m \\
		0
		\end{pmatrix} \right].
	\end{align*}
	Owing to~\eqref{eq:Akmm} and~\eqref{eq:Akmmbis} which reveal how $A$ block-diagonalizes in the basis $e$, we can further write
	\begin{align*}
		\inner{\begin{pmatrix} e_{m} \\ 0 \end{pmatrix}}{\begin{pmatrix} s_{\tau+j} - s_{\tau} \\ s_{\tau+j-1} - s_{\tau} \end{pmatrix}} & = -v^{(m)} (A_m^j)_{12} - \eta \sum_{k=0}^{j-1} \left(g^{(m)} + (\delta_{\tau+k}')^{(m)}\right) (A_m^{j-1-k})_{11}.
	\end{align*}
	This reveals the expansion coefficients of $s_{\tau+j} - s_{\tau}$ in the basis $e_1, \ldots, e_d$, which is enough to study the norm of $s_{\tau+j} - s_{\tau}$.
	Explicitly,
	\begin{align}
		\|s_{\tau+j} - s_{\tau}\|^2 & = \sum_{m = 1}^d \left( v^{(m)} b_{m, j} - \eta \sum_{k=0}^{j-1} \left(g^{(m)} + (\delta_{\tau+k}')^{(m)}\right)a_{m, j-1-k} \right)^2,
		\label{eq:focpAfirstbound}
	\end{align}
	where we introduce the notation
	\begin{align}
		a_{m, t} & = (A_m^t)_{11}, & b_{m, t} & = -(A_m^t)_{12}.
		\label{eq:amtbmt}
	\end{align}

	To proceed, we need control over the coefficients $a_{m, t}$ and $b_{m, t}$, as provided by~\citep[Lem.~30]{jin2018agdescapes}.
	We explore this for $m$ in the set
	\begin{align*}
		S^c = \left\{ m : \eta\lambda_m \leq \frac{\theta^2}{(2-\theta)^2} \right\},
	\end{align*}
	that is, for the eigenvectors orthogonal to $\calS$.
	Under our general assumptions it holds that $\|\nabla^2 \hat f_x(0)\| \leq \ell$, so that $|\lambda_m| \leq \ell$ for all $m$.
	This ensures $\eta \lambda_m \in [-1/4, \theta^2/(2-\theta)^2]$ for $m \in S^c$.
	Recall that $A_m$~\eqref{eq:Am} is a $2 \times 2$ matrix which depends on $\theta$ and $\eta\lambda_m$.
	It is reasonably straightforward to diagonalize $A_m$ (or rather, to put it in Jordan normal form), and from there to get an explicit expression for any entry of $A_m^k$.
	The quantity $\sum_{k = 0}^{j-1} a_{m,k}$ is a sum of such entries over a range of powers: this can be controlled as one would a geometric series.
	In~\citep[Lem.~30]{jin2018agdescapes}, it is shown that, for $m \in S^c$, if $j \geq 1 + 2/\theta$ and $\theta \in (0, 1/4]$, then
	\begin{align}
		\sum_{k = 0}^{j-1} a_{m,k} & \geq \frac{1}{c_4 \theta^2} & \textrm{ and } && \frac{b_{m,j}}{\sum_{k' = 0}^{j-1} a_{m,k'}} & \leq c_5^{1/2} \max\!\left( \theta, \sqrt{|\eta\lambda_m|} \right),
		\label{eq:sumsamk}
	\end{align}
	with some universal constants $c_4, c_5$.
	The lemma applies because $\theta \in (0, 1/4]$ by Lemma~\ref{lem:params} and also $j = \mathscr{T}/4 = \chi (c/48) \cdot 3/ \theta \geq 3/\theta \geq 4\sqrt{\kappa} + 2/\theta \geq 1 + 2/\theta$, with $c \geq 48$.
	
	Building on the latter comments, we can define the following scalars for $m \in S^c$:
	\begin{align*}
		p_{m,k,j} & = \frac{a_{m, j-1-k}}{\sum_{k' = 0}^{j-1} a_{m, k'}}, &
		q_{m,j}   & = -\frac{b_{m, j}}{\eta \sum_{k' = 0}^{j-1} a_{m, k'}}, \\
		\tilde\delta_j'^{(m)} & = \sum_{k = 0}^{j-1} p_{m,k,j} (\delta_{\tau+k}')^{(m)} &
		\tilde v_j^{(m)} & = q_{m,j} v^{(m)}.
	\end{align*}
	In analogy with notation in~\eqref{eq:expansionsfocpA}, we also consider vectors $\tilde\delta_j'$ and $\tilde v_j$ with expansion coefficients as above.
	These definitions are crafted specifically so that~\eqref{eq:focpAfirstbound} yields: % yes, I checked; note: the sign inside the square changed globally.
	\begin{align*}
		\|s_{\tau+j} - s_{\tau}\|^2 & \geq \sum_{m \in S^c} \left( \eta \left(\sum_{k = 0}^{j-1} a_{m,k} \right) \left( g^{(m)} + \tilde\delta_j'^{(m)} + \tilde v_j^{(m)} \right) \right)^2.
	\end{align*}
	We deduce from~\eqref{eq:sumsamk} that % (still with $j = \mathscr{T}/4$),
	\begin{align}
		\|s_{\tau+j} - s_{\tau}\| & \geq \frac{\eta}{c_4 \theta^2} \sqrt{\sum_{m \in S^c} \left( g^{(m)} + \tilde\delta_j'^{(m)} + \tilde v_j^{(m)} \right)^2} \nonumber\\
								  & = \frac{\eta}{c_4 \theta^2} \left\|P_{\calS^c}\!\left(\nabla \hat{f}_x(s_{\tau}) + \tilde\delta_j' + \tilde v_j \right)\right\| \nonumber\\
								  & \geq \frac{\eta}{c_4 \theta^2} \left( \frac{\epsilon}{6} - \|P_{\calS^c}(\tilde\delta_j')\| - \|P_{\calS^c}(\tilde v_j )\| \right),
								  \label{eq:staujminstaubound}
	\end{align}
	where $\calS^c$ is the orthogonal complement of $\calS$, that is, it is the subspace of $\T_x\calM$ spanned by eigenvectors $\{e_m\}_{m \in S^c}$, and $P_{\calS^c}$ is the orthogonal projector to $\calS^c$.
	In the last line, we used a triangular inequality and the assumption that $\|P_{\calS^c}(\nabla \hat{f}_x(s_{\tau}))\| \geq \epsilon/6$.
	Our goal now is to show that $\|P_{\calS^c}(\tilde\delta_j')\|$ and $\|P_{\calS^c}(\tilde v_j )\|$ are suitably small.
	
	Consider the following vector with notation as in~\eqref{eq:deltakdeltakprime}:
	\begin{align*}
		\Delta = \delta_{\tau + k} - \delta_{\tau + k}' & = \nabla \hat f_x(s_\tau) - \nabla \hat f_x(0) - \nabla^2 \hat f_x(0)[s_\tau] 
				  = \left( \int_{0}^{1} \nabla^2 \hat f_x(\phi s_\tau) - \nabla^2 \hat f_x(0) \dphi \right)\![s_\tau].
	\end{align*}
	By the Lipschitz-like properties of $\nabla^2 \hat f_x$ and the assumption $\|s_\tau\| \leq \mathscr{L} < b$, we deduce that
	\begin{align*}
		\|\Delta\| & \leq \frac{1}{2} \hat\rho \|s_\tau\|^2 \leq \frac{1}{2} \hat\rho \mathscr{L}^2.
	\end{align*}
	Note that $\sum_{k = 0}^{j-1} p_{m, k, j} = 1$. This and the fact that $\Delta$ is independent of $k$ justify that:
	\begin{align*}
		\|P_{\calS^c}(\tilde\delta_j')\|^2 = \sum_{m \in S^c} \left( \tilde\delta_j'^{(m)} \right)^2 & = \sum_{m \in S^c} \left( \sum_{k = 0}^{j-1} p_{m,k,j} (\delta_{\tau+k}')^{(m)} \right)^2 \\
				& = \sum_{m \in S^c} \left( \sum_{k = 0}^{j-1} p_{m,k,j} \left( (\delta_{\tau+k})^{(m)} - \Delta^{(m)} \right) \right)^2 \\
				& = \sum_{m \in S^c} \left( \sum_{k = 0}^{j-1} p_{m,k,j} (\delta_{\tau+k})^{(m)} - \Delta^{(m)} \right)^2,
	\end{align*}
	where $\Delta^{(m)}$ denotes the expansion coefficients of $\Delta$ in the basis $e$.
	Define the vector $\tilde\delta_j$ (without ``prime'') with expansion coefficients $\tilde\delta_j^{(m)} = \sum_{k = 0}^{j-1} p_{m,k,j} (\delta_{\tau+k})^{(m)}$.
	Then, by construction,
	\begin{align*}
		\|P_{\calS^c}(\tilde\delta_j')\| = \|P_{\calS^c}( \tilde\delta_j - \Delta )\| \leq \|P_{\calS^c}( \tilde\delta_j )\| + \|\Delta\| \leq \|P_{\calS^c}( \tilde\delta_j )\| + \hat\rho \mathscr{L}^2.
	\end{align*}
	Through a simple reasoning using~\citep[Lem.~24, 26]{jin2018agdescapes} one can conclude that, under our setting, both eigenvalues of $A_m$ (for $m \in S^c$) are positive, and as a result that the coefficients $a_{m,k}$ (hence also $p_{m, k, j}$) are positive.
	%
	% Chris' e-mail about positivity of the "a" coeffs on June 12, 2020:
	%
	% "We have an eigenvalue mu (which I'll denote as u) satisfying u^2 - (2-t)(1-x)u + (1-t)(1-x) = 0 (Lemma 26), where t denotes theta.
	%  We assume $-1/4 <= x <= t^2/(2-t)^2$. We also know that $0 < t <= 1/4$.
	%  By the quadratic formula, u = 0.5 * ( b +/- sqrt[ Delta ] ) where b = (2-t)(1-x) and Delta = b^2 - 4*(1-t)(1-x).
	%  As Jin observed, Delta >= 0 (Lemma 26).
	%  Note that b is positive because t \in [0, 1] and x <= .25.
	%  Thus, 0.5 * ( b + sqrt[ Delta ] ) is positive.
	%  Also note that  0.5 * ( b - sqrt[ Delta ] ) is positive since 0 <= Delta < b^2."
	%
	% Combine this with the fact that if the eigenvalues mu_1, mu_2 are positive, then the "a" coeffs are positive (Lemma 24).
	%
	Therefore,
	\begin{align*}
		\|P_{\calS^c}(\tilde\delta_j)\|^2 & = \sum_{m \in S^c} \left( \sum_{k = 0}^{j-1} p_{m,k,j} (\delta_{\tau+k})^{(m)} \right)^2 \\
										  & \leq \sum_{m \in S^c} \left( \sum_{k = 0}^{j-1} p_{m,k,j} \left( |(\delta_{\tau})^{(m)}| + |(\delta_{\tau+k})^{(m)} - (\delta_{\tau})^{(m)}| \right) \right)^2.
	\end{align*}
	Notice that for all $0 \leq k \leq j-1$ we have
	\begin{align*}
		|(\delta_{\tau+k})^{(m)} - (\delta_{\tau})^{(m)}| \leq \sum_{k' = 1}^{k} |(\delta_{\tau+k'})^{(m)} - (\delta_{\tau+k'-1})^{(m)}| \leq \sum_{k' = 1}^{j-1} |(\delta_{\tau+k'})^{(m)} - (\delta_{\tau+k'-1})^{(m)}|,
	\end{align*}
	and this right-hand side is independent of $k$.
	Thus, we can factor out $\sum_{k = 0}^{j-1} p_{m,k,j} = 1$ in the expression above to get:
	\begin{align*}
		\|P_{\calS^c}(\tilde\delta_j)\|^2 & \leq \sum_{m \in S^c} \left( |(\delta_{\tau})^{(m)}| + \sum_{k = 1}^{j-1} |(\delta_{\tau+k})^{(m)} - (\delta_{\tau+k-1})^{(m)}| \right)^2.
	\end{align*}
	Use first $(a + b)^2 \leq 2a^2 + 2b^2$ then (another) Cauchy--Schwarz to deduce
	\begin{align*}
		\|P_{\calS^c}(\tilde\delta_j)\|^2 & \leq 2 \sum_{m \in S^c} |(\delta_{\tau})^{(m)}|^2   +  2(j-1) \sum_{m \in S^c} \sum_{k = 1}^{j-1} |(\delta_{\tau+k})^{(m)} - (\delta_{\tau+k-1})^{(m)}|^2 \\
										  & \leq 2 \|\delta_\tau\|^2 + 2j \sum_{k = 1}^{j-1} \|\delta_{\tau+k} - \delta_{\tau+k-1} \|^2.
	\end{align*}
	To bound this further, we call upon Lemma~\ref{Lemma14} with $R = \frac{3}{2} \mathscr{L} \leq \frac{1}{3}b$, $q' = \tau$ and $q = \tau + \frac{\mathscr{T}}{4} - 1$.
	To this end, we must first verify that $\|s_{\tau + k}\| \leq R$ for $k = -1, \ldots, \frac{\mathscr{T}}{4} - 1$.
	This is indeed the case owing to~\eqref{eq:staujminstaucontradiction} and the assumption $\|s_\tau\| \leq \mathscr{L}$:
	\begin{align*}
		\|s_{\tau+k}\| & \leq \|s_{\tau+k} - s_\tau\| + \|s_\tau\| \leq \frac{1}{2} \mathscr{L} + \mathscr{L} = R & \textrm{ for } & & k = -1, \ldots, \mathscr{T}/4.
	\end{align*}
	This confirms that we can use the conclusions of Lemma~\ref{Lemma14}, reaching:
	\begin{align*}
	\|P_{\calS^c}(\tilde\delta_j)\|^2 & \leq 50 \hat\rho^2 R^4 + 288 \hat\rho^2 R^2 \cdot j\sum_{k = 0}^{j-1} \|s_{\tau+k} - s_{\tau+k-1} \|^2 \\
				& = \frac{4050}{16} \hat\rho^2 \mathscr{L}^4 + 648 \hat\rho^2 \mathscr{L}^2 \cdot j \sum_{k = \tau-1}^{\tau+j-2} \|s_{k+1} - s_{k} \|^2 \\
				& \leq 256 \hat\rho^2 \mathscr{L}^4 + 648 \hat\rho^2 \mathscr{L}^2 \cdot 16 \sqrt{\kappa} \eta j (E_{\tau-1} - E_{\tau+j-2}),
	\end{align*}
	where the first and last lines follow from the definition of $R$ and from Lemma~\ref{lem:TSStraveldist}, respectively.
	Recall that we assume $E_{\tau-1} - E_{\tau + \mathscr{T}/4} < \mathscr{E}$ for contradiction.
	Then, monotonic decrease of the Hamiltonian (Lemma~\ref{lem:TSSEj}) tells us that $E_{\tau-1} - E_{\tau+j-2} < \mathscr{E}$ for $0 \leq j \leq \mathscr{T}/4$.
	Combining with $16 \sqrt{\kappa} \eta \mathscr{T} \mathscr{E} = \mathscr{L}^2$ (Lemma~\ref{lem:params}), we find:
	\begin{align*}
		\|P_{\calS^c}(\tilde\delta_j)\|^2 & \leq 256 \hat\rho^2 \mathscr{L}^4 + 162 \hat\rho^2 \mathscr{L}^4 = 418 \hat\rho^2 \mathscr{L}^4.
	\end{align*}
	Thus, $\|P_{\calS^c}(\tilde\delta_j)\| \leq 21 \hat\rho \mathscr{L}^2 = 84 \epsilon \chi^{-4} c^{-6} \leq \epsilon/24$ with $c \geq 4$ and $\chi \geq 1$, for $0 \leq j \leq \mathscr{T}/4$.
	
	Recall that we aim to make progress from bound~\eqref{eq:staujminstaubound}.
	The bound $\|P_{\calS^c}(\tilde\delta_j)\| \leq \epsilon/24$ we just established is a first step.
	We now turn to bounding $\|P_{\calS^c}(\tilde v_j )\|$.
	Owing to~\eqref{eq:sumsamk}, we have this first bound assuming $j = \mathscr{T}/4$:
	\begin{align}
		\|P_{\calS^c}(\tilde v_j )\|^2 & = \sum_{m \in S^c} q_{m,j}^2 (v^{(m)})^2 \nonumber\\
									   & = \sum_{m \in S^c} \left(\frac{b_{m, j}}{\eta \sum_{k' = 0}^{j-1} a_{m, k'}}\right)^2 (v^{(m)})^2 \leq \frac{c_5}{\eta^2} \sum_{m \in S^c} (v^{(m)})^2 \max\!\left( \theta^2, |\eta\lambda_m| \right).
	   \label{eq:psctildevj}
	\end{align}
	(Recall from~\eqref{eq:expansionsfocpA} that $v^{(m)}$ denotes the coefficients of $v_\tau$ in the basis $e_1, \ldots, e_d$.)
	We split the sum in order to resolve the max.
	To this end, note that $\theta \in [0, 1]$ implies $\theta^2 \geq \frac{\theta^2}{(2-\theta)^2}$, so that the $\max$ evaluates to $\theta^2$ exactly when $-\theta^2 \leq \eta\lambda_m \leq \frac{\theta^2}{(2-\theta)^2}$ (remembering that $\eta\lambda_m \leq \frac{\theta^2}{(2-\theta)^2}$ because $m \in S^c$).
	Thus,
	\begin{multline*}
		\sum_{m \in S^c} (v^{(m)})^2 \max\!\left( \theta^2, |\eta\lambda_m| \right) = \sum_{m : -\theta^2 \leq \eta\lambda_m \leq \frac{\theta^2}{(2-\theta)^2}} (v^{(m)})^2 \theta^2 - \sum_{m : \eta\lambda_m < -\theta^2} (v^{(m)})^2 \eta\lambda_m.
	\end{multline*}
	Let us rework the last sum (we get a first bound by extending the summation range, exploiting that the summands are nonpositive):
	\begin{align*}
		-\sum_{m : \eta\lambda_m < -\theta^2} (v^{(m)})^2 \eta\lambda_m & \leq -\sum_{m : \eta\lambda_m \leq 0} (v^{(m)})^2 \eta\lambda_m \\
				 & = \sum_{m : \eta\lambda_m > 0} (v^{(m)})^2 \eta\lambda_m - \sum_{m = 1}^d (v^{(m)})^2 \eta\lambda_m \\
				 & = \sum_{m : \eta\lambda_m > 0} (v^{(m)})^2 \eta\lambda_m - \eta\inner{v_\tau}{\mathcal{H} v_\tau} \\
				 & = \sum_{m : 0 < \eta\lambda_m \leq \frac{\theta^2}{(2-\theta)^2}} (v^{(m)})^2 \eta\lambda_m + \eta \inner{P_\calS v_\tau}{\mathcal{H}P_\calS v_\tau} - \eta \inner{v_\tau}{\mathcal{H} v_\tau} \\
				 & \leq \theta^2 \|v_\tau\|^2 + \eta \inner{P_\calS v_\tau}{\mathcal{H}P_\calS v_\tau} - \eta \inner{v_\tau}{\mathcal{H} v_\tau}.
	\end{align*}
	(Recall that $P_\calS$ projects to the subspace spanned by eigenvectors with eigenvalues strictly above $\frac{\theta^2}{\eta(2-\theta)^2}$.)
	Combining all work done since~\eqref{eq:psctildevj}, it follows that
	\begin{align*}
		\|P_{\calS^c}(\tilde v_j )\|^2 & \leq \frac{c_5}{\eta^2} \left( 2\theta^2 \|v_\tau\|^2 + \eta \inner{P_\calS v_\tau}{\mathcal{H}P_\calS v_\tau} - \eta \inner{v_\tau}{\mathcal{H} v_\tau} \right).
	\end{align*}
	Use assumptions $\|v_\tau\| \leq \mathscr{M}$ and $\inner{P_\calS v_\tau}{\mathcal{H}P_\calS v_\tau} \leq \sqrt{\hat \rho \epsilon} \mathscr{M}^2$ to see that
	\begin{align}
		\|P_{\calS^c}(\tilde v_j )\|^2 & \leq \frac{c_5}{\eta^2} \left( 2\theta^2 \mathscr{M}^2 + \eta \sqrt{\hat \rho \epsilon} \mathscr{M}^2 - \eta \inner{v_\tau}{\mathcal{H} v_\tau} \right) \nonumber\\
						& = 4\ell c_5 \left( \frac{3}{2} \sqrt{\hat \rho \epsilon} \mathscr{M}^2 - \inner{v_\tau}{\mathcal{H} v_\tau} \right).
						\label{eq:PcsSctildevj}
	\end{align}
	(For the last equality, use $2\theta^2 = \frac{\sqrt{\hat\rho \epsilon}}{2} \eta$ and $\eta = 1/4\ell$.)
	To proceed, we must bound $\inner{v_\tau}{\mathcal{H} v_\tau}$.
	To this end, notice that by assumption the \eqref{eq:NCC} condition did not trigger for $(x, s_\tau, u_\tau)$.
	Therefore, we know that
	\begin{align*}
		\hat{f}_x(s_\tau) & \geq \hat{f}_x(u_\tau) + \innersmall{\nabla \hat{f}_x(u_\tau)}{s_\tau - u_\tau} - \frac{\gamma}{2}\norm{s_\tau - u_\tau}^2. % \\
%						  & = \hat{f}_x(u_\tau) - (1-\theta) \innersmall{\nabla \hat{f}_x(u_\tau)}{v_\tau} - (1-\theta)^2 \frac{\gamma}{2}\norm{v_\tau}^2.
	\end{align*}
	Moreover, it always holds that
	\begin{align*}
		\hat{f}_x(s_\tau) & = \hat f_x(u_\tau) + \innersmall{\nabla \hat{f}_x(u_\tau)}{s_\tau - u_\tau} + \frac{1}{2} \innersmall{s_\tau - u_\tau}{\nabla^2 \hat f_x(\phi s_\tau + (1-\phi)u_\tau)[s_\tau - u_\tau]}
	\end{align*}
	for some $\phi \in [0, 1]$.
	Also using $u_\tau = s_\tau + (1-\theta) v_\tau$, we deduce that
	\begin{align*}
		\innersmall{v_\tau}{\nabla^2 \hat f_x(\phi s_\tau + (1-\phi)u_\tau)[v_\tau]} \geq -\gamma \|v_\tau\|^2.
	\end{align*}
	With the help of Lemma~\ref{lem:params}, note that
	\begin{align*}
		\|\phi s_\tau + (1-\phi)u_\tau\| & = \|s_\tau + (1-\phi)(1-\theta) v_\tau\| \leq \|s_\tau\| + \|v_\tau\| \leq \mathscr{L} + \mathscr{M} \leq b.
	\end{align*}
	Thus, the Lipschitz-type properties of $\nabla^2 \hat f_x$ apply up to that point and we get
	\begin{align*}
		\|\nabla^2 \hat f_x(\phi s_\tau + (1-\phi)u_\tau) - \mathcal{H}\| \leq \hat\rho (\mathscr{L} + \mathscr{M}) \leq \sqrt{\hat\rho \epsilon}.
	\end{align*}
	Since $\gamma = \frac{\sqrt{\hat\rho \epsilon}}{4}$, it follows overall that
	\begin{align*}
		\innersmall{v_\tau}{\mathcal{H}v_\tau} \geq -\frac{5}{4} \sqrt{\hat\rho \epsilon} \|v_\tau\|^2 \geq -\frac{5}{4}\sqrt{\hat\rho \epsilon} \mathscr{M}^2.
	\end{align*}
	Plugging this back into~\eqref{eq:PcsSctildevj} with $c \geq 80\sqrt{c_5}$ reveals that
	\begin{align*}
		\|P_{\calS^c}(\tilde v_j )\|^2 & \leq 11 \ell c_5 \sqrt{\hat \rho \epsilon} \mathscr{M}^2 = 11 c_5 \epsilon^2 c^{-2} \leq \epsilon^2 / 24^2.
	\end{align*}
	This shows that $\|P_{\calS^c}(\tilde v_j )\| \leq \epsilon / 24$ for $j = \mathscr{T}/4$.
	
	We plug $\|P_{\calS^c}(\tilde\delta_j)\| \leq \epsilon/24$ and $\|P_{\calS^c}(\tilde v_j )\| \leq \epsilon / 24$  into~\eqref{eq:staujminstaubound} to state that, with $j = \mathscr{T} / 4$,
	\begin{align*}
		\|s_{\tau+j} - s_{\tau}\| & \geq \frac{\eta}{c_4 \theta^2} \left( \frac{\epsilon}{6} - \frac{\epsilon}{24} - \frac{\epsilon}{24} \right) = \frac{\eta \epsilon}{12 c_4 \theta^2} = \frac{1}{3c_4}\sqrt{\frac{\epsilon}{\hat\rho}} > \sqrt{\frac{\epsilon}{\hat\rho}} \chi^{-2} c^{-3} = \mathscr{L} / 2.
	\end{align*}
	(We used $4\theta^2 = \sqrt{\hat\rho \epsilon} \eta$, then we also set $c > (3c_4)^{1/3}$.)
	This last inequality contradicts~\eqref{eq:staujminstaucontradiction}.
	Thus, the proof by contradiction is complete and we conclude that $E_{\tau-1} - E_{\tau + \mathscr{T}/4} \geq \mathscr{E}$.
\end{proof}

What follows is the equivalent of the proof of~\citep[Lem.~22]{jin2018agdescapes}, with the small changes needed for our purpose.
\begin{proof}[Proof of Lemma~\ref{lem:focpB}]
	Since $E_0 - E_{\mathscr{T}/2} \leq \mathscr{E}$ and $s_0 = 0$, Lemmas~\ref{lem:TSSEj}, \ref{lem:TSStraveldist} and~\ref{lem:params} yield:
	\begin{align}
		\forall j \leq \mathscr{T}/2, && \|s_j\| & = \|s_j - s_0\| \leq \sqrt{8 \sqrt{\kappa} \eta \mathscr{T} \mathscr{E}} = \frac{\mathscr{L}}{\sqrt{2}} \leq \mathscr{L} \leq b.
		\label{eq:lemfocpBstarters}
	\end{align}
	By Lemma~\ref{Lemma13} with $\tau = 0$ and noting that $s_0 = 0$, $s_{-1} = s_0 - v_0 = 0$, we know that, for all $j$,
	\begin{align}
		\begin{pmatrix} s_{j} \\ s_{j-1} \end{pmatrix} & = - \eta \sum_{k = 0}^{j-1} A^{j-1-k} \begin{pmatrix} \nabla \hat f_x(0) + \delta_{k} \\ 0 \end{pmatrix}.
		\label{eq:focpBsjsjminone}
	\end{align}
	Define the operator $\Delta_j = \int_{0}^{1} \nabla^2 \hat f_x(\phi s_j) - \mathcal{H} \dphi$ with $\mathcal{H} = \nabla^2 \hat f_x(0)$.
	We can write:
	\begin{align}
		P_\calS \nabla \hat f_x(s_j) & = P_\calS\!\left( \nabla \hat f_x(0) + \mathcal{H}s_j + \Delta_js_j \right).
		\label{eq:PcalSnablahatfxsj}
	\end{align}
	We shall bound this term by term.
	
	The third term is straightforward, so let us start with this one.
	Owing to~\eqref{eq:lemfocpBstarters}, the Lipschitz-like properties of the Hessian apply to claim $\|\Delta_j\| \leq \frac{1}{2}\hat\rho \|s_j\|$.
	Therefore,
	\begin{align}
		\|P_\calS \Delta_js_j\| \leq \|\Delta_j\| \|s_j\| \leq \frac{1}{2}\hat\rho \|s_j\|^2 \leq \frac{1}{2}\hat\rho \mathscr{L}^2 = 2\epsilon \chi^{-4} c^{-6} \leq \epsilon/18
		\label{eq:blablatermthree}
	\end{align}
	with $c \geq 2$ and $\chi \geq 1$.
	Below, we work toward bounding the other two terms.
	
	As we did in the proof of Lemma~\ref{lem:focpA}, let $e_1, \ldots, e_d$ form an orthonormal basis of eigenvectors for $\mathcal{H}$ with eigenvalues $\lambda_1 \leq \cdots \leq \lambda_d$.
	Expand $\nabla \hat f_x(0)$ and $\delta_{k}$ in that basis as
	\begin{align*}
		\nabla \hat f_x(0) & = \sum_{m = 1}^{d} g^{(m)} e_m, & \delta_{k} & = \sum_{m = 1}^{d} \delta_k^{(m)} e_m.
	\end{align*}
	From~\eqref{eq:focpBsjsjminone} and~\eqref{eq:Akmm} it follows that
	\begin{align*}
		s_j & = \sum_{m' = 1}^{d} \inner{e_{m'}}{s_j} e_{m'} = -\eta \sum_{m' = 1}^{d} \sum_{k = 0}^{j-1} \sum_{m = 1}^d \inner{\begin{pmatrix} e_{m'} \\ 0 \end{pmatrix}}{A^{j-1-k} \begin{pmatrix} e_m \\ 0 \end{pmatrix}} (g^{(m)} + \delta_k^{(m)}) e_{m'} \\ & = -\eta \sum_{k = 0}^{j-1} \sum_{m = 1}^d (A_m^{j-1-k})_{11} (g^{(m)} + \delta_k^{(m)}) e_{m}.
	\end{align*}
	Motivated by~\eqref{eq:PcalSnablahatfxsj} and reusing notation $a_{m, j-1-k} = (A_m^{j-1-k})_{11}$ as in~\eqref{eq:amtbmt}, we further write
	\begin{align}
		P_\calS\!\left(\nabla \hat f_x(0) + \mathcal{H}s_j\right) & = \sum_{m \in S} \left[ g^{(m)} - \eta\lambda_m \sum_{k = 0}^{j-1} a_{m, j-1-k} (g^{(m)} + \delta_k^{(m)})  \right] e_m \nonumber \\
			& = \sum_{m \in S} \left[ \left( 1 -  \eta\lambda_m \sum_{k = 0}^{j-1} a_{m, k} \right) g^{(m)} - \eta\lambda_m \sum_{k = 0}^{j-1} a_{m, j-1-k} \delta_k^{(m)} \right] e_m,
			\label{eq:yetanothername}
	\end{align}
	where $S = \big\{ m : \eta \lambda_m > \frac{\theta^2}{(2-\theta)^2} \big\}$ indexes the eigenvalues of the eigenvectors which span $\calS$.
	This identity splits in two parts, each of which we now aim to bound.
	
	In the spirit of the comments surrounding~\eqref{eq:sumsamk}, here too it is possible to control the coefficients $a_{m,k}$ and $b_{m,k}$ (both defined as in~\eqref{eq:amtbmt}), this time for $m \in S$.
	Specifically, combining~\citep[Lem.~25]{jin2018agdescapes} with an identity in the proof of~\citep[Lem.~29]{jin2018agdescapes}, we see that
	\begin{align}
		1 - \eta \lambda_m \sum_{k = 0}^{j-1} a_{m, k} & = a_{m, j} - b_{m, j}.
		\label{eq:amjbmjminus}
	\end{align}
	Moreover, owing to~\citep[Lem.~32]{jin2018agdescapes} we know that
	\begin{align}
		\forall j \geq 0, \forall m \in S, && \max(|a_{m, j}|, |b_{m, j}|) \leq (j+1)(1-\theta)^{j/2}.
		\label{eq:amjbmjbound}
	\end{align}
	Thus, the first part of~\eqref{eq:yetanothername} is bounded as:
	\begin{align*}
		\left\|\sum_{m \in S} \left( 1 -  \eta\lambda_m \sum_{k = 0}^{j-1} a_{m, k} \right) g^{(m)} e_m \right\|^2 & = \sum_{m \in S} (a_{m,j} - b_{m,j})^2 (g^{(m)})^2 \leq 4 (j+1)^2 (1-\theta)^{j} \|\nabla \hat f_x(0)\|^2.
	\end{align*}
	One can show using $\theta \in (0, 1/4]$, $\chi \geq \log_2(\theta^{-1})$ and $c \geq 256$ (which we all assume) that
	\begin{align}
		\forall j \geq \mathscr{T} / 4, && (j+1)^2 & \leq (1-\theta)^{-j/2}.
		\label{eq:jpowerbound}
	\end{align}
	%
	% Details: all rechecked on June 16, 2020 for the gazilionth time: it holds.
	%
	% Differentiating with respect to $j$, we see that $(j+1)^{-2}(1-\theta)^{-j/2}$ is increasing for $j \geq -1-\frac{4}{\log({1-\theta})}$. -- True: https://www.wolframalpha.com/input/?i=differentiate+%28j%2B1%29%5E%28-2%29*%281-t%29%5E%7B-j%2F2%7D+with+respect+to+j
	% Observe that both of these hold for theta in (0, 1/4], and are decreasing in theta:
	%    $\frac{16}{\theta}\log_{2}(\frac{1}{\theta}) \geq -1 - \frac{4}{\log({1-\theta})}$ for $\theta \in (0,1/4]$: https://www.wolframalpha.com/input/?i=plot+%2816%2Ft%29*log2%281%2Ft%29+%2B+1+%2B+4%2Flog%281-t%29+for+t+in+0+to+1%2F4
	%    $(1+16 \theta^{-1} \log_2(\theta^{-1}))^{-2}(1-\theta)^{-8\theta^{-1}\log_2{(\theta^{-1})}} \geq 1$: https://www.wolframalpha.com/input/?i=plot+%281%2B16*%281%2Ft%29*log2%28%281%2Ft%29%29%29%5E%28-2%29+*+%281-t%29%5E%28-8*%281%2Ft%29*log2%28%28%281%2Ft%29%29%29%29+-+1+for+t+in+0+to+1%2F4.
	% Therefore, $(j+1)^{-2}(1-\theta)^{-j/2} \geq 1$ for $j \geq \frac{16}{\theta}\log_{2}(\frac{1}{\theta})$.  Using that $\chi \geq \log_2(\theta^{-1})$, $\mathscr{T} / 4 = \chi \theta^{-1} c / 16 \geq (c/16) \theta^{-1}\log_2(\theta^{-1}) \geq 16 \theta^{-1}\log_2(\theta^{-1})$ since $c \geq 256$.  Therefore, $(1-\theta)^{-j/2} \geq (j+1)^2$ for $j \geq \mathscr{T}/4$.
	%
	%
	Then use the assumption $\|\nabla \hat f_x(0)\| \leq 2\ell\mathscr{M}$ and $j \geq \mathscr{T} / 4$ again to replace the power with $j / 2 \geq \sqrt{\kappa} \chi c / 8 \geq 4 \sqrt{\kappa} \cdot 2\chi$ (with $c \geq 64$) and see that
	\begin{align*}
		\left\|\sum_{m \in S} \left( 1 -  \eta\lambda_m \sum_{k = 0}^{j-1} a_{m, k} \right) g^{(m)} e_m \right\|^2 & \leq 16 \ell^2 \mathscr{M}^2 (1-\theta)^{j/2} \leq 16 \epsilon^2 \kappa c^{-2} \left( 1 - \frac{1}{4\sqrt{\kappa}} \right)^{4 \sqrt{\kappa} \cdot 2\chi}.
	\end{align*}
	Use the fact that $0 < (1-t^{-1})^t < e^{-1} \leq 2^{-1}$ for $t \geq 4$
	% https://www.wolframalpha.com/input/?i=plot+%281-1%2Ft%29%5Et+for+t+from+4+to+1000
	% https://www.wolframalpha.com/input/?i=1%2Fe+-+%281-1%2Ft%29%5Et+at+t+%3D+100000
	together with $\kappa \geq 1$ to bound the right-hand side by $16 \epsilon^2 \kappa c^{-2} 2^{-2\chi}$.
	This itself is bounded by $16 \epsilon^2 \kappa c^{-2} \theta^2 = \epsilon^2 c^{-2}$ using $\chi \geq \log_2(\theta^{-1})$.
	Overall, we have shown that
	\begin{align}
		\left\|\sum_{m \in S} \left( 1 -  \eta\lambda_m \sum_{k = 0}^{j-1} a_{m, k} \right) g^{(m)} e_m \right\| & \leq \epsilon / 18,
		\label{eq:blablatermone}
	\end{align}
	with $c \geq 18$.
	This covers the first term in~\eqref{eq:yetanothername}.
	
	We turn to bounding the second term in~\eqref{eq:yetanothername}.
	For this one, we need~\citep[Lem.~34]{jin2018agdescapes} which states that, for $m \in S$ and $j \geq \mathscr{T} / 4$, for any sequence $\{\epsilon_k\}$, we have
	\begin{align}
		\sum_{k = 0}^{j-1} a_{m, k} \epsilon_k & \leq \frac{\sqrt{c_2}}{\eta \lambda_m} \left( |\epsilon_0| + \sum_{k = 1}^{j-1} |\epsilon_{k} - \epsilon_{k-1}| \right), \textrm{ and } \label{eq:c2} \\
		\sum_{k = 0}^{j-1} (a_{m, k} - a_{m, k-1}) \epsilon_k & \leq \frac{\sqrt{c_3}}{\sqrt{\eta\lambda_m}} \left( |\epsilon_0| + \sum_{k = 1}^{j-1} |\epsilon_{k} - \epsilon_{k-1}| \right), \label{eq:c3} 
	\end{align}
	with some positive constants $c_1, c_2, c_3$ and $c \geq c_1$.
	Thus, to bound the remaining term in~\eqref{eq:yetanothername} we start with:
	\begin{align}
		\left\| \sum_{m \in S} \eta\lambda_m \sum_{k = 0}^{j-1} a_{m, j-1-k}^{} \delta_{k}^{(m)} e_m \right\|^2 & \leq c_2 \sum_{m \in S} \left( |\delta_{j-1}^{(m)}| + \sum_{k = 1}^{j-1} | \delta_k^{(m)} - \delta_{k-1}^{(m)} | \right)^2 \nonumber\\
			& \leq 2c_2 \sum_{m \in S} \left[ |\delta_{j-1}^{(m)}|^2 + \left( \sum_{k = 1}^{j-1} | \delta_k^{(m)} - \delta_{k-1}^{(m)} | \right)^2 \right] \nonumber\\
			& \leq 2c_2 \sum_{m \in S} \left[ |\delta_{j-1}^{(m)}|^2 + (j-1) \sum_{k = 1}^{j-1} | \delta_k^{(m)} - \delta_{k-1}^{(m)} |^2 \right] \nonumber\\
			& \leq 2c_2 \|\delta_{j-1}\|^2 + 2c_2 j \sum_{k = 1}^{j-1} \|\delta_k - \delta_{k-1}\|^2.
			\label{eq:spirou}
	\end{align}
	(We used $(a+b)^2 \leq 2a^2 + 2b^2$ again, and another Cauchy--Schwarz on the remaining sum.)
	In order to proceed, we call upon Lemma~\ref{Lemma14} with $R = \mathscr{L}$, $q' = 0$ and $q = j-1$, which is justified by~\eqref{eq:lemfocpBstarters} (recall that $s_{-1} = 0$).
	This yields the first inequality in:
	\begin{align}
		 \left\| \sum_{m \in S} \eta\lambda_m \sum_{k = 0}^{j-1} a_{m, j-1-k}^{} \delta_{k}^{(m)} e_m \right\|^2 & \leq 50 c_2 \hat\rho^2 \mathscr{L}^4 + 2c_2 j \cdot 144 \hat\rho^2 \mathscr{L}^2 \sum_{k = 0}^{j-1} \|s_k - s_{k-1}\|^2 \nonumber\\
			 & \leq 50 c_2 \hat\rho^2 \mathscr{L}^4 + 144 c_2 \hat\rho^2 \mathscr{L}^4.
			 \label{eq:nameofmyfirstborn}
	\end{align}
	The second inequality above is supported by Lemmas~\ref{lem:TSSEj}, \ref{lem:TSStraveldist} and~\ref{lem:params} as well as $j \leq \mathscr{T} / 2$ and the assumption $E_0 - E_{\mathscr{T}/2} \leq \mathscr{E}$, through:
	\begin{align}
		j \sum_{k = 0}^{j-1} \|s_k - s_{k-1}\|^2 \leq 16 \sqrt{\kappa} \eta j (E_{0} - E_{j}) \leq 8 \sqrt{\kappa} \eta \mathscr{T} \mathscr{E} = \mathscr{L}^2 / 2.
		\label{eq:bombadil}
	\end{align}
	Continuing from~\eqref{eq:nameofmyfirstborn}, we see that the right- (hence also left-) hand side is upper-bounded by
	\begin{align*}
		194 c_2 \cdot \hat\rho^2 \mathscr{L}^4 & = 194 c_2 \cdot 16 \epsilon^2 \chi^{-8} c^{-12} \leq \epsilon^2 / 18^2,
	\end{align*}
	with $c \geq 4 c_2^{1/12}$ and $\chi \geq 1$.
	Combine this result with~\eqref{eq:PcalSnablahatfxsj}, \eqref{eq:blablatermthree}, \eqref{eq:yetanothername} and~\eqref{eq:blablatermone} to conclude that
	\begin{align*}
		\left\| P_\calS \nabla \hat f_x(s_j) \right\| & \leq \frac{\epsilon}{18} + \frac{\epsilon}{18} + \frac{\epsilon}{18} = \frac{\epsilon}{6}
	\end{align*}
	for all $\mathscr{T}/4 \leq j \leq \mathscr{T} / 2$. This proves the first part of the lemma.
	
	For the second part of the result, consider~\eqref{eq:focpBsjsjminone} anew then~\eqref{eq:Akmm} and~\eqref{eq:Akmmter} to see that:
	\begin{align*}
		v_j = s_j - s_{j-1} & = \sum_{m' = 1}^{d} \inner{\begin{pmatrix} s_{j} \\ s_{j-1} \end{pmatrix}}{\begin{pmatrix} e_{m'} \\ -e_{m'	} \end{pmatrix}} e_{m'} \\
			& = - \eta \sum_{m' = 1}^{d} \sum_{k = 0}^{j-1} \sum_{m = 1}^d \left( g^{(m)} - \delta_k^{(m)} \right) \inner{A^{j-1-k} \begin{pmatrix} e_m \\ 0 \end{pmatrix}}{\begin{pmatrix} e_{m'} \\ -e_{m'} \end{pmatrix}} e_{m'} \\
%			& = - \eta \sum_{k = 0}^{j-1} \sum_{m = 1}^d \left( g^{(m)} - \delta_k^{(m)} \right) \inner{A^{j-1-k} \begin{pmatrix} e_m \\ 0 \end{pmatrix}}{\begin{pmatrix} e_{m} \\ -e_{m} \end{pmatrix}} e_{m} \\
			& = - \eta \sum_{k = 0}^{j-1} \sum_{m = 1}^d \left( g^{(m)} - \delta_k^{(m)} \right) \left( (A_m^{j-1-k})_{11} - (A_m^{j-2-k})_{11} \right) e_{m}.	
	\end{align*}
	Using notation as in~\eqref{eq:amtbmt} for $a_{m, t}$, it follows that
	\begin{align*}
		P_\calS v_j & = - \eta \sum_{m \in S}  \sum_{k = 0}^{j-1} \left( g^{(m)} - \delta_k^{(m)} \right) \left( a_{m, j-1-k} - a_{m, j-2-k} \right) e_{m}.
	\end{align*}
	We aim to upper-bound $\inner{P_\calS v_j}{\mathcal{H} P_\calS v_j}$.
	Compute, then use~\eqref{eq:c3} to bound the sum in $k$:
	\begin{align}
		 \inner{P_\calS v_j}{\mathcal{H} P_\calS v_j} & = \eta^2 { \sum_{m \in S} \lambda_m \left(\sum_{k = 0}^{j-1} \left( g^{(m)} - \delta_k^{(m)} \right) \left( a_{m, j-1-k} - a_{m, j-2-k} \right)\right)^2 } \nonumber \\
		 & = \eta^2 { \sum_{m \in S} \lambda_m \left(g^{(m)} \sum_{k = 0}^{j-1} \left( a_{m, k} - a_{m, k-1} \right) - \sum_{k = 0}^{j-1} \delta_k^{(m)} \left( a_{m, j-1-k} - a_{m, j-2-k} \right)\right)^2 } \nonumber \\
		 & \leq 2 \eta^2 \sum_{m \in S} \lambda_m \left(g^{(m)} \sum_{k = 0}^{j-1} \left( a_{m, k} - a_{m, k-1} \right) \right)^2 \nonumber \\
		 & \qquad + 2 \eta^2 \sum_{m \in S} \lambda_m \left(\sum_{k = 0}^{j-1} \delta_k^{(m)} \left( a_{m, j-1-k} - a_{m, j-2-k} \right)\right)^2.
		 \label{eq:PcalSvjHPcalSvj}
	\end{align}
%	We bound the first term above using the assumption $\|\nabla \hat f_x(0)\| \leq 2\ell\mathscr{M}$ then Lemma~\ref{lem:params}:
%	\begin{align*}
%		2c_3^2 \eta \sum_{m \in S} \left| g^{(m)} \right|^2 \leq 8 c_3^2 \eta \ell^2 \mathscr{M}^2 = 2 c_3^2 \ell \mathscr{M}^2
%	\end{align*}
%	\TODO{This approach does not work: neither terms get a bound in the desired form from this. Need to roll back.}
	(We used $(a + b)^2 \leq 2a^2 + 2b^2$ again.)
	
	Focusing on the first term of~\eqref{eq:PcalSvjHPcalSvj}, use~\eqref{eq:amjbmjminus} twice to see that
	\begin{align*}
		\sum_{k = 0}^{j-1} \left( a_{m, k} - a_{m, k-1} \right) & = \frac{1}{\eta\lambda_m}(1 - a_{m, j} + b_{m, j}) - \frac{1}{\eta\lambda_m}(1 - a_{m, j-1} + b_{m, j-1}) - a_{m, -1} \\
			& = \frac{1}{\eta\lambda_m}(a_{m, j-1} - b_{m, j-1} - a_{m, j} + b_{m, j}).
	\end{align*}
	(Indeed, $a_{m, -1} = 0$ as it is the top-left entry of a matrix of the form $\left(\begin{smallmatrix} a & b \\ 1 & 0 \end{smallmatrix}\right)^{-1}$: that is zero regardless of $a$ and $b \neq 0$.)
	% https://www.wolframalpha.com/input/?i=invert+%7B%7Ba%2C+b%7D%2C+%7B1%2C+0%7D%7D
	Hence, the first term in~\eqref{eq:PcalSvjHPcalSvj} is equal to the right-hand side below; the first bound follows from $(a+b+c+d)^2 \leq 4(a^2 + b^2 + c^2 + d^2)$ (Cauchy--Schwarz) and~\eqref{eq:amjbmjbound}, while the second bound follows from~\eqref{eq:jpowerbound} for $j \geq \mathscr{T} / 4$:
	\begin{align*}
		\sum_{m \in S} \frac{2}{\lambda_m} \left|g^{(m)}\right|^2 \left(a_{m, j-1} - b_{m, j-1} - a_{m, j} + b_{m, j}\right)^2 & \leq \sum_{m \in S} \frac{16}{\lambda_m} \left|g^{(m)}\right|^2 \left( (j+1)^2(1-\theta)^{j} + j^2(1-\theta)^{j-1} \right) \\
			& \leq \sum_{m \in S} \frac{16}{\lambda_m} \left|g^{(m)}\right|^2 \left( (1-\theta)^{j/2} + (1-\theta)^{j/2-1} \right) \\
			& \leq \sum_{m \in S} \frac{128}{3\lambda_m} \left|g^{(m)}\right|^2 (1-\theta)^{j/2}.
	\end{align*}
	(The last inequality uses $\theta \in (0, 1/4]$ so that $(1-\theta)^{-1} \leq 4/3$.)
	Moreover, for $m \in S$ we have $\lambda_m > \frac{\theta^2}{\eta (2-\theta)^2} \geq \frac{1}{4\eta} \theta^2 = \frac{1}{4\eta} \frac{1}{16} \frac{\sqrt{\hat\rho \epsilon}}{\ell} = \frac{\sqrt{\hat\rho \epsilon}}{16}$.
	Therefore, in light of the latest considerations and using the assumption $\|\nabla \hat f_x(0)\| \leq 2\ell\mathscr{M}$ and also $j / 2 \geq \sqrt{\kappa} \chi c / 8$ owing to $j \geq \mathscr{T} / 4$, the first term in~\eqref{eq:PcalSvjHPcalSvj} is upper-bounded by:
	\begin{align*}
		\sum_{m \in S} \frac{128}{3} \frac{16}{\sqrt{\hat\rho \epsilon}} \left|g^{(m)}\right|^2 (1-\theta)^{j/2} & \leq 3000 \frac{\ell^2\mathscr{M}^2}{\sqrt{\hat\rho \epsilon}} (1-\theta)^{\sqrt{\kappa} \chi c / 8} \\
		& = 3000 \mathscr{M}^2 \sqrt{\hat \rho \epsilon} \kappa^2 \left(1-\frac{1}{4\sqrt{\kappa}}\right)^{4\sqrt{\kappa} \cdot  4\chi \cdot c / 128} \\
		& \leq 3000 \mathscr{M}^2 \sqrt{\hat \rho \epsilon} \kappa^2 \cdot 2^{-4\chi} 2^{-c / 128}
%		\leq 12 \mathscr{M}^2 \sqrt{\hat \rho \epsilon} 2^{-c / 128}
		\leq \frac{1}{4} \mathscr{M}^2 \sqrt{\hat \rho \epsilon},
	\end{align*}
	where the second-to-last inequality uses again that $0 < (1-t^{-1})^t < 2^{-1}$ for $t \geq 4$, as well as $4\chi \cdot c / 128 \geq 4\chi + c / 128$ with $c \geq 128$; and the last inequality uses $\chi \geq \log_2(\theta^{-1}) = \log_2(4\sqrt{\kappa})$ to see that $\kappa^2 2^{-4\chi} \leq 4^{-4}$, and also $3000 \cdot 4^{-4} \cdot 2^{-c/128} \leq 1/4$ with $c \geq 720$.
	(With care, one could improve the constant, here and in many other places.)
	% 3000*4^(-4)*2^(-720/128)
	
	Now focusing on the second term of~\eqref{eq:PcalSvjHPcalSvj}, we start with~\eqref{eq:c3} to see that
	\begin{align*}
		2 \eta^2 \sum_{m \in S} \lambda_m \left(\sum_{k = 0}^{j-1} \delta_k^{(m)} \left( a_{m, j-1-k} - a_{m, j-2-k} \right)\right)^2 & \leq 2 c_3 \eta \sum_{m \in S} \left( |\delta_{j-1}^{(m)}| + \sum_{k = 1}^{j-1} | \delta_k^{(m)} - \delta_{k-1}^{(m)} | \right)^2 \\
			& \leq 4c_3\eta\|\delta_{j-1}\|^2 + 4c_3\eta j \sum_{k = 1}^{j-1} \| \delta_k - \delta_{k-1} \|^2 \\
			& \leq 388 c_3 \eta \cdot \hat\rho^2 \mathscr{L}^4.
	\end{align*}
	The last inequality follows through the same reasoning that was applied to go from~\eqref{eq:spirou} to~\eqref{eq:nameofmyfirstborn}.
	Through simple parameter manipulation we find
	\begin{align*}
		388 c_3 \eta \cdot \hat\rho^2 \mathscr{L}^4 & = \frac{97c_3}{\ell} \cdot 16\epsilon^2 \chi^{-8}c^{-12} \cdot \frac{\ell^2}{\epsilon^2 \kappa}c^{2} \cdot \mathscr{M}^2
		                                              = 97c_3 \cdot 16 \chi^{-8}c^{-10} \cdot \sqrt{\hat \rho \epsilon} \mathscr{M}^2 \leq \frac{1}{4} \mathscr{M}^2 \sqrt{\hat \rho \epsilon},
	\end{align*}
	with $c \geq 3 c_3^{1/10}$ and $\chi \geq 1$.
	
	To conclude, we combine the two main results about~\eqref{eq:PcalSvjHPcalSvj} to confirm that $\inner{P_\calS v_j}{\mathcal{H} P_\calS v_j} \leq \frac{1}{4} \mathscr{M}^2 \sqrt{\hat \rho \epsilon} + \frac{1}{4} \mathscr{M}^2 \sqrt{\hat \rho \epsilon} = \frac{1}{2} \mathscr{M}^2 \sqrt{\hat \rho \epsilon} \leq \mathscr{M}^2 \sqrt{\hat \rho \epsilon}$ for all $\mathscr{T}/4 \leq j \leq \mathscr{T} / 2$. This proves the second part of the lemma.
\end{proof}

\section{Proof from Section~\ref{sec:socp} about $\PARGD$} \label{app:SOCP}

\begin{proof}[Proof of Lemma~\ref{lem:socp}]
	For contradiction, assume $E_0 - E_{\mathscr{T}}$ and $E_0' - E_{\mathscr{T}}'$ are both strictly less than $2\mathscr{E}$.
	Then, by Lemmas~\ref{lem:TSSEj}, \ref{lem:TSStraveldist} and~\ref{lem:params} and the assumption $\|s_0\|, \|s_0'\| \leq r$, we have
	\begin{align}
		\forall j \leq \mathscr{T}, && \|s_j\|  & \leq r + \|s_j - s_0\| \leq \mathscr{L}/64 + \sqrt{32 \sqrt{\kappa} \eta \mathscr{T} \mathscr{E}} = (1/64 + \sqrt{2}) \mathscr{L} \leq 2\mathscr{L}, \nonumber\\
									&& \|s_j'\| & \leq 2\mathscr{L}.
									\label{eq:socpboundssj}
	\end{align}
	The aim is to show that this cannot hold for $j = \mathscr{T}$.
	
	Define $w_j = s_j - s_j'$ for all $j$.
	Observe $w_{-1} = s_{-1}^{} - s_{-1}' = (s_0 - v_0) - (s_0' - v_0') = s_0^{} - s_0' = w_0$ since $v_0 = v_0' = 0$.
	Then, Lemma~\ref{ForLemma18} provides that
	\begin{align}
		\begin{pmatrix} w_j \\ w_{j-1} \end{pmatrix} & = A^j \begin{pmatrix} w_0 \\ w_{0} \end{pmatrix} - \eta \sum_{k = 0}^{j-1} A^{j-1-k} \begin{pmatrix} \delta_k'' \\ 0 \end{pmatrix},
		\label{eq:wjwjminone}
	\end{align}
	where $A$ is as defined and discussed in Appendix~\ref{sec:supportlemmas}, and
	\begin{align*}
		\delta_k'' & \triangleq \nabla \hat{f}_x(u_k)-\nabla \hat{f}_x(u_k') - \mathcal{H} (u_k - u_k') \\
			& = \left(\int_{0}^{1} \left( \nabla^2 \hat f_x(\phi u_k + (1-\phi)u_k') - \nabla^2 \hat f_x(0) \right) \dphi\right)\![u_k - u_k'].
	\end{align*}
	Recall that $u_k = (2-\theta) s_k - (1-\theta) s_{k-1}$.
	In particular, using~\eqref{eq:socpboundssj} and Lemma~\ref{lem:params} we have: % (as reasoned in the proof of Lemma~\ref{Lemma13}),
	\begin{align*}
		\|u_k\| & \leq |2-\theta| \|s_k\| + |1-\theta| \|s_{k-1}\| \leq 6\mathscr{L} \leq b.
	\end{align*}
	The same holds for $\|u_k'\|$, and $\|\phi u_k + (1-\phi)u_k'\| \leq \max\!\left( \|u_k\|, \|u_k'\| \right) \leq 6\mathscr{L} \leq b$ for $\phi \in [0, 1]$.
	It follows that the Lipschitz-type properties of $\nabla^2 \hat f_x$ apply along rays from the origin of $\T_x\calM$ to any point of the form $\phi u_k + (1-\phi)u_k'$ for $\phi \in [0, 1]$.
	Therefore,
	\begin{align}
		\|\delta_k''\| & \leq 6 \hat\rho \mathscr{L} \| u_k - u_k' \| = 6 \hat\rho \mathscr{L} \| (2-\theta) w_k - (1-\theta) w_{k-1} \| \leq 12 \hat\rho \mathscr{L} \left( \|w_k\| + \|w_{k-1}\| \right).
		\label{eq:deltakppbound}
	\end{align}
	This will come in handy momentarily.
	
	As we did in previous proofs, let $e_1, \ldots, e_d$ form an orthonormal basis of eigenvectors for $\mathcal{H}$ with eigenvalues $\lambda_1 \leq \cdots \leq \lambda_d$.
	Expand the vectors $w_j$ and $\delta_k''$ in this basis as:
	\begin{align*}
		w_j & = \sum_{m = 1}^d w_j^{(m)} e_m , & \delta_k'' & = \sum_{m = 1}^{d} (\delta_k'')^{(m)} e_m.
	\end{align*}
	Going back to~\eqref{eq:wjwjminone}, we can write
	\begin{align*}
		w_j & = \sum_{m' = 1}^{d} \inner{\begin{pmatrix} e_{m'} \\ 0 \end{pmatrix}}{\begin{pmatrix} w_j \\ w_{j-1} \end{pmatrix}} e_{m'} \\
			& = \sum_{m' = 1}^{d} \sum_{m = 1}^{d} \left[ \inner{\begin{pmatrix} e_{m'} \\ 0 \end{pmatrix}}{A^j \begin{pmatrix} e_{m} \\ e_{m} \end{pmatrix}} w_0^{(m)}  -  \eta \sum_{k = 0}^{j-1} \inner{\begin{pmatrix} e_{m'} \\ 0 \end{pmatrix}}{A^{j-1-k} \begin{pmatrix} e_m \\ 0 \end{pmatrix}} (\delta_k'')^{(m)} \right]e_{m'}.
	\end{align*}
	Owing to~\eqref{eq:Akmm} and~\eqref{eq:Akmmbis}, only the terms with $m = m'$ survive.
	Also, recalling that $w_0 = r_0 e_1$ by assumption, we have
	\begin{align}
		w_j & = \left( a_{1,j} - b_{1,j} \right) r_0 e_1  -  \eta \sum_{m = 1}^{d} \sum_{k = 0}^{j-1} a_{m, j-1-k} (\delta_k'')^{(m)} e_{m},
		\label{eq:wjexpansion}
	\end{align}
	where $a_{m,j}$, $b_{m,j}$ are defined by~\eqref{eq:amtbmt}.
%	where each $A_m$ is a $2 \times 2$ matrix as defined by~\eqref{eq:Am}.
	
	We aim to show that $w_{\mathscr{T}} = s_\mathscr{T}^{} - s_\mathscr{T}'$ is larger than $4\mathscr{L}$, as this will contradict the claim that both $\|s_\mathscr{T}\|$ and $\|s_\mathscr{T}'\|$ are smaller than $2\mathscr{L}$: in view of~\eqref{eq:socpboundssj}, this is sufficient to prove the lemma.
	To this end, we introduce two new sequences of vectors to split $w_j$ according to~\eqref{eq:wjexpansion}:
	\begin{align*}
		w_j & = y_j - z_j, & y_j & = \left( a_{1,j} - b_{1,j} \right) r_0 e_1, 
			& z_j & = \eta \sum_{m = 1}^{d} \sum_{k = 0}^{j-1} a_{m, j-1-k} (\delta_k'')^{(m)} e_m.
	\end{align*}
	First, we show by induction that $\|z_j\| \leq \frac{1}{2}\|y_j\|$ for all $j$.
	The base case holds since $z_0 = 0$.
	Now assuming the claim holds for $z_0, \ldots, z_j$, we must prove that $\|z_{j+1}\| \leq \frac{1}{2} \|y_{j+1}\|$.
	Owing to the induction hypothesis, we know that
	\begin{align}
		\forall j' \leq j, && \|w_{j'}\| \leq \|y_{j'}\| + \|z_{j'}\| \leq \frac{3}{2} \|y_{j'}\|.
		\label{eq:wjboundearly}
	\end{align}
	By assumption, $\lambda_1$ (the smallest eigenvalue of $\nabla^2 \hat f_x(0)$) is less than $-\sqrt{\hat\rho \epsilon}$.
	In particular, it is nonpositive.
	Hence~\citep[Lem.~37]{jin2018agdescapes} asserts that $\max_{m = 1, \ldots, d} |a_{m, j-k}| = |a_{1, j-k}|$, so that, also using~\eqref{eq:deltakppbound} then~\eqref{eq:wjboundearly}:
	\begin{align*}
		\|z_{j+1}\| & \leq \eta \sum_{k = 0}^{j} \left\| \sum_{m = 1}^{d} a_{m, j-k} (\delta_k'')^{(m)} e_m \right\|
		\leq \eta \sum_{k = 0}^{j}  |a_{1, j-k}| \|\delta_k''\| \\
		& \leq 12 \eta \hat\rho \mathscr{L} \sum_{k = 0}^{j}  |a_{1, j-k}| \left( \|w_k\| + \|w_{k-1}\| \right)
		\leq 18 \eta \hat\rho \mathscr{L} \sum_{k = 0}^{j}  |a_{1, j-k}| \left( \|y_{k}\| + \|y_{k-1}\| \right).
	\end{align*}
	Moreover, \citep[Lem.~38]{jin2018agdescapes} applies and tells us that
	\begin{align}
		\forall j', && \|y_{j'+1}\| \geq \|y_{j'}\| & \geq \frac{\theta r_0}{2} \left( 1 + \frac{1}{2} \min\!\left( \frac{|\eta\lambda_1|}{\theta}, \sqrt{|\eta\lambda_1|} \right) \right)^{j'}.
		\label{eq:lem38}
	\end{align}
	In particular, $\|y_j\|$ is non-decreasing with $j$.
	Thus, continuing from above, we find that
	\begin{align*}
		\|z_{j+1}\| & \leq 36 \eta \hat\rho \mathscr{L} \sum_{k = 0}^{j}  |a_{1, j-k}| \|y_{k}\| = 36 \eta \hat\rho \mathscr{L} r_0 \sum_{k = 0}^{j}  |a_{1, j-k}| |a_{1,k} - b_{1,k}|,
	\end{align*}
	where the last equality follows from the definition of $y_k$.
	Owing to~\citep[Lem.~36]{jin2018agdescapes}, the fact that $\lambda_1$ is nonpositive implies that
	\begin{align*}
		\forall 0 \leq k \leq j, && |a_{1, j-k}| |a_{1,k} - b_{1,k}| \leq \left( \frac{2}{\theta} + (j+1) \right) |a_{1,k+1} - b_{1,k+1}|.
	\end{align*}
	Moreover, $j+1 \leq \mathscr{T}$ (as otherwise we are done with the proof by induction), and $\frac{2}{\theta} \leq 2\mathscr{T}$ with $c \geq 4$.
	Hence,
	\begin{align*}
		\|z_{j+1}\| & \leq 108 \eta \hat\rho \mathscr{L} \mathscr{T} r_0 \sum_{k = 0}^{j}  |a_{1,k+1} - b_{1,k+1}| = 108 \eta \hat\rho \mathscr{L} \mathscr{T} \sum_{k = 0}^{j}  \|y_{k+1}\|.
	\end{align*}
	Recall that $\|y_k\|$ is non-decreasing with $k$ to see that, using $j+1 \leq \mathscr{T}$ once more:
	\begin{align*}
		\|z_{j+1}\| & \leq 108 \eta \hat\rho \mathscr{L} \mathscr{T}^2 \|y_{j+1}\| \leq \frac{1}{2} \|y_{j+1}\|.
	\end{align*}
	(The last inequality holds with $c \geq 108$ because $108 \eta \hat\rho \mathscr{L} \mathscr{T}^2 = 54 c^{-1}$.)
	This concludes the induction, from which we learn that $\|w_j\| \geq \|y_j\| - \|z_j\| \geq \frac{1}{2} \|y_j\|$ for all $j \leq \mathscr{T}$.
	In particular, it holds owing to~\eqref{eq:lem38} that
	\begin{align*}
		\|w_\mathscr{T}\| & \geq \frac{1}{2} \|y_{\mathscr{T}}\| \geq \frac{\theta r_0}{4} \left( 1 + \frac{1}{2} \min\!\left( \frac{|\eta\lambda_1|}{\theta}, \sqrt{|\eta\lambda_1|} \right) \right)^{\mathscr{T}}.
	\end{align*}
	As per our assumptions, $\lambda_1 \leq -\sqrt{\hat\rho \epsilon}$. Therefore, using the definitions of $\theta$, $\eta$ and $\kappa$,
	\begin{align*}
		\min\!\left( \frac{|\eta\lambda_1|}{\theta}, \sqrt{|\eta\lambda_1|} \right) & \geq \min\!\left( \frac{\sqrt{\hat\rho\epsilon} \sqrt{\kappa}}{\ell}, \sqrt{\frac{\sqrt{\hat\rho \epsilon}}{4\ell}} \right) = \min\!\left( \frac{1}{\sqrt{\kappa}}, \frac{1}{2} \frac{1}{\sqrt{\kappa}} \right) = \frac{1}{2} \frac{1}{\sqrt{\kappa}}.
	\end{align*}
	Moreover, $\mathscr{T} = \sqrt{\kappa} \chi c = 4 \sqrt{\kappa} \chi c/4$, so that, using $(1+1/t)^t \geq 2$ for $t \geq 4$ and $\kappa \geq 1$, $\chi c \geq 4$:
	\begin{align*}
		\|w_\mathscr{T}\| & \geq \frac{\theta r_0}{4} \left( 1 + \frac{1}{4\sqrt{\kappa}} \right)^{4 \sqrt{\kappa} \cdot \chi c/4} \geq \frac{\theta r_0}{4} 2^{\chi c/4} \geq \frac{\theta}{4} \frac{\delta \mathscr{E}}{2\Delta_f} \frac{r}{\sqrt{d}} 2^{\chi (c/4-1)} 2^{\chi}.
	\end{align*}
	At this point, we finally use the assumption $\chi \geq \log_2\!\left( \frac{d^{1/2} \ell^{3/2} \Delta_f}{(\hat \rho \epsilon)^{1/4} \epsilon^2 \delta} \right)$ on the $2^\chi$ factor:
	\begin{align*}
		\|w_\mathscr{T}\| & \geq \frac{\theta}{4} \frac{\delta \mathscr{E}}{2\Delta_f} \frac{r}{\sqrt{d}} 2^{\chi (c/4-1)} \frac{d^{1/2} \ell^{3/2} \Delta_f}{(\hat \rho \epsilon)^{1/4} \epsilon^2 \delta}
		                     = \frac{1}{1024} \chi^{-8} c^{-12} 2^{\chi (c/4-1)} \cdot 4\mathscr{L} > 4\mathscr{L}.
	\end{align*}
	(The last inequality holds with $c \geq 500$ and $\chi \geq 1$:
	% https://www.wolframalpha.com/input/?i=%282%5E124+%2F+500%5E12%29+%2F+1024
	% https://www.wolframalpha.com/input/?i=differentiate+2%5E%28x*%28c%2F4+-+1%29%29%2Fx%5E8%2Fc%5E12+with+respect+to+x  and reason for the sign with x >= 1, c >= 500
	% https://www.wolframalpha.com/input/?i=differentiate+2%5E%28x*%28c%2F4+-+1%29%29%2Fx%5E8%2Fc%5E12+with+respect+to+c
	this fact is straightforward to show by taking derivatives of $\frac{2^{\chi(c/4-1)}}{\chi^{8} c^{12}}$ with respect to $\chi$ and $c$, and showing those derivatives are positive.)
	This concludes the proof by contradiction, from which we deduce that at least one of $E_0 - E_{\mathscr{T}}$ or $E_0' - E_{\mathscr{T}}'$ must be larger than or equal to $2\mathscr{E}$.
\end{proof}

\add{
\section{What if $L, \rho$ are unknown?} \label{backtrackingTAGDsection} \label{backtrackingTAGD}
Consider the scenario where we do not know the Lipschitz constants for the gradient and Hessian of the objective function.  
We assume we have knowledge of $\epsilon, \Delta_f, K$ and $F$.

In this section, $L, \hat{\rho}$ denote the true Lipschitz constants (as usual), and $L', \hat{\rho}'$ are guesses for $L, \hat{\rho}$.
Also, $L_{ini}, \hat{\rho}_{ini}$ denote our initial guesses for $L, \hat{\rho}$.
Given guesses $L', \hat{\rho}'$, let $\ARGD(x_0, L', \hat{\rho}', \epsilon)$ denote the output of running $\ARGD$ for at most $T_1$ steps (i.e., $t \leq T_1$) with parameters (including $T_1$) defined by equations~\eqref{ellrhohatb},~\eqref{definingparams1},~\eqref{definingparams2},~\eqref{chidefinition},~\eqref{T1definition}
except with $L'$ replacing $L$ and $\hat{\rho}'$ replacing $\hat{\rho}$.

Consider the algorithm $\mathtt{backtrackingTAGD}$.
We will show that \\ $\mathtt{backtrackingTAGD}(x_0, L_{ini}, \hat{\rho}_{ini}, \beta)$ finds an $\epsilon$-approximate first-order critical point in 
$$O\Bigg(\hat{\rho}^{1/4} L^{1/2} \epsilon^{-7/4} \bigg(\max\{1, \log(\hat{\rho} / \hat{\rho}_{ini})\} + \max\{1, \log(L / L_{ini})\}\bigg)^2\Bigg)$$
queries.
For comparison, recall that Riemannian gradient descent with backtracking line search finds an $\epsilon$-critical point in $O(L \epsilon^{-2} \max\{1, \log(L / L_{ini})\})$~\cite[Sec.~4.5]{boumal2020intromanifolds}.

Indeed, let $m', n'$ be the minimum nonnegative integers satisfying 
\begin{align*}
L_{ini} \beta^{n'} \geq L, \quad \quad \hat{\rho}_{ini} \beta^{2m'} \geq \hat{\rho}.
\end{align*}
Observe that for any $m, n$ with $m+n = p \leq m' + n'$, $\ARGD(x_0, L', \hat{\rho}', \epsilon)$ (with $L', \hat{\rho}'$ defined as in the algorithm $\mathtt{backtrackingTAGD}$) terminates in
\begin{align*}
O(T_1) &= O(\hat{\rho}'^{1/4} L'^{1/2} \epsilon^{-7/4}) = O(\hat{\rho}_{ini}^{1/4} \beta^{m/2} L_{ini}^{1/2} \beta^{n/2} \epsilon^{-7/4}) = O(\hat{\rho}_{ini}^{1/4} L_{ini}^{1/2} \epsilon^{-7/4} \beta^{p/2}) 
\\ &\leq O(\hat{\rho}_{ini}^{1/4} \beta^{m'/2} L_{ini}^{1/2} \beta^{n'/2} \epsilon^{-7/4}) = O({\hat{\rho}}^{1/4} L^{1/2} \epsilon^{-7/4})
\end{align*}
queries regardless of if an $\epsilon$-approximate first-order critical point is found.

On the other hand, owing to our main theorems and how $\mathtt{backtrackingTAGD}$ works, an $\epsilon$-approximate first-order critical point must be found if $p \geq m' + n'$.
So there are at most $O((m' + n')^2)$ runs of $\ARGD(x_0, L', \hat{\rho}', \epsilon)$.
Therefore, $\mathtt{backtrackingTAGD}$ requires at most $O(\hat{\rho}^{1/4} L^{1/2} \epsilon^{-7/4} (m' + n')^2)$ queries to find an $\epsilon$-approximate first order critical point.
% Note that this argument will not apply to finding approximate second-order critical points because it is not clear how to efficiently test whether a point is approximately second-order critical using only first order information when $L$ and $\hat{\rho}$ are unknown.

\begin{algorithm}[t] \label{backtrackingTAGDalgo}
	\caption{$\mathtt{backtrackingTAGD}(x_0, L_{ini}, \rho_{ini}, \beta)$ with $x_0 \in \calM$ and $\beta > 1$} %, r, T
	\label{algo:backtrackingARGD}
	\begin{algorithmic}[1]	
		\For {$p = 0, 1, \ldots$}
			\For {$m = 0, 1, \ldots, p$}
				\State $n = p - m$
				\State $L' = L_{ini} \beta^{n}$
				\State $\hat{\rho}' = \hat{\rho}_{ini} \beta^{2 m}$
				\State $x_{p, m} = \ARGD(x_0, L', \rho', \epsilon)$ \Comment{See Appendix~\ref{backtrackingTAGDsection} for definition of $\ARGD(x_0, L', \rho', \epsilon)$.}
				\If {$\norm{\grad f(x_{p, m})} \leq \epsilon$}
					\State \textbf{return} $x_{p, m}$
				\EndIf
			\EndFor
		\EndFor
	\end{algorithmic}
\end{algorithm}
}

\add{
\section{Curvature for positive definite matrices} \label{appPDmatrices}
Let $d \geq 2$.
Let $\Sym(d)$ be the set of real $d \times d$ symmetric matrices.
The set of $d\times d$ positive definite matrices
	$$\calP_d = \{P \in \Sym(d) : P \succ 0\}$$ 
	endowed with the so-called affine invariant metric 
$$\inner{X}{Y}_P = \trace(P^{-1} X P^{-1} Y) \quad \text{for } P \in \calP_d \text{, and } X, Y \in \T_P \calP_d \cong \Sym(d)$$
is a Hadamard manifold, meaning all of its sectional curvatures are less than or equal to zero---see~\citep[Thm. 10.39]{bridsonmetric} or~\citep[Prop. 3.1]{dolcetti2018differential}.	

The Riemannian manifold $\calP_d$ is also a symmetric space~\citep[Prop. 3.1]{dolcetti2018differential}.  All symmetric spaces are locally symmetric~\citep[Exercise~6-19, Exercise~7-3 and p78]{lee2018riemannian}, so $\nabla R = 0$ for $\calP_d$: we can pick $F=0$ in assumption~\aref{assu:Mandfintrinsic}.
The following proposition shows that we can also pick $K = \frac{1}{2}$ and this is the best constant.

\begin{proposition}
The sectional curvatures of $\calP_d$ are at least $-1/2$, and this bound is tight.
\end{proposition}
\begin{proof}
By Proposition 2.3 of~\citep{dolcetti2018differential}, or alternatively Theorem 2.1 of~\citep{skovgaard1984riemgeogaussians}, the curvature tensor of $\calP_d$ at $P \in \calP_d$ is
$$\inner{R(W, X) Y}{Z}= -\frac{1}{4} \trace([P^{-1} W, P^{-1} X] [P^{-1} Y, P^{-1} Z]), \quad \text{for } W, X, Y, Z \in \Sym(d)$$
where $[X, Y] = X Y - Y X$ is the matrix commutator of $X, Y$.
Since $\calP_d$ is homogeneous, it is sufficient to consider $P = I$.

Let $X, Y \in \Sym(d)$ be two orthonormal tangent vectors at $P=I$, i.e., 
\begin{align*}
\trace(X Y) = 0, \quad \quad \trace(X^2) = \trace(Y^2) = 1.
\end{align*}
The sectional curvature corresponding to the 2-dimensional subspace spanned by $X, Y$ is therefore
$$K(X, Y) = \frac{\inner{R(X, Y)Y}{X}}{\trace(X^2) \trace(Y^2) - \trace(X Y)^2} = -\frac{1}{4} \trace([X, Y] [Y, X]) = -\frac{1}{4} \trace([X, Y]^\top [X, Y]).$$
So let us consider the optimization problem
$$\max_{X, Y} f(X, Y) \quad \text{with} \quad f(X,Y) = -4 K(X, Y) = \trace([X, Y]^\top [X, Y]) = 2 \trace(X^2 Y^2) - 2 \trace((X Y)^2)$$
subject to $X, Y \in \Sym(d), \trace(X Y) = 0, \trace(X^2) = \trace(Y^2) = 1$.
%\TODO{You should explain that this amounts to selecting two orthonormal tangent vectors at P = I, and that this selects a 2-d subspace of the tangent space, and f(X, Y) evaluates the [something related to the sectional curvature of the manifold at that point along that subspace].}
We will show the max value is $2$ which implies the sectional curvatures of $\calP_d$ are bounded below by $-2 / 4 = -1/2$.

\textbf{Step 1}:
Note that the constraint set 
$$\{(X, Y) \in \Sym(d) \times \Sym(d) : \trace(X Y) = 0, \trace(X^2) = \trace(Y^2) = 1\}$$
is compact, so a maximizer $(X^*, Y^*)$ of $f$ exists.
Fixing $X = X^*$, consider the problem 
$$\max_{Y \in \calN} f_{X^*}(Y) \quad \text{with} \quad f_{X^*}(Y) = f(X^*, Y), \quad \calN = \{Y \in \Sym(d) : \trace(X^* Y) = 0, \trace(Y^2) = 1\}.$$
Of course $Y^*$ is a maximizer of this problem, and the constraint set $\calN$ is the intersection of a sphere with two linear subspaces.
Therefore, $\calN$ is a smooth manifold.
Treating $\calN$ as an embedded submanifold of $\reals^{d \times d}$ with the Frobenius inner product, let $\grad f_{X^*}$ denote the Riemannian gradient of $f_{X^*}$ on $\calN$.
As $\calN$ is a compact manifold (without boundary), $Y^*$ must be a critical point: $\grad f_{X^*}(Y^*) = 0$.

Consider the smooth extension 
$$\bar{f}_{X^*} : \reals^{d \times d} \rightarrow \reals, \quad \quad \bar{f}_{X^*}(Y) = \trace([X^*, Y]^\top [X^*, Y]).$$
The Euclidean gradient of this extension is $\nabla \bar{f}_{X^*}(Y) = 2 [X^*, [X^*, Y]]$, and so 
$$\grad f_{X^*}(Y) = \Proj_Y(\nabla \bar{f}_{X^*}(Y)) = \Proj_Y(2 [X^*, [X^*, Y]])$$
where $\Proj_Y \colon \reals^{d \times d} \rightarrow \T_Y \calN$ denotes orthogonal projection onto the tangent space $\T_Y \calN$ ~\citep[Sec.~3.8]{boumal2020intromanifolds}.

Projection onto a tangent space of the sphere $\{Y \in \reals^{d \times d}: \trace(Y^\top Y) = 1\}$ is given by
$$\Proj_Y^{\mathbb{S}}(U) = U - \trace(Y^\top U) Y, \quad \text{ for } U \in \reals^{d \times d}.$$
For $Y \in \calN$, note that 
$$\Proj_Y^{\mathbb{S}}(2 [X^*, [X^*, Y]]) = 2([X^*, [X^*, Y]] - \trace(Y^\top [X^*, [X^*, Y]]) Y) = 2( [X^*, [X^*, Y]] - f(X^*, Y) Y)$$
is contained in both linear subspaces $\Sym(d)$ and $\{Y \in \reals^{d \times d} : \trace(X^* Y) = 0\}$.
Therefore,
$$\grad f_{X^*}(Y) = \Proj_Y(\nabla \bar{f}_{X^*}(Y)) = \Proj_Y^{\mathbb{S}}(2 [X^*, [X^*, Y]]) = 2( [X^*, [X^*, Y]] - f(X^*, Y) Y).$$
So, $\grad f_{X^*}(Y^*) = 0$ implies
$$[X^*, [X^*, Y^*]] = f(X^*, Y^*) Y^*.$$

\textbf{Step 2}:
If $X$ has eigenvalue decomposition $X = Q D Q^{-1}$ with $Q$ orthogonal and $D$ diagonal,
\begin{align*}
f(X, Y) &= 2 \trace(X^2 Y^2) - 2 \trace(X Y X Y) = 2 \trace(Q D^2 Q^{-1} Y^2) - 2 \trace(Q D Q^{-1} Y Q D Q^{-1} Y)
\\ &= 2 \trace(D^2 Q^{-1} Y^2 Q) - 2 \trace(D Q^{-1} Y Q D Q^{-1} Y Q) = 2 \trace(D^2 Z^2) - 2 \trace(D Z D Z) \\
&= f(D, Z)
\end{align*}
where $Z = Q^{-1} Y Q$, using the cyclic property of the trace.
Therefore, without loss of generality, we can assume that $X^*$ is diagonal.

Note that 
$$[X^*, [X^*, Y^*]] = X^* [X^*, Y^*] - [X^*, Y^*] X^* = X^* X^* Y^* - 2 X^* Y^* X^* + Y^* X^* X^*$$
has $i$-th diagonal entry
\begin{align*}
[X^*, [X^*, Y^*]]_{ii} = \sum_{j, k} X^*_{ij} X^*_{jk} Y^*_{ki} - 2 X^*_{ij} Y^*_{jk} X^*_{ki} + Y^*_{ij} X^*_{jk} X^*_{ki} \\
=  X^*_{ii} X^*_{ii} Y^*_{ii} - 2 X^*_{ii} Y^*_{ii} X^*_{ii} + Y^*_{ii} X^*_{ii} X^*_{ii} = 0.
\end{align*}
So $[X^*, [X^*, Y^*]]$ is a matrix whose diagonal entries are all zero, hence the diagonal entries of $Y^*$ are also all zero.

In the following two series of equations and inequalities~\eqref{eqn1PDmatrices} and~\eqref{eqn2PDmatrices}, we denote $X^*$ and $Y^*$ simply by $X$ and $Y$ for notational convenience (in particular, $X$ is diagonal and $Y$ has all diagonal entries equal to zero).
With $i, j$ ranging from $1$ to $d$,
\begin{equation} \label{eqn1PDmatrices}
\begin{split}
\trace(X^2 Y^2) &= \sum_{i, j} X_{ii}^2 Y_{ij}^2 = \sum_{i \neq j} X_{ii}^2 Y_{ij}^2 = \sum_{i} \Big(X_{ii}^2 \sum_{j : j \neq i} Y_{ij}^2\Big) \leq \Big(\max_{i} \sum_{j : j \neq i} Y_{ij}^2\Big) \sum_{i} X_{ii}^2 \\
&= \Big(\max_{i} \sum_{j : j \neq i} Y_{ij}^2\Big) \trace(X^2) = \max_{i} \sum_{j : j \neq i} Y_{ij}^2 \leq \trace(Y^2) / 2 = 1/2.
\end{split}
\end{equation}
Also,
\begin{equation} \label{eqn2PDmatrices}
\begin{split}
\trace((X Y)^2) &= \sum_{i, j} X_{ii} X_{jj} Y_{ij}^2 = \sum_{i \neq j} X_{ii} X_{jj} Y_{ij}^2 \geq - \Big(\max_{i \neq j} \left| X_{ii} X_{jj} \right| \Big) \sum_{i \neq j} Y_{ij}^2 \\
&\geq - \Big(\max_{i \neq j} \left| X_{ii} X_{jj} \right| \Big) \trace(Y^2) = -\max_{i \neq j} \left| X_{ii} X_{jj} \right| \geq -1/2.
\end{split}
\end{equation}
So $f(X^*, Y^*) \leq 2(1/2) - 2(-1/2) = 2$.

\textbf{Step 3}:
Finally, it is easy to construct an example showing that this is in fact tight.  Indeed, consider $X = \diag(0, ..., 0, \frac{1}{\sqrt{2}}, - \frac{1}{\sqrt{2}})$ and $Y$ such that $Y_{d-1, d} = Y_{d, d-1} = \frac{1}{\sqrt{2}}$ and all other entries of $Y$ equal to zero.
\end{proof}
}

\bibliographystyle{abbrvnat}
\bibliography{../bibtex/boumal}

\end{document}